\documentclass[11pt]{article}
\usepackage{amsfonts}
\usepackage{amssymb}
\usepackage{amsthm}
\usepackage{amsmath}
\usepackage{graphicx}
\usepackage{empheq}
\usepackage{indentfirst}
\usepackage{cite} 
\usepackage{mathrsfs}
\usepackage{cases}
\usepackage{graphics}
\usepackage{xcolor}  
\usepackage{bm}
\usepackage{makeidx}
\usepackage[T1]{fontenc}
\usepackage{times}  
\usepackage{appendix}
\newtheorem{theorem}{\textbf{Theorem}}[section]
\newtheorem{lemma}{\textbf{Lemma}}[section]
\newtheorem{proposition}{\textbf{Proposition}}[section]
\newtheorem{corollary}{\textbf{Corollary}}[section]
\newtheorem{remark}{\textbf{Remark}}[section]
\newtheorem{definition}{\textbf{Definition}}[section]
\allowdisplaybreaks[4] 
\def\be{\begin{equation}}
\def\ee{\end{equation}}
\def\bea{\begin{eqnarray}}
\def\eea{\end{eqnarray}}
\def\bt{\begin{theorem}}
\def\et{\end{theorem}}
\def\bl{\begin{lemma}}
\def\el{\end{lemma}}
\def\br{\begin{remark}}
\def\er{\end{remark}}
\def\bp{\begin{proposition}}
\def\ep{\end{proposition}}
\def\bc{\begin{corollary}}
\def\ec{\end{corollary}}
\def\bd{\begin{definition}}
\def\ed{\end{definition}}
\begin{document}

\title{Regularity Propagation of Global Weak Solutions to a  Navier--Stokes--Cahn--Hilliard System for Incompressible Two-phase Flows with Chemotaxis and Active Transport}

\author{
	Jingning He
	\footnote{School of Mathematics, Hangzhou Normal University, Hangzhou 311121, Zhejiang Province, P.R. China. Email: \texttt{jingninghe2020@gmail.com}
	},
\ \ Hao Wu\footnote{Corresponding author. School of Mathematical Sciences, Fudan University, Shanghai 200433, P.R. China.
		Email: \texttt{haowufd@fudan.edu.cn}
	}
}
\date{\today}
\maketitle


\begin{abstract}
\noindent
We analyze a diffuse interface model that describes the dynamics of incompressible viscous two-phase flows, incorporating mechanisms such as chemotaxis, active transport, and long-range interactions of Oono's type. The evolution system couples the Navier--Stokes equations for the volume-averaged fluid velocity $\bm{v}$, a convective Cahn--Hilliard equation for the phase-field variable $\varphi$, and an advection-diffusion equation for the density of a chemical substance $\sigma$. For the initial boundary value problem with a physically relevant singular potential in three dimensions, we demonstrate that every global weak solution $(\bm{v}, \varphi, \sigma)$ exhibits a propagation of regularity over time. Specifically, after an arbitrary positive time, the phase-field variable $\varphi$ transitions into a strong solution, whereas the chemical density $\sigma$ only partially regularizes. Subsequently, the velocity field $\bm{v}$ becomes regular after a sufficiently large time, followed by a further regularization of the chemical density $\sigma$, which in turn enhances the spatial regularity of $\varphi$. Furthermore, we show that every global weak solution stabilizes towards a single equilibrium as $t\to +\infty$. Our analysis uncovers the influence of chemotaxis, active transport, and long-range interactions on the propagation of regularity at different stages of time. The proof relies on several key points, including a novel regularity result for a convective Cahn--Hilliard--diffusion system with a velocity field $\bm{v}$ of Leray type, the strict separation property of $\varphi$ for large times, as well as two conditional uniqueness results pertaining to the full system and its subsystem for $(\varphi, \sigma)$ with a given velocity, respectively.
\medskip \\
\noindent
\textit{Keywords}: Navier--Stokes--Cahn--Hilliard system, chemotaxis, active transport, nonlocal interaction, global regularity, long-time behavior.
\medskip \\
\noindent
\textit{MSC 2020}: 35A01, 35A02, 35K35, 35Q92, 76D05.
\end{abstract}

\tableofcontents

\section{Introduction}
\setcounter{equation}{0}
\noindent

The study of multi-phase flows has garnered significant attention in various fields of science and engineering. One of the efficient mathematical approaches is the diffuse interface model, which assumes that an immiscible fluid mixture undergoes a smooth and rapid transition within an interfacial region between its components, known as partial diffusive mixing. This approach results in an Eulerian formulation of the phase-field equation and permits topology changes such as breakup and coalescence. Different types of diffuse interface models have been introduced to accommodate a wide range of physical and biological applications, see, e.g., \cite{A2012,AMW,EG19jde,Gur,GLSS,HH,LS,Mi19} and the references cited therein.

In this work, we are interested in the mathematical analysis of the following Navier--Stokes--Cahn--Hilliard type system
\begin{subequations}\label{nsch}
	\begin{alignat}{3}
	&\partial_t  \bm{ v}+\bm{ v} \cdot \nabla  \bm {v}-\textrm{div} \big(2  \nu(\varphi) D\bm{v} \big)+\nabla p=(\mu+\chi \sigma)\nabla \varphi,\label{f3.c} \\
	&\textrm{div}\, \bm{v}=0,\label{f3.c1}\\
	&\partial_t \varphi+\bm{v} \cdot \nabla \varphi=\Delta \mu-\alpha(\overline{\varphi}-c_0),\label{f1.a} \\
	&\mu=-\varepsilon \Delta \varphi+ \frac{1}{\varepsilon} \varPsi'(\varphi)-\chi \sigma+\beta\mathcal{N}(\varphi-\overline{\varphi}),\label{f4.d}\\
	&\partial_t \sigma+\bm{v} \cdot \nabla \sigma= \Delta (\sigma-\chi\varphi), \label{f2.b}
	\end{alignat}
\end{subequations}
in $\Omega\times(0,+\infty)$. This system is subject to the following boundary and initial conditions
\begin{alignat}{3}
&\bm{v}=\mathbf{0},\quad {\partial}_{\bm{n}}\varphi={\partial}_{\bm{n}}\mu={\partial}_{\bm{n}}\sigma=0,\qquad\qquad &\textrm{on}& \   \partial\Omega\times(0,+\infty),
\label{boundary}
\\
&\bm{v}|_{t=0}=\bm{v}_{0}(x),\quad \varphi|_{t=0}=\varphi_{0}(x), \quad  \sigma|_{t=0}=\sigma_{0}(x), \qquad &\textrm{in}&\ \Omega.
\label{ini0}
\end{alignat}
Here, $\Omega \subset\mathbb{R}^3$ is a bounded domain with smooth boundary $\partial\Omega$, $\bm{n}$ is the unit outward normal vector on $\partial \Omega$, and $\partial_{\bm{n}}$ denotes the outer normal derivative on $\partial \Omega$.

System \eqref{nsch} can be regarded as a simplified version of the general thermodynamically consistent diffuse interface model derived in \cite{LW}. This model describes the dynamics of a mixture of two viscous incompressible fluids, with a (soluble) chemical species subject to the diffusion process, along with significant mechanisms such as chemotaxis and active mass transport (see also \cite{Sitka} and the references cited therein). The parameter $\varepsilon>0$ characterizes the thickness of interfacial layers between two fluid components. Since the asymptotic behavior as $\varepsilon\to 0^+$ (i.e., the sharp-interface limit) will not be investigated in this study, we simply set $\varepsilon =1$ in the subsequent analysis. The state variables of system \eqref{nsch} are given by $(\bm{v},p,\varphi, \mu, \sigma)$. The vectorial function $\bm{v}: \Omega \times[0,+\infty)\to \mathbb{R}^3$ denotes the volume-averaged velocity with $D\bm{v}=\frac{1}{2}(\nabla\bm{ v}+(\nabla\bm{ v})^\mathrm{T})$ being the symmetrized velocity gradient, and the scalar function $p:\Omega \times[0,+\infty)\to \mathbb{R}$ denotes the (modified) pressure. The order parameter (also called the phase-field variable)  $\varphi:\Omega \times[0,+\infty)\to [-1,1]$ denotes the difference in volume fractions of the binary mixture such that the region $\{\varphi=1\}$ represents the  fluid 1 and $\{\varphi=-1\}$ represents the fluid 2 (i.e., the values $\pm 1$ represent the pure phases). The scalar function $\sigma: \Omega \times[0,+\infty)\to \mathbb{R}$ denotes the concentration of an unspecified chemical species (e.g., a nutrient in the context of tumor growth modeling \cite{GLSS,GL17e}). Finally, the scalar function $\mu: \Omega \times(0,+\infty)\to \mathbb{R}$ is referred to as the chemical potential in the terminology of Cahn--Hilliard \cite{CH}. For the sake of simplicity, we assume that the density difference between the two components of the fluid mixture is negligible and set the fluid density to be one. Additionally, we assume that the mobilities for $\varphi$ and $\sigma$ are positive constants (set to be one), while allowing the fluid mixture to have unmatched viscosities. Assuming that $\nu_1$, $\nu_2>0$ are viscosities of the two homogeneous fluids, the mean viscosity is given by the concentration dependent function $\nu=\nu(\varphi)$, for instance, the linear combination:
\be
\nu(\varphi)=\nu_1\frac{1+\varphi}{2}+\nu_2\frac{1-\varphi}{2}.
\label{vis}
\ee
The nonlinear function $\varPsi$ in \eqref{f4.d} denotes the homogeneous free energy density for the mixture, which has a double-well structure with two minima and a local unstable maximum in between. A physically significant example is the Flory--Huggins type (see \cite{CH, Gur}):
\be
\varPsi (r)=\frac{\theta}{2}\big[(1-r)\ln(1-r)+(1+r)\ln(1+r)\big]+\frac{\theta_{0}}{2}(1-r^2),\quad \forall\, r\in[-1,1],
\label{pot}
\ee
with $0<\theta<\theta_{0}$. In the literature, the singular potential \eqref{pot} is often approximated by a fourth-order polynomial (e.g., in the regime of shallow quench)
\be
\varPsi(r)=\frac{1}{4}\big(1-r^2\big)^2,\quad \forall\, r\in\mathbb{R}, \label{regular}
\ee
or by some more general polynomial functions (see \cite{Mi19}). This type of approximation rules out possible singularities of the logarithmic potential \eqref{pot} (and its derivatives) at $\pm 1$ and brings great convenience in the mathematical analysis as well as numerical simulations for the Cahn--Hilliard equation and its variants, see \cite{BGM,DFW,GG2010,JWZ,LS,ZWH} and the references therein.

The evolution of velocity $\bm{v}$ and pressure $p$ is governed by the Navier--Stokes equations \eqref{f3.c}--\eqref{f3.c1} under the influence of a Korteweg-type force proportional to $(\mu+\chi \sigma)\nabla \varphi$. The dynamics of $(\varphi, \mu)$ is governed by a convective Cahn--Hilliard equation \eqref{f1.a}--\eqref{f4.d}, while $\sigma$ satisfies an advection-diffusion equation \eqref{f2.b}.
In \eqref{f1.a} and \eqref{f4.d} we denote by
$\overline{\varphi}=|\Omega|^{-1}\int_\Omega \varphi(x)\,\mathrm{d}x$ the mean of the phase function $\varphi$ over $\Omega$. In addition, $\mathcal{N}$ denotes the inverse of the minus Laplacian operator subject to the homogeneous Neumann boundary condition in some function space with a zero mean constraint (see Section \ref{pm}). As shown in \cite{LW}, the hydrodynamic system \eqref{nsch}--\eqref{ini0} is thermodynamically consistent.  Sufficiently regular solutions satisfy the law of energy dissipation
\begin{align}
& \frac{\mathrm{d}}{\mathrm{d}t} \left( \int_{\Omega} \frac{1}{2}|\bm{v}|^2\, \mathrm{d}x+ \mathcal{F}(\varphi,\sigma)\right)   +\int_{\Omega} \left( 2\nu(\varphi)|D\bm{v}|^2 + |\nabla \mu|^2+|\nabla(\sigma-\chi\varphi)|^2\right) \mathrm{d}x\nonumber\\
&\quad  =-\alpha\int_\Omega (\overline{\varphi}-c_0)\mu\, \mathrm{d}x,
\quad \forall\, t>0,
\label{BEL}
\end{align}
where the free energy is given by (taking $\varepsilon=1$)
\be
\mathcal{F}(\varphi,\sigma)=\int_{\Omega}
\left(\frac{1} {2}|\nabla \varphi|^2+ \varPsi(\varphi)+\frac{1}{2}|\sigma|^2-\chi\sigma\varphi +\frac{\beta}{2}|\nabla\mathcal{N}(\overline{\varphi}-\varphi)|^2\right) \,\mathrm{d}x.
\label{fe1}
\ee
Then we see that the chemical potential $\mu$ is obtained by the first variation of $\mathcal{F}(\cdot,\sigma)$.
Next, (formally) integrating the equation \eqref{f1.a} over $\Omega$ and using the boundary condition \eqref{boundary}, we have the mass relation
\begin{align}
\frac{\mathrm{d}}{\mathrm{d}t}(\overline{\varphi}-c_0) +\alpha(\overline{\varphi}-c_0)=0,
\notag
\end{align}
which implies
\begin{align}
\overline{\varphi}(t)-c_0=(\overline{\varphi_0}-c_0)e^{-\alpha t},\quad \forall\, t\geq 0.
\label{mph2}
\end{align}
A similar argument for \eqref{f2.b} yields
\begin{align}
\overline{\sigma}(t)=\overline{\sigma_0},\quad \forall\, t\geq 0.
\label{mph3}
\end{align}

The coupling structures of the system \eqref{nsch} are characterized by the parameters $\alpha$, $\beta$ and $\chi$. The two terms $\alpha(\overline{\varphi}-c_0)$ and $\beta\mathcal{N}(\varphi-\overline{\varphi})$ represent certain nonlocal interactions (i.e., reactions) between the two fluid components. Throughout the paper, we assume that $\alpha\geq 0$ and $\beta\in \mathbb{R}$.
The role of $\alpha(\overline{\varphi}-c_0)$ has two aspects. In view of \eqref{BEL}, it can influence the energy dissipation of the system. On the other hand, it may yield different mass relations for $\varphi$, that is, if $\alpha=0$ or $\overline{\varphi_0}=c_0$ for $\alpha>0$, then the mass $\overline{\varphi}(t)$ is conserved in time, otherwise $\overline{\varphi}(t)$ converges exponentially fast to $c_0$ provided that $\alpha>0$ (i.e., the so-called off-critical case, cf. \cite{BGM,GGM2017,MT}). Neglecting the coupling with $\bm{v}$ and $\sigma$, for $\alpha=\beta=0$, we recover the classical Cahn--Hilliard equation with constant mobility \cite{CH,CMZ,Mi19}. When taking $\alpha=\beta>0$, we obtain the well-known Cahn--Hilliard--Oono (CHO) equation
\begin{align}
\partial_t \varphi+\alpha(\varphi-c_0) =\Delta(-\Delta \varphi+ \varPsi'(\varphi)),
\label{CHO}
\end{align}
which was proposed to describe the dynamics of
microphase separation of diblock copolymers \cite{OP87}. See also \cite{Glo95} for a different physical origin from a binary mixture with reversible isomerization chemical reaction, where the parameters $\alpha$ and $c_0$ can be determined by the forward and backward reaction rates. If $\overline{\varphi_0}=c_0$, the CHO equation \eqref{CHO} can be viewed as a conserved gradient flow of the Ohta--Kawasaki functional (with $\varPsi$ given by \eqref{regular})
$$
\mathcal{F}_{\mathrm{OK}}(\varphi)= \int_{\Omega}
 \left(\frac{1} {2}|\nabla \varphi|^2+ \varPsi(\varphi) \right) \mathrm{d}x + \alpha \int_{\Omega}\!\int_\Omega
 G(x,y)(\varphi(x)-\overline{\varphi}) (\varphi(y)-\overline{\varphi})\,\mathrm{d}x\mathrm{d}y,
$$
where $G$ denotes the Green function associated to the (minus) Laplacian subject to a homogeneous Neumann boundary condition (see e.g., \cite{Hoff22}). For results on well-posedness and long-time behavior of solutions to the CHO equation, we refer to, e.g., \cite{GGM2017,Mi11}.

The coefficient $\chi$ is related to some specific transport mechanisms, such as chemotaxis (as seen in \eqref{f4.d}) and active transport (as seen in \eqref{f2.b}). In the context of tumor growth modeling, $\chi$ characterizes the chemotactic response of the mixture to the nutrient, as well as the evolution of the nutrient influenced by the mixture. For detailed explanations, see \cite{GLSS,LW}.
The term ``active transport" is used when $\chi>0$ in the biological sense, indicating that some kind of mechanism is required to maintain the transport. This is in contrast to passive transport processes, such as diffusion, which are driven only by the concentration gradient \cite{GL17e}. From a mathematical perspective, the term $-\chi\Delta \varphi$ introduces a cross-diffusion structure in \eqref{f2.b}, which poses additional challenges in the analysis. However, the sign of $\chi$ does not play a role, so we allow $\chi$ to be any real number, that is, $\chi\in \mathbb{R}$.

We observe that the well-known ``Model H'' for incompressible viscous two-phase flows (see \cite{HH,Gur}) can be recovered from the system \eqref{nsch} by neglecting $\sigma$ and setting $\alpha=\beta=\chi=0$. According to \eqref{BEL} and \eqref{mph2}, the resulting Navier--Stokes--Cahn--Hilliard (NSCH) system possesses a dissipative structure (i.e., its total energy decreases over time) and conserves $\overline{\varphi}$ for all $t\geq 0$. Due to these two crucial properties, the NSCH system (with either a regular or a singular potential) has been extensively studied in the literature. In this regard, we refer to \cite{A2009,Boyer,GG2010,GMT,LS,ZWH} and the references therein for the results on well-posedness, long-time behavior of global solutions, and asymptotic analyses. See also \cite{A2012,A2022,Gio2021,Gio2022} for contributions on the more complex case with unmatched densities. When the equation for $\sigma$ is neglected, our system \eqref{nsch} also reduces to a variant of the Navier--Stokes--Cahn--Hilliard--Oono (NSCHO) system, which has been studied in \cite{BGM,MT,H03}.
Compared to the basic energy law \eqref{BEL}, the NSCHO system in \cite{BGM,MT} exhibits a higher order energy production rate without a definite sign, given by $-\alpha\int_\Omega (\varphi-c_0)\widetilde{\mu}\, \mathrm{d}x$ with $\widetilde{\mu}=-\Delta \varphi+\varPsi'(\varphi)$ (recalling that $\overline{\varphi}$ depends solely on time and decays to $c_0$ exponentially fast). This fact leads to significant challenges in studying global well-posedness and long-time behavior, as discussed in \cite{MT}. For the single CHO equation \eqref{CHO} without fluid interaction, the authors of \cite{GGM2017} overcame the difficulty by considering the Ohta--Kawasaki functional $\mathcal{F}_{\mathrm{OK}}$ and rewriting the CHO equation as
 \begin{align}
\partial_t \varphi+\alpha(\overline{\varphi}-c_0) =\Delta\big(-\Delta \varphi+ \varPsi'(\varphi)+\alpha \mathcal{N}(\varphi-\overline{\varphi})\big)=\Delta \frac{\delta \mathcal{F}_{\mathrm{OK}}(\varphi)}{\delta \varphi}.
\notag
\end{align}
For the hydrodynamic system \eqref{nsch}--\eqref{f2.b},
we proposed in \cite{HW2022} a decomposition of Oono's nonlocal interaction term, given by the general form $\beta(\varphi-\overline{\varphi})+\alpha(\overline{\varphi}-c_0)$
with $\beta\in \mathbb{R}$ and $\alpha \geq 0$. This modification incorporates nonlocal effects into the capillary force, see \eqref{f3.c}. It not only aligns with the variational framework of \cite{GLSS,LW}, but also offers a suitable correction to the NSCHO system presented in the previous literature \cite{BGM,MT}.

Thanks to the energy balance \eqref{BEL} and the mass relations \eqref{mph2}, \eqref{mph3}, in our recent work \cite{HW2022}, we were able to show that, in the three-dimensional setting, if the initial velocity is small and both the initial phase-field function and the initial chemical density are small perturbations of a local minimizer of the free energy $\mathcal{F}$, then problem \eqref{nsch}--\eqref{ini0} admits a unique global strong solution. Furthermore, we proved the uniqueness of the asymptotic limit for every global strong solution as $t\to +\infty$ and provided an estimate on the convergence rate. The proofs relied on the dissipative structure of the system and an extended {\L}ojasiewicz--Simon type gradient inequality. However, whether global weak solutions to problem \eqref{nsch}--\eqref{ini0} exist and will regularize like solutions to the NSCH system, as demonstrated in \cite{A2009,A2022}, remains an open question.

The aim of this paper is to establish the regularity and long-time behavior of global weak solutions to problem \eqref{nsch}--\eqref{ini0} in the three-dimensional setting, when a physically relevant singular potential such as \eqref{pot} is chosen.

(1) Our first result states that every global weak solution $(\bm{v},\varphi, \sigma)$ exhibits a propagation of regularity over time, see Theorem \ref{2main}. Specifically, after an arbitrary positive time, the phase-field $\varphi$ transitions into a strong solution, whereas the chemical density $\sigma$ only partially regularizes. Subsequently, the velocity field $\bm{v}$ becomes regular after a sufficiently large time, followed by a further regularization of the chemical density $\sigma$, which in turn enhances the spatial regularity of $\varphi$. The proof of Theorem \ref{2main} is inspired by recent work \cite{A2022} on the global regularity of weak solutions to the NSCH system with unmatched densities. Although we tackle the simpler case of a fluid mixture with constant density, additional challenges arise due to chemotaxis, mass transport, and nonlocal interactions, making it rather difficult to establish the regularity of $(\varphi, \sigma)$. Specifically, the chemical density $\sigma$ exhibits a novel phenomenon, undergoing regularization at different stages of time. For discussions on the strategy of proofs, we refer to Section \ref{sec:sPF}.

(2) Our second result pertains to the long-time behavior. Under the additional assumption that $\varPsi$ is real analytic in $(-1,1)$ (which is true for \eqref{pot}), we demonstrate that every global weak solution $(\bm{v},\varphi, \sigma)$ converges to a single equilibrium $(\mathbf{0}, \varphi_\infty,\sigma_\infty)$ as $t\to +\infty$ and provide an estimate of the convergence rate, as stated in Theorem \ref{3main1}. This conclusion is reached by applying Theorem \ref{2main} and employing the same argumentation as presented in \cite{HW2022}.

Before concluding the Introduction, we review some related results in the literature. With a regular potential, including \eqref{regular}, and some more general reaction terms, problem \eqref{nsch}--\eqref{ini0} was first studied in \cite{LW}. In this work, the authors proved the existence of global weak solutions in both two and three dimensions, and the existence of a unique global strong solution in two dimensions. However, due to the coupling between $\varphi$ and $\sigma$, as well as the loss of maximum principle for the Cahn--Hilliard equation with regular potentials, they imposed certain technical assumptions on the coefficients to achieve their goals. Later, in \cite{H}, this restriction was removed by considering a singular potential like \eqref{pot}, which is more physically relevant. Using a semi-Galerkin approach, the author proved the existence of global weak solutions in both two and three dimensions, as well as the uniqueness of global weak solutions in two dimensions. For the same system, the authors of \cite{H1} further established the global strong well-posedness for arbitrarily large regular initial data in two dimensions. The arguments in \cite{H,H1} rely on the key property that the singular potential $\varPsi$ ensures that the phase function $\varphi$ remains within the physical interval $[-1,1]$ throughout the time evolution. Moreover, in the two-dimensional case, $\varphi$ is strictly separated from the pure states $\pm 1$ for $t>0$. This separation property plays a crucial role in studying the regularity and long-time behavior of solutions to the Cahn--Hilliard equation with a singular potential, see \cite{A2007,CMZ,GGM2017,Mi19,MZ04}. We also refer to \cite{A2009,Boyer,CG,Gio2021,GGW,GMT} for generalized systems with fluid coupling. Finally, we note that Oono's linear term represents possibly the simplest case of a (nonlocal) reaction. Our results provide answers to some open questions raised in \cite{MT} concerning the global regularity and asymptotic stabilization of global weak solutions to the NSCHO system with singular potential. On the other hand, more general reaction terms have been integrated into the Cahn--Hilliard equation for biological applications \cite{Fa15,Lam22}, as well as in extended systems for tumor modeling \cite{GL17,GL17e,GLSS,KS22}. Investigating the global well-posedness and long-time behavior of global weak/strong solutions to the hydrodynamic system \eqref{nsch}--\eqref{ini0} with these general reaction terms constitutes an interesting and challenging task.

\textit{Plan of the paper}. The remainder of this paper is organized as follows.  In Section \ref{pm}, we introduce the mathematical settings and state the main results. In Section \ref{ws} we analyze a convective Cahn--Hilliard--diffusion system with a prescribed divergence-free velocity, which may have independent interest. Section \ref{proof-main} is devoted to the proofs of the main results.

\section{Main Results}\label{pm}
\setcounter{equation}{0}
\subsection{Preliminaries}
First, we introduce some notation and basic tools that will be used throughout the paper.

Let $X$ be a real Banach space. Its dual space is denoted by $X'$, and the duality pairing is denoted by $\langle \cdot,\cdot\rangle_{X',X}$. The bold letter $\bm{X}$ denotes the generic space of vectors or matrices, with each component belonging to $X$.
We assume that $\Omega \subset\mathbb{R}^3$ is a bounded domain with a sufficiently smooth boundary $\partial\Omega$. For any $q \in [1,+\infty]$, $L^{q}(\Omega)$ denotes the Lebesgue space with norm $\|\cdot\|_{L^{q}}$. For $k\in \mathbb{Z}^+$, $q\in [1,+\infty]$, $W^{k,q}(\Omega)$ denotes the Sobolev space with norm $\|\cdot\|_{W^{k,q}}$. The Hilbert space $W^{k,2}(\Omega)$ is denoted by $H^{k}(\Omega)$ with the norm $\|\cdot\|_{H^{k}}$.
For $s\geq 0$ and $q\in [1,\infty )$,
we denote by $H^{s,q}(\Omega )$ the Bessel-potential spaces and by $W^{s,q}(\Omega )$ the Slobodeckij spaces. It holds $H^{s,2}(\Omega)=W^{s,2}(\Omega )$ for all $s$, but for $q\neq 2$ the identity $H^{s,q}(\Omega )=W^{s,q}(\Omega )$ is only true if $s\in \mathbb{N}$. For clarity, we denote by $\mathbb{N}$ the set of natural numbers including zero. If $s\in \mathbb{N}$, then $H^{s,q}(\Omega )$ and $W^{s,q}(\Omega )$ coincide with the usual Sobolev spaces.
For simplicity of notation, the norm and inner product on $L^{2}(\Omega)$ (as well as $\bm{L}^{2}(\Omega)$) are denoted by $\|\cdot\|$ and $(\cdot,\cdot)$, respectively. Given a measurable set $J\subset \mathbb{R}$, $L^q(J;X)$ with $q\in [1,+\infty]$ denotes the space of Bochner measurable $q$-integrable/essentially bounded functions with values in the Banach space $X$. If $J=(a, b)$ is an interval, we write, for simplicity, $L^q(a,b;X)$. Moreover, $f\in L^q_{\mathrm{loc}}(a,+\infty;X)$ if and only if $f\in L^q(a,b;X)$ for every $b>a$. The space $L^q_{\mathrm{uloc}}([a,+\infty);X)$ consists of all measurable functions $f:[a,+\infty)\to X$ such that
$$
\|f\|_{L^q_{\mathrm{uloc}}([a,+\infty);X)}=\sup_{t\geq a}\|f\|_{L^q(t,t+1;X)}<+\infty.
$$
For $q\in [1,+\infty]$, $W^{1,q}(J;X)$ denotes the space of functions $f$ such that $f\in L^q(J;X)$ with $\partial_t f\in L^q(J;X)$, where $\partial_t $ denotes the vector-valued distributional derivative of $f$. For $q=2$, we shall set $H^1(J;X)=W^{1,q}(J;X)$ as well as $H^1_{\mathrm{uloc}}([a,+\infty);X)=W^{1,q}_{\mathrm{uloc}}([a,+\infty);X)$.
The notation $C([a,b];X)$ denotes the Banach
space of all continuous functions $f:[a,b]\to X$ equipped with the supremum
norm. Furthermore, $BC([a,+\infty); X)$ denotes the Banach space of all bounded and continuous functions $f : [a,+\infty)\to X$ endowed with the supremum norm, while $BUC([a,+\infty); X)$ stands for the subspace of all bounded and uniformly continuous functions. In addition, we define $BC_w([a,+\infty); X)$ as the topological vector space of all bounded and weakly continuous functions $f : [a,+\infty)\to X$.

For every $f\in (H^1(\Omega))'$, we denote its generalized mean on $\Omega$ by
$\overline{f}=|\Omega|^{-1}\langle f,1\rangle_{(H^1)',H^1}$; if $f\in L^1(\Omega)$, then it holds $\overline{f}=|\Omega|^{-1}\int_\Omega f \,\mathrm{d}x$.
We denote by
$$
L^2_{0}(\Omega):=\big\{f\in L^2(\Omega)\ |\ \overline{f} =0\big\}
$$
the linear subspace of $L^2(\Omega)$ with zero mean. In addition, in view of the homogeneous Neumann boundary condition \eqref{boundary}, we introduce the space
$$
H^2_{N}(\Omega):= \big\{f\in H^2(\Omega)\ |\  \partial_{\bm{n}}f=0 \ \textrm{on}\  \partial \Omega\big\}.
$$
Then we denote by $\mathcal{A}_N\in \mathcal{L}(H^1(\Omega),(H^1(\Omega))')$ the realization of minus Laplacian $-\Delta$ subject to the homogeneous Neumann boundary condition such that
\begin{equation}\nonumber
	\langle \mathcal{A}_N u,v\rangle_{(H^1)',H^1} := \int_\Omega \nabla u\cdot \nabla v \, \mathrm{d}x,\quad \forall\, u,v\in H^1(\Omega).
\end{equation}
For the linear space $V_0= H^1(\Omega)\cap L_0^2(\Omega)$, we denote
$$
V_0^{-1}= \big\{ u \in (H^1(\Omega))'\ |\  \overline{u}=0 \big\}\subset (V_0)'.
$$
The restriction of $\mathcal{A}_N$ from $V_0$ onto $V_0^{-1}$ is an isomorphism. In particular, $\mathcal{A}_N$ is positively defined on $V_0$ and is self-adjoint. We denote its inverse map by $\mathcal{N} =\mathcal{A}_N^{-1}: V_0^{-1} \to V_0$. For every $f\in V_0^{-1}$, $u= \mathcal{N} f \in V_0$ is the unique weak solution of the Neumann problem
$$
\begin{cases}
	-\Delta u=f, \quad \text{in} \ \Omega,\\
	\partial_{\bm{n}} u=0, \quad \ \  \text{on}\ \partial \Omega.
\end{cases}
$$
For any $f\in V_0^{-1}$, we set $\|f\|_{V_0^{-1}}=\|\nabla \mathcal{N} f\|$.
It is well-known that $f \to \|f\|_{V_0^{-1}}$ and $
f \to\big(\|f-\overline{f}\|_{V_0^{-1}}^2+|\overline{f}|^2\big)^\frac12$ are equivalent norms on $V_0^{-1}$ and $(H^1(\Omega))'$,
respectively (see, e.g., \cite{MZ04}). From the Poincar\'{e}--Wirtinger inequality:
\begin{equation}
	\notag
	\|f-\overline{f}\|\leq C \|\nabla f\|,\qquad \forall\,
	f\in H^1(\Omega),
\end{equation}
where $C>0$ depends only on $\Omega$, we find that $f\to \|\nabla f\|$ and $f\to \big(\|\nabla f\|^2+|\overline{f}|^2\big)^\frac12$ are equivalent norms on $V_0$ and $H^1(\Omega)$, respectively.
Moreover, we report the following standard Hilbert interpolation inequality and elliptic estimates for the Neumann problem
\begin{align*}
	&\|f\|  \leq \|f\|_{V_0^{-1}}^{\frac12} \| \nabla f\|^{\frac12},
	\qquad\quad \forall\, f \in V_0, \\
	&\|\nabla \mathcal{N} f\|_{\bm{H}^{k}}  \leq C \|f\|_{H^{k-1}},
	\qquad \forall\, f\in H^{k-1}(\Omega)\cap L^2_0(\Omega),\quad k\in\mathbb{Z}^+.
\end{align*}

Next, we introduce the Hilbert spaces of solenoidal vector-valued functions (see, e.g., \cite{G,S}). We denote by $\bm{L}^2_{0,\mathrm{div}}(\Omega)$, $\bm{H}^1_{0,\mathrm{div}}(\Omega) $ the closure of   $$
C_{0,\mathrm{div}}^{\infty}(\Omega;\mathbb{R}^3)=\big\{\bm{f}\in C_0^{\infty}(\Omega;\mathbb{R}^3):\, \textrm{div}\bm{f}=0\big\}
$$
in $\bm{L}^2(\Omega)$ and $\bm{H}^1(\Omega)$, respectively. For $\bm{L}^2_{0,\mathrm{div}}(\Omega)$, we also use $(\cdot,\cdot)$ and $\|\cdot\|$ for its inner product and norm. For any function $\bm{f} \in \bm{L}^2(\Omega)$, the Helmholtz--Weyl decomposition holds (see  \cite[Chapter \uppercase\expandafter{\romannumeral3}]{G}):
\be
\bm{f}=\bm{f}_{0}+\nabla z,\quad\text{where}\ \ \bm{f}_{0} \in \bm{L}^2_{0,\mathrm{div}}(\Omega),\ z \in H^1(\Omega).\nonumber
\ee
As a consequence, we define the Leray projection $\bm{P}:\bm{L}^2(\Omega)\to \bm{L}^2_{0,\mathrm{div}}(\Omega)$ such that $\bm{P}(\bm{f})=\bm{f}_{0}$.
It holds $\|\bm{P}(\bm{f})\|\leq \|\bm{f}\|$ for all $\bm{f}\in \bm{L}^2(\Omega)$.
The space $\bm{H}^1_{0,\mathrm{div}}(\Omega)$ is equipped with the inner product $(\bm{u},\bm{v})_{\bm{H}^1_{0,\mathrm{div}}}:=(\nabla \bm{u},\nabla \bm{v})$ and the norm $\|\bm{u}\|_{\bm{H}^1_{0,\mathrm{div}}}=\|\nabla \bm{u}\|$.
From Korn's inequality
$$
\|\nabla \bm{u}\|\leq \sqrt{2}\|D\bm{u}\|,\quad \forall\, \bm{u}\in \bm{H}^1_{0,\mathrm{div}}(\Omega),
$$
we see that $\|D\bm{u}\|$ yields an equivalent norm for $\bm{H}^1_{0,\mathrm{div}}(\Omega)$. Next, we consider the Stokes operator $\bm{S}=\bm{P}(-\Delta)$ with $D(\bm{S})= \bm{H}^1_{0,\mathrm{div}}(\Omega)\cap\bm{H}^2(\Omega)$.
The space $D(\bm{S})$ is equipped with the inner product $(\bm{S}\bm{u},\bm{S}\bm{v})$ and the equivalent norm $\|\bm{S}\bm{u}\|$ (see, e.g., \cite[Chapter III]{S}).
For any $\bm{u}\in D(\bm{S})$ and $\bm{\zeta} \in \bm{H}^1_{0,\mathrm{div}}(\Omega)$, it holds
$(\bm{S}\bm{u},\bm{\zeta})=(\nabla \bm{u},\nabla\bm{\zeta})$. The operator $\bm{S}$ is a canonical isomorphism from $\bm{H}^1_{0,\mathrm{div}}(\Omega)$ to $(\bm{H}^1_{0,\mathrm{div}}(\Omega))'$ and we denote its inverse by $\bm{S}^{-1}$. For any $\bm{f}\in (\bm{H}^1_{0,\mathrm{div}}(\Omega))'$, there is a unique $\bm{u}=\bm{S}^{-1}\bm{f}\in\bm{H}^1_{0,\mathrm{div}}(\Omega)$ such that
$$
(\nabla\bm{S}^{-1}\bm{f},\nabla \bm{\zeta}) =\langle\bm{f},\bm{\zeta}\rangle_{(\bm{H}^1_{0,\mathrm{div}})', \bm{H}^1_{0,\mathrm{div}}},
\quad \forall\,\bm{\zeta} \in \bm{H}^1_{0,\mathrm{div}}(\Omega).
$$
Hence, $\|\nabla\bm{S}^{-1}\bm{f}\| =\langle\bm{f},\bm{S}^{-1}\bm{f}\rangle_{(\bm{H}^1_{0,\mathrm{div}})',\bm{H}^1_{0,\mathrm{div}}}^{\frac{1}{2}}$ yields an equivalent norm on $(\bm{H}^1_{0,\mathrm{div}}(\Omega))'$.

Let us recall the Ladyzhenskaya, Agmon and Gagliardo--Nirenberg inequalities in three dimensions
\begin{align*}
	&\|f\|_{L^3}\leq C\|f\|_{H^1}^\frac12\|f\|^\frac12,\qquad\quad\   \forall\,f\in H^1(\Omega),\\
	&\|f\|_{L^\infty}\leq  C\|f\|_{H^2}^\frac12\|f\|_{H^1}^\frac12,\qquad \ \ \forall\,f\in H^2(\Omega),\\
	&\|f\|_{W^{1,4}}\leq C\|f\|_{H^2}^\frac12\|f\|_{L^\infty}^\frac12,\qquad \forall\, f\in H^2(\Omega),
\end{align*}
where the constant $C>0$ depends only on $\Omega$.
We also recall the following regularity result for the Stokes operator (see e.g., \cite[Chapter III, Theorem 2.1.1]{S} and \cite[Lemma B.2]{GMT}):
\bl
Let $\Omega$ be a bounded domain of class $C^2$ in $\mathbb{R}^3$. For any $\bm{f} \in \bm{L}^2_{0,\mathrm{div}}(\Omega)$,
there exists a unique pair $(\bm{u},p)\in D(\bm{S})\times V_0$ such that $-\Delta \bm{u}+\nabla p=\bm{f}$ a.e. in $\Omega$, that is, $\bm{u}=\bm{S}^{-1}\bm{f}$. Moreover, it holds
$$
\|\bm{u}\|_{\bm{H}^2}+\|\nabla p\|\le C\|\bm{f}\|,
$$
where the constant $C>0$ only depends on $\Omega$.
\el

In subsequent sections, the symbols $C$, $C_i$ represent generic positive constants that can change from line to line. The specific dependence of these constants in terms of the data will be pointed out if necessary.

\subsection{Main results}
\noindent
Throughout the paper, we adopt the following hypotheses.
\begin{enumerate}
	\item[(H1)] The viscosity $\nu$ satisfies $\nu \in C^{2}(\mathbb{R})$ and
	\be
	\nu_{*} \leq \nu(r)\leq \nu^*,\quad |\nu'(r)|\leq \nu_{0},\quad|\nu''(r)|\leq \nu_{1},\quad \forall\, r \in \mathbb{R},\nonumber
	\ee
	where $\nu_{*}$, $\nu^*$, $\nu_{0}$ and $\nu_{1}$ are some positive constants.
	\item[(H2)] The singular potential $\varPsi$ belongs to the class of functions $C\big([-1,1]\big)\cap C^{3}\big((-1,1)\big)$ and can be decomposed into the following form
	\begin{equation}
	\varPsi(r)=\varPsi_{0}(r)-\frac{\theta_{0}}{2}r^2,\nonumber
	\end{equation}
	such that
	\begin{equation}
	\lim_{r\to \pm 1} \varPsi_{0}'(r)=\pm \infty\quad \text{and}\quad  \varPsi_{0}''(r)\ge \theta,\quad \forall\, r\in (-1,1),\nonumber
	\end{equation}
 with strictly positive constants $\theta_{0},\  \theta$ satisfying $\theta_{0}-\theta>0$.
 There exists some small $\epsilon_0\in(0,1)$ such that $\varPsi_{0}''$ is non-decreasing in $[1-\epsilon_0,1)$ and non-increasing in $(-1,-1+\epsilon_0]$.
	Besides, we make the extension $\varPsi_{0}(r)=+\infty$ for $r\notin[-1,1]$.
	\item[(H3)] The coefficients $\chi$, $c_0$,  $\alpha$, $\beta$ are  prescribed  constants that satisfy
	\be
	\chi \in \mathbb{R},\quad c_0\in(-1,1),\quad \alpha\geq 0,\quad \beta\in\mathbb{R}. \nonumber
	\ee
\end{enumerate}
\begin{remark}\rm
 As indicated in \cite[Remark 2.1]{GMT}, one can extend the linear viscosity function \eqref{vis} to $\mathbb{R}$ in such a way to comply \textrm{(H1)}. In fact, since the singular potential guarantees that the phase function satisfies $\varphi\in [-1,1]$, the value of $\nu$ outside of $[-1,1]$ is not important and can be chosen as in \textrm{(H1)}. It is straightforward to check that the logarithmic potential \eqref{pot} satisfies the assumption \textrm{(H2)}.
\end{remark}

We first report a result on the existence of global weak solutions to problem \eqref{nsch}--\eqref{ini0} on $[0,+\infty)$.
\bp[Existence of global weak solutions] \label{main}
Let $\Omega$ be a bounded domain in $\mathbb{R}^3$, with boundary $\partial \Omega$ of class $C^2$. Suppose that hypotheses (H1)--(H3) are satisfied. For any initial data $\bm {v}_{0} \in \bm {L}^2_{0,\mathrm{div}}(\Omega)$, $\varphi_{0}\in H^1(\Omega)\cap L^\infty(\Omega)$ with $\|  \varphi_{0} \|_{L^{\infty}} \le 1$, $|\overline{\varphi}_{0}|<1$ and $\sigma_{0}\in L^2(\Omega)$, the initial boundary value problem \eqref{nsch}--\eqref{ini0} admits a global weak solution $(\bm{v},\varphi,\mu,\sigma)$ on $[0,+\infty)$ satisfying
\begin{align}
&\bm{v} \in L^{\infty} (0,+\infty;\bm{L}^2_{0,\mathrm{div}}(\Omega)) \cap L^{2}(0,+\infty;\bm{H}^1_{0,\mathrm{div}}(\Omega))\cap  W^{1,\frac{4}{3}}_{\mathrm{uloc}}([0,+\infty);(\bm{H}^1_{0,\mathrm{div}}(\Omega))'),
\notag\\
&\varphi \in L^{\infty}(0,+\infty;H^1(\Omega))\cap L^{4}_{\mathrm{uloc}}([0,+\infty);H^2_{N}(\Omega))\cap L^2_{\mathrm{uloc}}([0,+\infty);W^{2,6}(\Omega)),
\notag \\
&\partial_t\varphi\in L^2(0,+\infty;(H^1(\Omega))'),
\quad \varPsi^{\prime}(\varphi) \in L_{\mathrm{uloc}}^{2}([0, +\infty); L^6(\Omega)),
\notag\\
&\mu \in L^{2}_{\mathrm{uloc}}([0,+\infty);H^1(\Omega)),\quad \nabla \mu \in   L^{2}(0,+\infty;\bm{L}^2(\Omega)),
\notag\\
&\varphi\in L^{\infty}(\Omega\times (0,+\infty))\ \textrm{with}\ \ |\varphi(x,t)|<1\ \ \textrm{a.e.\ in}\ \Omega\times(0,+\infty),
\notag\\
&\sigma  \in L^{\infty}(0,+\infty;L^2(\Omega))\cap L^{2}_{\mathrm{uloc}}([0,+\infty);H^1(\Omega)) \cap W^{1,\frac{4}{3}}_{\mathrm{uloc}}([0,+\infty);(H^1(\Omega))'),
\notag
\end{align}
such that for almost all $t\in (0,+\infty)$,
\begin{align}
&\left \langle\partial_t  \bm{ v},\bm{\zeta}\right \rangle_{(\bm{H}^1_{0,\mathrm{div}})',\bm{H}^1_{0,\mathrm{div}}}
+\big(( \bm{ v} \cdot \nabla)  \bm {v},\bm{ \zeta}\big)
+\big(  2\nu(\varphi) D\bm{v},D\bm{ \zeta}\big)
=\big((\mu+\chi \sigma)\nabla \varphi,\bm {\zeta}\big),
&\label{test3.c} \\
&\left \langle \partial_t \varphi,\xi\right \rangle_{(H^1)',H^1}
+({\bm{v} \cdot \nabla \varphi},\xi)=- (\nabla \mu,\nabla \xi)-\alpha(\overline{\varphi}-c_0,\xi),&\label{test1.a} \\
&\left \langle\partial_t \sigma,\xi\right \rangle_{(H^1)',H^1}
+({\bm{v} \cdot \nabla \sigma},\xi) + (\nabla \sigma,\nabla \xi)= \chi ( \nabla \varphi,\nabla \xi),& \label{test2.b}
\end{align}
hold for all $\bm {\zeta} \in \bm{H}^1_{0,\mathrm{div}}$ and $\xi \in H^1(\Omega)$, where
\begin{align}
& \mu=-\Delta \varphi+\varPsi'(\varphi)-\chi \sigma+\beta\mathcal{N}(\varphi-\overline{\varphi}), \quad
\text{a.e. in}\ \Omega\times (0,+\infty),
\label{test4.d}
\end{align}
as well as
$$
\bm {v}|_{t=0}=\bm{v}_{0},\quad \varphi|_{t=0}=\varphi_{0},\quad  \sigma|_{t=0}=\sigma_{0},\quad \text{a.e. in}\ \Omega.
$$
In addition, for the total energy
$$
E(\bm{v},\varphi, \sigma)= \int_{\Omega} \frac{1}{2}|\bm{v}|^2\, \mathrm{d}x+ \mathcal{F}(\varphi,\sigma),
$$
with $\mathcal{F}(\varphi,\sigma)$ given by \eqref{fe1}, the following energy inequality
\begin{align}
& E(\bm{v}(t_2),\varphi(t_2), \sigma(t_2))
+  \int_{t_1}^{t_2}\!\int_{\Omega} \left( 2\nu(\varphi)|D\bm{v}|^2 + |\nabla \mu|^2+|\nabla(\sigma-\chi\varphi)|^2\right) \mathrm{d}x\mathrm{d}s\nonumber\\
&\quad \leq E(\bm{v}(t_1),\varphi(t_1), \sigma(t_1)) -\alpha\int_{t_1}^{t_2}\!\int_\Omega (\overline{\varphi}-c_0)\mu\, \mathrm{d}x\mathrm{d}s,
\label{wBEL}
\end{align}
holds for all $t_2\in [t_1,+\infty)$ and almost all $t_1\in [0,+\infty)$ (including $t_1=0$).
\ep
\begin{remark}\rm \label{rem:weak-ini}
Applying a uniformly local variant of \cite[Theorem 4.10.2, Chapter III]{Amann} and \cite[Lemma 4.1]{A2009b}, we find
$$\bm{v} \in BC_w ([0,+\infty);\bm{L}^2_{0,\mathrm{div}}(\Omega)),\quad
\varphi \in BC_w([0,+\infty);H^1(\Omega)),\quad \sigma\in BC_w([0,+\infty);L^2(\Omega)),
$$
so that the initial data can be attained.
\end{remark}

Our aim is to show regularity and long-time behavior of global weak solutions of problem \eqref{nsch}--\eqref{ini0} obtained in Proposition \ref{main}. The results read as follows:

\begin{theorem}[Global regularity]
\label{2main}
Let $\Omega$ be a bounded domain in $\mathbb{R}^3$, with boundary $\partial \Omega$ of class $C^3$.
Suppose that hypotheses (H1)--(H3) are satisfied. Let $(\bm{v}, \varphi,\mu,\sigma)$ be a global weak solution to problem \eqref{nsch}--\eqref{ini0} given by Proposition \ref{main}. Then the following results hold:
	
	\emph{(1) Instantaneous regularity of} $(\varphi,\mu, \sigma)$: for any $\tau\in (0,1)$, we have
	$$
	\begin{aligned}
	& \varphi \in L^{\infty}(\tau, +\infty; W^{2, 6}(\Omega)), \quad \partial_t \varphi \in L^2_{\mathrm{uloc}}([\tau, +\infty); H^1(\Omega)), \\
	& \mu \in L^{\infty}(\tau, +\infty; H^1(\Omega)) \cap L_{\mathrm{uloc }}^2([\tau, +\infty); H^3(\Omega)), \\
    & \varPsi'(\varphi) \in L^{\infty}(\tau, +\infty; L^6(\Omega)),\\
	& \sigma\in L^{\infty}(\tau, +\infty; L^6(\Omega)),\quad \partial_t\sigma\in L^2(\tau,+\infty;(H^1(\Omega))'),
	\end{aligned}
	$$
The equations \eqref{f1.a}--\eqref{f4.d} are satisfied almost everywhere in $\Omega \times(\tau, +\infty)$ and the boundary condition $\partial_{\bm{n}}\mu=0$ holds almost everywhere on $\partial \Omega \times(\tau, +\infty)$.
	
	\emph{(2) Eventual separation property of} $\varphi$: there exist $T_{\mathrm{SP}}\geq 1$ and $\delta\in (0,1)$ such that
	$$
	|\varphi(x, t)| \leq 1-\delta, \quad \forall\,(x, t) \in \overline{\Omega }\times [T_{\mathrm{SP}}, +\infty).
	$$

	\emph{(3) Eventual regularity of} $(\bm{v},\sigma)$: there exist  $T_\mathrm{R}\geq T_{\mathrm{SP}}$ such that
	\begin{align*}
	& \bm{v} \in L^{\infty}(T_\mathrm{R}, +\infty; \bm{H}_{0,\mathrm{div}}^1(\Omega))
\cap L^2(T_\mathrm{R}, +\infty; \bm{H}^2(\Omega)),
\quad
\partial_t \bm{v}\in L^2(T_\mathrm{R}, +\infty; \bm{L}_{0,\mathrm{div}}^2(\Omega)),
\\
&\sigma \in L^\infty(T_\mathrm{R},+\infty; H^1(\Omega))\cap L^2_{\mathrm{uloc}}([T_\mathrm{R},+\infty); H^2_N(\Omega)),\quad
\partial_t\sigma \in L^2_{\mathrm{uloc}}([T_\mathrm{R},+\infty); L^2(\Omega)).
	\end{align*}
Moreover, we have $\varphi\in L^\infty(T_{\mathrm{R}},+\infty; H^3(\Omega))$.
\end{theorem}


\begin{theorem}[Convergence to a single equilibrium]\label{3main1}
Let the assumptions of Theorem \ref{2main} be satisfied.
Assume in addition, the potential function $\varPsi$ is real analytic in  $(-1,1)$. For every global weak solution $(\bm{v},\varphi,\sigma)$ to problem \eqref{nsch}--\eqref{ini0}, it holds
\begin{equation}
	\lim_{t\to +\infty} \big(\|\bm{v}(t)\|_{\bm{H}^1} +\|\varphi(t)-\varphi_\infty\|_{H^{3}} +\|\sigma(t)-\sigma_\infty\|_{H^{1}}\big) =0.
\notag
\end{equation}
Here, $(\varphi_{\infty},\sigma_{\infty})\in \big(H^3(\Omega)\cap H^2_N(\Omega)\big)\times H^2_N(\Omega)$ is a solution to the following stationary problem
\begin{subequations}
		\begin{alignat}{3}
		&-\Delta \varphi_\infty +\varPsi^{\prime}(\varphi_\infty) -\chi\sigma_\infty +\beta\mathcal{N}(\varphi_\infty-c_0) =\overline{\varPsi^{\prime}(\varphi_\infty)}
		-\chi\overline{\sigma_\infty},\quad
		&\ \text{in } \Omega,   \label{5bchv}\\
		& \Delta (\sigma_\infty-\chi\varphi_\infty)=0,
		&\ \text{in } \Omega,   \label{5f2.b}  \\
		&\partial_{\bm{n}} \varphi_\infty = \partial_{\bm{n}} \sigma_\infty=0,
		&\ \text{on } \partial \Omega, \label{5cchv}
		\end{alignat}
	\end{subequations}
subject to the mass constraints
$\overline{\sigma_\infty}=\overline{\sigma_0}$ and $\overline{\varphi_\infty}=c_0$ if $\alpha>0$, $\overline{\varphi_\infty}=\overline{\varphi_0}$ if $\alpha=0$.
Moreover, there exists some constant $\kappa\in(0,1/2)$ depending on $(\varphi_{\infty},\sigma_{\infty})$ such that
\begin{align}
	\|\bm{v}(t)\|+ \|\varphi(t)-\varphi_\infty\|_{H^1} +\|\sigma(t)-\sigma_\infty\| \le C(1+t)^{-\frac{\kappa}{1-2 \kappa}},\quad \forall\,t\geq 0,
\notag
\end{align}
where the constant $C>0$ depends on
$\|\bm{v}_0\|$, $\|\varphi_0\|_{H^1}$, $\|\varphi_\infty\|_{H^2}$, $\|\sigma_0\|$, $\|\sigma_\infty\|_{H^1}$, coefficients of the system and $\Omega$.
\end{theorem}

\subsection{Proof strategies}
\label{sec:sPF}
We conclude this section by discussing the applied techniques and proof strategies.

\subsubsection{Existence of global weak solutions in $[0,+\infty)$}
Regarding Proposition \ref{main}, for an arbitrary but fixed final time $T\in (0,+\infty)$, the existence of a weak solution $(\bm{v}, \varphi, \mu, \sigma)$ to problem \eqref{nsch}--\eqref{ini0} on $[0,T]$ can be established through a suitable semi-Galerkin scheme and a compactness argument similar to \cite[Section 3]{H}, with minor modifications. However, owing to the lack of uniqueness for weak solutions in three dimensions, in order to obtain the global existence on the unbounded interval $[0,+\infty)$, one cannot simply extend the solution defined on $[0,T]$ by letting $T\to +\infty$. Instead,  we observe that the approximate solutions $(\bm{v}^m, \varphi^m, \mu^m,\sigma^m)$ constructed via the semi-Galerkin scheme and Schauder's fixed-point theorem as in \cite{H} are indeed unique ($m\in \mathbb{Z}^+$ being the parameter in the Galerkin approximation). Furthermore, these solutions satisfy the following mass relations and the energy equality on their corresponding existence interval $[0,T_m]$:
\begin{align}
\overline{\varphi^m}(t)=c_0+e^{-\alpha t}(\overline{\varphi_0}-c_0),
\quad \overline{\sigma^m}(t)=\overline{\sigma_0},
\quad \forall\, t\in [0,T_m],
\label{awmass}
\end{align}
and
\begin{align}
& E(\bm{v}^m(t_2),\varphi^m(t_2), \sigma^m(t_2)) +  \int_{t_1}^{t_2}\!\int_{\Omega} \left( 2\nu(\varphi^m)|D\bm{v}^m|^2 + |\nabla \mu^m|^2+|\nabla(\sigma^m-\chi\varphi^m)|^2\right) \mathrm{d}x\mathrm{d}s\nonumber\\
&\quad = E(\bm{v}^m(t_1),\varphi^m(t_1), \sigma^m(t_1)) -\alpha\int_{t_1}^{t_2}\!\int_\Omega (\overline{\varphi^m}-c_0)\mu^m\, \mathrm{d}x\mathrm{d}s,
\quad 0\leq t_1<t_2\leq T_m,
\label{awBEL}
\end{align}
where
$$
E(\bm{v}^m,\varphi^m, \sigma^m)= \int_{\Omega} \frac{1}{2}|\bm{v}^m|^2\, \mathrm{d}x+ \mathcal{F}(\varphi^m,\sigma^m)
$$
and $\mathcal{F}$ is defined as in \eqref{fe1}. If $\alpha=0$, the total energy decreases with respect to time.
If $\alpha>0$, applying the same argument for \cite[(3.20)]{HW2022}, we can control the reminder term on the right-hand side of \eqref{awBEL} by
\begin{align}
-\alpha \int_\Omega (\overline{\varphi^m}-c_0)\mu^m\, \mathrm{d}x
\leq \frac12 \big(\|\nabla \mu^m\|^2+\|\nabla(\sigma^m-\chi\varphi^m)\|^2\big)
+ C(1+\alpha^3) \alpha e^{-\alpha t}|\overline{\varphi_{0}}-c_0|,\notag
\end{align}
for almost all $t\in [0,T_m]$, where $C>0$ depends on $\overline{\varphi_0}$, $\overline{\sigma_0}$, $\Omega$ and the coefficients of the system except $\alpha$. We note that the first two terms on the right-hand side of the above inequality can be absorbed into the left-hand side of \eqref{awBEL} (after integration on $[t_1,t_2]$), while the last term belongs to $L^1(0,+\infty)$. Thus, based on \eqref{awmass} and \eqref{awBEL}, for every fixed $m\in \mathbb{Z}^+$, we can obtain uniform estimates with respect to time for the approximate solution $(\bm{v}^m, \varphi^m, \mu^m,\sigma^m)$ and extend it from the finite interval $[0,T_m]$ to $[0,+\infty)$.
Similar arguments can be found in the proof of Theorem \ref{vch-weak} below. 
On the other hand, since \eqref{awmass} and \eqref{awBEL} provide sufficient estimates for $(\bm{v}^m, \varphi^m, \mu^m,\sigma^m)$ that are uniform with respect to the approximating parameter $m$ (see also the lower-order energy estimates in \cite[Section 3.2]{HW2022}), we can employ a compactness argument and a diagonal extraction process analogous to that in \cite[Chapter V, Section 1.3.6]{Bo13} for the three-dimensional Navier--Stokes equations. This allows us to find a subsequence of $\big\{(\bm{v}^m, \varphi^m, \mu^m,\sigma^m)\big\}$ that converges, as $m\to +\infty$, to a weak solution of problem \eqref{nsch}--\eqref{ini0} on the interval $[0, N]$ for every $N\in \mathbb{Z}^+$.
Consequently, the limit of this convergent subsequence yields a global weak solution to problem \eqref{nsch}--\eqref{ini0} defined on $[0,+\infty)$, possessing the required regularity properties stated in Proposition \ref{main}.

From the mass relation \eqref{awmass} (i.e., $\overline{\varphi^m}$ is bounded and indeed independent of $m$) and the fact that for any $T\in (0,+\infty)$ the approximate chemical potential $\mu^m$ weakly converges (up to a subsequence) in $L^2(0,T; L^2(\Omega))$, we find that for all $0\leq t_1\leq t_2\leq T$, it holds
$$
\alpha\int_{t_1}^{t_2}\!\!\int_\Omega (\overline{\varphi^m}-c_0)\mu^m\, \mathrm{d}x\mathrm{d}s\ \to\ \alpha\int_{t_1}^{t_2}\!\!\int_\Omega (\overline{\varphi}-c_0)\mu\, \mathrm{d}x\mathrm{d}s,
$$
as $m\to +\infty$ (again up to a subsequence). This observation, combined with the energy equality \eqref{awBEL} for the approximate solutions $(\bm{v}^m, \varphi^m, \mu^m,\sigma^m)$, and a standard argument for the Navier--Stokes--Cahn--Hilliard system (see, e.g., \cite[Theorem 1]{A2009}) enables us to conclude the energy inequality \eqref{wBEL}.

Taking into account the above discussions, the details of the proof for Proposition \ref{main} are omitted in this paper.

\subsubsection{Global regularity}

The proof of Theorem \ref{2main} is inspired by the recent work \cite{A2022}. One key point therein is to establish full regularity for the solution of the Cahn--Hilliard equation with the logarithmic potential \eqref{pot} and a divergence-free drift, under the weak assumption
$\bm{v}\in L^2(0,+\infty;\bm{H}^1_{0,\mathrm{div}}(\Omega))$ (see \cite[Theorem 2.4]{A2022}). For our problem \eqref{nsch}--\eqref{ini0}, in order to address the coupling due to chemotaxis and mass transport, we need to investigate the regularity of weak solutions to a convective Cahn--Hilliard--diffusion system for $(\varphi,\mu,\sigma)$ with a prescribed divergence-free velocity $\bm{v}\in L^2(0,+\infty;\bm{H}^1_{0,\mathrm{div}}(\Omega))$, as seen in \eqref{1f1}--\eqref{ini01}. Detailed analysis will be presented in Section \ref{ws}. Here, new challenges emerge in the second-order parabolic equation for $\sigma$ due to the low regularity of $\bm{v}$ and the cross-diffusion structure. This leads to a limited regularity of $\sigma$, which subsequently hinders the attainment of a higher-order spatial regularity for $\varphi$ through \eqref{1f4.d}. We first prove the existence of global weak solutions to problem \eqref{1f1}--\eqref{ini01} along with a conditional uniqueness result, as outlined in Theorem \ref{vch-weak}. Subsequently, in Theorem \ref{vch}, we establish the existence and uniqueness of more regular solutions under stronger assumptions on the initial data, with $\bm{v}$ belonging to the Leray--Hopf class, i.e.,
$$
\bm{v}\in L^\infty(0,+\infty;\bm{L}^2_{0,\mathrm{div}}(\Omega))\cap L^2(0,+\infty;\bm{H}^1_{0,\mathrm{div}}(\Omega)),
$$
which coincides with the regularity of the global weak solution to the hydrodynamic problem \eqref{nsch}--\eqref{ini0}. To prove Theorem \ref{vch}, we study two decoupled auxiliary systems for $\sigma$ and $(\varphi, \mu)$, respectively, and then apply a bootstrap argument. These results imply the instantaneous regularization of global weak solutions $(\varphi, \mu, \sigma)$ to problem \eqref{nsch}--\eqref{ini0} after an arbitrary positive time. Although we can recover the full regularity of $\varphi$ as in \cite{A2022}, there is a loss of regularity for $\sigma$ (due to the low regularity of $\bm{v}$). For example, with the initial datum $\sigma_0\in H^1(\Omega)$, we only obtain $\sigma \in L^\infty(0,+\infty, L^6(\Omega))$ instead of $\sigma \in L^\infty(0,+\infty, H^1(\Omega))$.

Next, an essential step towards achieving the eventual regularity of $\bm{v}$ is to investigate the long-time behavior of global weak solutions through the $\omega$-limit set, see Proposition \ref{sta}. Specifically, we demonstrate the strict separation of $\varphi$ from the pure states $\pm 1$ after a sufficiently large time. The eventual separation property is crucial because it ensures that the singular potential $\varPsi$ and its derivatives become smooth and bounded afterwards. This property further enables us to derive a weak-strong uniqueness result for the full system \eqref{nsch}--\eqref{ini0} via the relative energy method, see Proposition \ref{ls}. By combining the result on local strong well-posedness obtained in \cite{HW2022}, we find that every weak solution transitions into a strong one for sufficiently large time. Subsequently, the global regularity of $\bm{v}$ is achieved by applying the theory of global well-posedness for the Navier--Stokes system with variable viscosity \cite{A2009}. At last, with the improved regularity of $\bm{v}$, we can show that the chemical density $\sigma$ will regularize again, in turn improving the spatial regularity of $\varphi$.

Details of the proof for Theorem \ref{2main} will be presented in Section \ref{proof-main}.

\subsubsection{Convergence to a single equilibrium}
Theorem \ref{2main} implies that every global weak solution $(\boldsymbol{v}, \varphi,\mu,\sigma)$ to problem \eqref{nsch}--\eqref{ini0} becomes a global strong one after a sufficiently large time, with uniform-in-time estimates in $\bm{H}^1_{0,\mathrm{div}}(\Omega)\times H^3(\Omega)\times H^1(\Omega)\times H^1(\Omega)$. Furthermore, under the additional assumption that $\varPsi$ is real analytic in  $(-1,1)$, we can apply an extended {\L}ojasiewicz--Simon type inequality (see \cite[Theorem 4.1]{HW2022}) to demonstrate that the solution $(\boldsymbol{v}, \varphi,\sigma)$ converges to a single equilibrium $(\mathbf{0}, \varphi_\infty,\sigma_\infty)$ as $t\to +\infty$. The proof of Theorem \ref{3main1} follows the same argument as that for \cite[Theorem 2.2]{HW2022}, so its details are omitted in this paper.

\section{A Cahn--Hilliard--Diffusion System with Divergence-free Drift}
\label{ws}
\setcounter{equation}{0}
An essential step to prove Theorem \ref{2main} is to show the instantaneous regularizing effect of $\varphi$ and $\sigma$. To this end, let us consider the following convective Cahn--Hilliard--diffusion system
\begin{subequations} \label{1f1}
	\begin{alignat}{3}
	&\partial_t \varphi+\bm{v} \cdot \nabla \varphi=\Delta \mu-\alpha(\overline{\varphi}-c_0),\label{1f1.a} \\
	&\mu=- \Delta \varphi+  \varPsi'(\varphi)-\chi \sigma+\beta\mathcal{N}(\varphi-\overline{\varphi}),\label{1f4.d}\\
	&\partial_t \sigma+\bm{v} \cdot \nabla \sigma= \Delta (\sigma-\chi\varphi), \label{1f2.b}
	\end{alignat}
\end{subequations}
in $\Omega \times (0,+\infty)$, subject to the boundary and initial conditions
\begin{alignat}{3}
&{\partial}_{\bm{n}}\varphi={\partial}_{\bm{n}}\mu={\partial}_{\bm{n}}\sigma=0,\qquad\qquad &\textrm{on}& \ \partial\Omega\times(0,+\infty),
\label{boundary1}
\\
& \varphi|_{t=0}=\varphi_{0}(x), \ \ \sigma|_{t=0}=\sigma_{0}(x), \qquad &\textrm{in}&\ \Omega.
\label{ini01}
\end{alignat}
In \eqref{1f1}, the prescribed divergence-free velocity field $\bm{v}$ is supposed to satisfy a mild assumption on its regularity, that is,
$$
\bm{v} \in L^\infty(0,+\infty;\bm{L}^2_{0,\mathrm{div}}(\Omega)) \cap L^2(0, +\infty; \bm{H}_{0, \mathrm{div}}^1(\Omega)).
$$

The first theorem establishes the existence and conditional uniqueness of global weak solutions to problem \eqref{1f1}--\eqref{ini01}.

\bt\label{vch-weak}
Let $\Omega$ be a bounded domain in $\mathbb{R}^3$, with boundary $\partial \Omega$ of class $C^2$. Suppose that hypotheses (H2)--(H3) are satisfied and
$$
\bm{v} \in L^2(0, +\infty; \bm{H}_{0, \mathrm{div}}^1(\Omega)).
$$
\begin{itemize}
\item[\emph{(1)}] \emph{Existence}. For any initial datum $\left(\varphi_{0},\sigma_0\right)$ satisfying
$\varphi_{0} \in H^{1}(\Omega)\cap L^\infty(\Omega)$ with $\|\varphi_0\|_{L^\infty}\leq 1$, $|\overline{\varphi_0}|<1$ and $\sigma_{0} \in L^2(\Omega)$, problem \eqref{1f1}--\eqref{ini01} admits a global weak solution $(\varphi, \mu, \sigma)$ such that
$$
\begin{aligned}
 & \varphi\in L^{\infty}(0, +\infty ; H^1(\Omega))\cap L^{4}_{\mathrm{uloc}}([0,+\infty);H^2_{N}(\Omega)) \cap L^{2}_{\mathrm{uloc}}([0,+\infty);W^{2,6}(\Omega)),
 \\
 & \partial_t \varphi \in L^2(0, +\infty ; (H^1(\Omega))'),\quad \varPsi'(\varphi) \in L^{2}_{\mathrm{uloc}}([0, +\infty); L^6(\Omega))
 \\
 & \varphi\in L^{\infty}(\Omega \times(0, +\infty))\ \text { with }\ |\varphi(x, t)|<1\ \text { a.e. in } \Omega \times(0, +\infty),
 \\
 & \mu\in L_{\mathrm{uloc}}^{2}([0, +\infty); H^1(\Omega)),\quad \nabla \mu \in L^2(0, +\infty; \bm{L}^2(\Omega)),
 \\
 & \sigma\in L^{\infty}(0, +\infty ; L^2(\Omega))\cap L_{\mathrm{uloc}}^2([0, +\infty); H^1(\Omega))\cap   W^{1,\frac{4}{3}}_{\mathrm{uloc}}([0, +\infty); (H^1(\Omega))'),\\
 &\nabla (\sigma-\chi\varphi)\in L^2(0,+\infty; \bm{L}^2(\Omega)).
\end{aligned}
$$
For almost all $t\in (0,+\infty)$, the solution $(\varphi, \mu, \sigma)$ satisfies
%
\begin{subequations}
			\begin{alignat}{3}
			&\left \langle \partial_t \varphi,\xi\right \rangle_{(H^1)',H^1}
+({\bm{v} \cdot \nabla \varphi},\xi)=- (\nabla \mu,\nabla \xi)-\alpha(\overline{\varphi}-c_0,\xi),&
\notag \\
			&\left \langle\partial_t \sigma,\xi\right \rangle_{(H^1)',H^1}
+ ( \bm{v}\cdot\nabla \sigma, \xi) + (\nabla \sigma,\nabla \xi)
= \chi ( \nabla \varphi,\nabla \xi),&
\notag
			\end{alignat}
\end{subequations}
for all $\xi \in H^1(\Omega)$, where
\begin{align}
& \mu=-\Delta \varphi+\varPsi'(\varphi)-\chi \sigma+\beta\mathcal{N}(\varphi-\overline{\varphi}), \quad
\text{a.e. in}\ \Omega\times (0,+\infty).
\notag
\end{align}
Moreover, the initial conditions are fulfilled
$$
 \varphi|_{t=0}=\varphi_{0},\quad  \sigma|_{t=0}=\sigma_{0},\quad \text{a.e. in}\ \Omega,
$$
and the solution $(\varphi, \sigma)$ satisfies the following energy inequality:
\begin{align}
\frac12\|\sigma(t)\|^2
+\int_0^t \|\nabla \sigma\|^2\,\mathrm{d}s
\leq \frac12\|\sigma_0\|^2
+ \chi\int_0^t\!\int_\Omega \nabla \sigma \cdot\nabla \varphi \,\mathrm{d}x\mathrm{d}s,
\quad \forall\, t >0.
\label{sig-ener-app1b}
\end{align}
\item[\emph{(2)}] \emph{Conditional uniqueness}. For any $T\in (0,+\infty)$, let $(\varphi_i,\mu_i,\sigma_i)$, $i=1,2$, be two weak solutions to problem \eqref{1f1}--\eqref{ini01} on $[0,T]$ corresponding to the same initial datum $(\varphi_0, \sigma_0)$ as described in \emph{(1)}.
Assume, in addition, $\bm{v}\in L^\infty(0,+\infty;\bm{L}^2_{0,\mathrm{div}}(\Omega))$ and $\sigma_2\in L^8(0,T;L^4(\Omega))$,
then
$$
(\varphi_1(t),\mu_1(t),\sigma_1(t))=(\varphi_2(t),\mu_2(t),\sigma_2(t)),\quad \forall\, t\in [0,T].
$$
\end{itemize}
\et
\begin{remark}\rm \label{rem:w-ini}
Similarly to Remark \ref{rem:weak-ini}, from the regularity of $(\varphi, \sigma)$ and using \cite[Theorem 4.10.2, Chapter III]{Amann}, \cite[Lemma 4.1]{A2009b}, we can conclude $\varphi \in BC_w([0,+\infty);H^1(\Omega))$ and $\sigma\in BC_w([0,+\infty);L^2(\Omega))$, so that the initial data can be achieved.
\end{remark}

\smallskip

Next, we show that, under additional assumptions on the domain and initial data, the weak solution becomes more regular and unique.
\bt\label{vch}
Let $\Omega$ be a bounded domain in $\mathbb{R}^3$, with boundary $\partial \Omega$ of class $C^3$. Suppose that hypotheses (H2)--(H3) are satisfied and
$$
\bm{v} \in L^\infty(0,+\infty;\bm{L}^2_{0,\mathrm{div}}(\Omega))\cap L^2(0, +\infty; \bm{H}_{0, \mathrm{div}}^1(\Omega)).
$$
For any initial datum $(\varphi_{0},\sigma_0)$ satisfying
$\varphi_{0} \in H^{2}_N(\Omega)$, $\|\varphi_0\|_{L^\infty}\leq 1$, $|\overline{\varphi_0}|<1$, $\mu_0 =-\Delta \varphi_0+ \varPsi'(\varphi_0)\in H^1(\Omega)$ and $\sigma_{0} \in H^{1}(\Omega)$, problem \eqref{1f1}--\eqref{ini01} admits a unique global solution $(\varphi, \mu, \sigma)$  such that
$$
\begin{aligned}
 & \varphi\in L^{\infty}(0, +\infty; W^{2, 6}(\Omega)\cap H^2_N(\Omega)),\\
 & \partial_t\varphi \in L^2_{\mathrm{uloc}}(0, +\infty; H^1(\Omega)) \cap L^\infty(0,+\infty;(H^1(\Omega))'),
 \\
 & \varphi\in L^{\infty}(\Omega \times(0, +\infty))\ \text { with }\ |\varphi(x, t)|<1\ \text { a.e. in } \Omega \times(0, +\infty),
 \\
 & \varPsi'(\varphi) \in L^{\infty}(0, +\infty; L^6(\Omega)),
 \\
 & \mu\in L^{\infty}(0, +\infty; H^1(\Omega)) \cap L_{\mathrm{uloc}}^2([0, +\infty); H^3(\Omega)\cap H^2_N(\Omega)),
 \\
 & \sigma\in L^{\infty}(0, +\infty ; L^6(\Omega))\cap L_{\mathrm{uloc}}^2([0, +\infty); H^1(\Omega))
 \cap H^1(0, +\infty ; (H^1(\Omega))').
\end{aligned}
$$
The solution $(\varphi, \mu, \sigma)$ satisfies
\begin{subequations}
			\begin{alignat}{3}
			& \partial_t \varphi+\bm{v} \cdot \nabla \varphi=\Delta \mu-\alpha(\overline{\varphi}-c_0),&
            \notag \\
			&  \mu=-\Delta \varphi+\varPsi'(\varphi)-\chi \sigma+\beta\mathcal{N}(\varphi-\overline{\varphi}),&
            \notag
			\end{alignat}
\end{subequations}
almost everywhere in $\Omega \times (0,+\infty)$ and
\begin{align}
&\left \langle\partial_t \sigma,\xi\right \rangle_{(H^1)',H^1}
+ ( \bm{v}\cdot\nabla \sigma, \xi) + (\nabla \sigma,\nabla \xi)
= \chi ( \nabla \varphi,\nabla \xi),
\notag
\end{align}
for almost all $t\in (0,+\infty)$ and all $\xi \in H^1(\Omega)$.
The initial condition \eqref{ini01} is satisfied almost everywhere in $\Omega$. In addition, the solution $(\varphi, \mu, \sigma)$ satisfies the energy equality
\begin{align}
& \mathcal{F}(\varphi(t_2), \sigma(t_2))
+  \int_{t_1}^{t_2}\!\!\int_{\Omega} \left(|\nabla \mu|^2+|\nabla(\sigma-\chi\varphi)|^2\right) \mathrm{d}x\mathrm{d}s
\nonumber\\
&\quad = \mathcal{F}(\varphi(t_1), \sigma(t_1))-\alpha\int_{t_1}^{t_2}\!\!\int_\Omega (\overline{\varphi}-c_0)\mu\, \mathrm{d}x \mathrm{d}s  + \int_{t_1}^{t_2}\!\!\int_\Omega  \big[\bm{v}\cdot \nabla (\mu +\chi \sigma)\big] \varphi\,\mathrm{d}x \mathrm{d}s,
\label{a-wBEL}
\end{align}
for all $0\leq t_1\leq t_2< +\infty$, where $\mathcal{F}(\varphi,\sigma)$ is given by \eqref{fe1}.
\et
\begin{remark} \rm
For the single Cahn--Hilliard equation with non-smooth potential and divergence-free drift, existence and uniqueness of global weak solutions have been shown in \cite[Theorem 6]{A2009} and regularity properties under the mild requirement on the velocity have been established in \cite[Theorem 2.4]{A2022}. Theorem \ref{vch-weak} and Theorem \ref{vch} generalize those previous results to the Cahn--Hilliard--diffusion system \eqref{1f1} with a cross-diffusion structure and a nonlocal interaction of Oono's type. It is worth mentioning that under the condition $\bm{v} \in L^\infty(0,+\infty;\bm{L}^2_{0,\mathrm{div}}(\Omega)) \cap L^2(0, +\infty; \bm{H}_{0, \mathrm{div}}^1(\Omega))$ we can still gain full regularity of $(\varphi, \mu)$, but there is a loss in regularity of $\sigma$.
\end{remark}

\subsection{Existence of weak solutions and conditional uniqueness}
We start by proving the existence of global weak solutions to problem \eqref{1f1}--\eqref{ini01} on $[0,+\infty)$.
\subsubsection{The approximate problem}
\label{sec:CHS-app-pro}
We first regularize the velocity field $\bm{v}\in L^2(0, +\infty; \bm{H}_{0,\mathrm{div}}^1(\Omega))$.
Since $C^\infty_0((0, +\infty);D(\bm{S}))$ is dense in $L^2(0, +\infty; \bm{H}_{0,\mathrm{div}}^1(\Omega))$, there exists a sequence $\big\{\boldsymbol{v}^k\big\}\subset C^\infty_0((0, +\infty);D(\bm{S}))$ such that
\begin{align}
 \lim_{k\to+\infty}\|\bm{v}^k-\bm{v}\|_{L^2(0,+\infty; \bm{H}_{0,\mathrm{div}}^1(\Omega))}=0.
 \label{vk-con}
\end{align}
As a result, $\|\bm{v}^k\|_{L^2(0,+\infty; \bm{H}_{0,\mathrm{div}}^1(\Omega))}\leq \|\bm{v}\|_{L^2(0,+\infty; \bm{H}_{0,\mathrm{div}}^1(\Omega))}+1$ for $k$ sufficiently large. For every $k\in \mathbb{Z}^+$, consider the following approximate system with the prescribed velocity field $\bm{v}^{k}$:
\begin{subequations} \label{1f1-app1}
	\begin{alignat}{3}
	&\partial_t \varphi^k + \bm{v}^k \cdot \nabla \varphi^k=\Delta \mu^k-\alpha(\overline{\varphi^k}-c_0),
\label{1f1.a-app1} \\
	&\mu^k=- \Delta \varphi^k+  \varPsi'(\varphi^k)-\chi \sigma^k +\beta\mathcal{N}(\varphi^k-\overline{\varphi^k}),
\label{1f4.d-app1} \\
	&\partial_t \sigma^k+\bm{v}^k \cdot \nabla \sigma^k= \Delta (\sigma^k-\chi\varphi^k),
\label{1f2.b-app1}
	\end{alignat}
\end{subequations}
in $\Omega \times (0,+\infty)$, subject to the boundary and initial conditions
\begin{alignat}{3}
&{\partial}_{\bm{n}}\varphi^k={\partial}_{\bm{n}}\mu^k ={\partial}_{\bm{n}}\sigma^k=0,\qquad\qquad &\textrm{on}& \   \partial\Omega\times(0,+\infty),
\label{boundary1-app1}
\\
& \varphi^k|_{t=0}=\varphi_{0}(x), \ \ \sigma^k|_{t=0}=\sigma_{0}(x), \qquad &\textrm{in}&\ \Omega.
\label{ini01-app1}
\end{alignat}
Since $\bm{v}^{k}$ is sufficiently regular, for arbitrary but fixed $T\in (0,+\infty)$, following the proof of \cite[Lemma 3.1]{H}, we can show that problem \eqref{1f1-app1}--\eqref{ini01-app1} admits a unique weak solution $(\varphi^k, \mu^k, \sigma^k)$ on $[0,T]$ satisfying
\begin{align}
&\varphi^{k} \in L^\infty(0,T;H^1(\Omega))\cap L^{2}(0,T;H^2_{N}(\Omega)) \cap H^{1}(0,T;(H^1(\Omega))'),
\notag \\
&\mu^{k} \in   L^{2}(0,T;H^1(\Omega)),
\notag \\
&\sigma^{k}   \in C([0,T];L^2(\Omega))\cap L^{2}(0,T;H^1(\Omega))
\cap H^1(0,T;(H^1(\Omega))'),
\notag \\
&\varphi^k\in   L^{\infty}(\Omega\times(0,T))\ \ \textrm{with}\ \ |\varphi^{k}(x,t)|<1\quad \textrm{a.e.\ in}\ \Omega\times(0,T),
\notag
\end{align}
and
\be
 \sup _{0\leq t \leq T}\|\varphi^k(t)\|_{L^{\infty}} \leq 1. \label{phi-app1}
\ee
Besides, under assumptions (H2), (H3), an argument similar to that in \cite[Section 3]{GGW} yields
$$
\varphi^k\in L^2(0, T; W^{2, 6}(\Omega))\cap L^4(0,T; H^2(\Omega)),\quad
\varPsi'(\varphi^k) \in L^{2}(0, T; L^6(\Omega)).
$$

\subsubsection{Uniform estimates}\label{uni-es-appW}
Next, we derive estimates for the approximate solutions $(\varphi^k, \mu^k, \sigma^k)$ that are uniform with respect to the parameter $k$ and the time. Without loss of generality, below we consider $T\geq 1$.
\medskip

\textbf{First estimate}.
Integrating \eqref{1f1.a-app1} over $\Omega$, exploiting \eqref{boundary1-app1}, the incompressibility condition and no-slip boundary condition of $\bm{v}^k$, we obtain
\be
\frac{\mathrm{d}}{\mathrm{d}t}\big(\overline{\varphi^k}(t)-c_0\big) +\alpha\big(\overline{\varphi^k}(t)-c_0\big)=0,
\notag
\ee
for almost all $t\in (0,T)$. This gives
\be
\overline{\varphi^k}(t)
= c_0 + e^{-\alpha t}\big(\overline{\varphi^k}(0)-c_0\big)
= c_0+e^{-\alpha t}\big(\overline{\varphi_{0}}-c_0\big),
\quad\forall\,t\in [0,T]. \label{aver-app1}
\ee
In a similar way, from \eqref{1f2.b-app1} we can derive the mass conservation property for $\sigma^k$:
\be
\overline{\sigma^k}(t)=\overline{\sigma^k}(0)=\overline{\sigma_0}, \quad\forall\,t\in [0,T].
\label{1aver-app1}
\ee

\textbf{Second estimate}.
Multiplying \eqref{1f1.a-app1} by $\mu^k$, \eqref{1f2.b-app1} by $\sigma^k-\chi\varphi^k$, respectively, integrating over $\Omega$ and adding the resultants together, from \eqref{1f4.d-app1} we infer that
\begin{align}
&\frac{\mathrm{d}}{\mathrm{d}t} \mathcal{F}^k(t)  +\mathcal{D}^k(t)\notag \\
&\quad =
-\alpha\big(\overline{\varphi^k}-c_0\big)\int_\Omega \mu^k\,\mathrm{d}x
-\int_\Omega (\bm{v}^k\cdot\nabla\varphi^k)\mu^k\,\mathrm{d}x
+\chi \int_\Omega (\bm{v}^k\cdot\nabla\sigma^k)\varphi^k\,\mathrm{d}x \notag\\
&\quad =: R_1+R_2+R_3.
\label{1menergy-app1}
\end{align}
for almost all $t\in (0,T)$, where
\begin{align}
\mathcal{F}^k(t)&= \int_\Omega \Big(\frac12 |\nabla \varphi^k(t)|^2+  \varPsi(\varphi^k(t))- \chi\sigma^k(t)\varphi^k(t) + \frac12 |\sigma^k(t)|^2 \Big)\,\mathrm{d}x \notag\\
&\quad +\frac{\beta}{2}\|\nabla\mathcal{N}\big(\varphi^k(t)-\overline{\varphi^k(t)}\big)\|^2,
\notag \\
\mathcal{D}^k(t)& =\|\nabla \mu^k(t)\|^2 +\|\nabla(\sigma^k(t)-\chi\varphi^k(t))\|^2.
\notag
\end{align}
Applying the argument in \cite[Section 3.2]{HW2022}, from (H2) and H\"{o}lder's inequality, the initial energy can be estimated as follows
\begin{align}
\mathcal{F}^k(0)&\leq  \frac{1}{2}\big(1+\|\varphi_{0}\|_{H^1}^2\big)
+ |\Omega|\max_{r\in[-1,1]}|\varPsi(r)|  + \frac12(1+\chi^2)\|\sigma_{0}\|^2+ C(|\beta|+1)|\Omega|,
\label{iniE1}
\end{align}
where $C>0$ depends only on $\Omega$.
On the other hand, it follows from the $L^\infty$-estimate \eqref{phi-app1} that
\begin{align}
\left|\int_\Omega \chi\sigma^k(t)\varphi^k(t)\, \mathrm{d}x\right|
\leq |\chi|\|\sigma^k(t)\|\|\varphi^k(t)\|
\leq \frac14\|\sigma^k(t)\|^2 + \chi^2|\Omega|,
\quad \forall\,t\in [0,T].
\label{couple1}
\end{align}
Thus, from (H2), \eqref{phi-app1} and \eqref{couple1}, we obtain the lower bound
\begin{align}
\mathcal{F}^k(t)&\geq  \frac{1}{2}\|\varphi^k(t)\|_{H^1}^2
+\frac{1}{4}\|\sigma^k(t)\|^2 - C_1,
\quad \forall\,t\in [0,T],
\label{low-bd1}
\end{align}
where $C_1>0$ depends only on $\chi$, $\beta$, and $\Omega$.

Concerning the first term on the right-hand side of \eqref{1menergy-app1}, if $\alpha=0$, we simply have $R_1=0$. If $\alpha>0$, $R_1$ can be estimated as in \cite[(3.20)]{HW2022} such that (with $\gamma=0$ therein)
\begin{align}
R_1
\leq \frac12 \big(\|\nabla \mu^k\|^2+\|\nabla(\sigma^k-\chi\varphi^k)\|^2\big)
+ C_2 (1+\alpha^3) \alpha e^{-\alpha t}|\overline{\varphi_{0}}-c_0|,
\label{R1es-app1}
\end{align}
where $C_2>0$ depends on $\overline{\varphi_0}$, $\overline{\sigma_0}$, $\Omega$ and the coefficients of the system except $\alpha$.
For the other two terms, using integration by parts and \eqref{phi-app1}, we get
\begin{align}
R_2& = \int_{\Omega} (\bm{v}^k \cdot \nabla \mu^k) \varphi^k\, \mathrm{d} x
 \leq \|\bm{v}^k \|\| \nabla \mu^k \| \| \varphi^k\|_{L^\infty}
 \notag\\
 &\leq  \|\bm{v}^k \|^2 + \frac14\|\nabla \mu^k\|^2,
\label{vph-app1}
\end{align}
and
\begin{align}
R_3 & =  \chi
\int_{\Omega} \big[\bm{v}^k \cdot \nabla   (\sigma^k-\chi \varphi^k)\big]  \varphi^k \mathrm{d} x
\notag\\
 &\leq   |\chi|
  \|\bm{v}^k \|\| \nabla   (\sigma^k-\chi \varphi^k)\| \| \varphi^k\|_{L^\infty}
 \notag\\
 &\leq   \chi^2 \|\bm{v}^k \|^2 + \frac14\| \nabla   (\sigma^k-\chi \varphi^k)\|^2.
\label{vphb-app1}
\end{align}
Hence, it follows from \eqref{1menergy-app1} and \eqref{R1es-app1}--\eqref{vphb-app1} that
\begin{align}
&\frac{\mathrm{d}}{\mathrm{d}t}\widetilde{\mathcal{F}}^k(t)
+\frac14\mathcal{D}^k(t)
\leq  (1+\chi^2) \|\bm{v}^k\|^2,
\notag
\end{align}
for almost all $t\in (0,T)$, where
\begin{align}
&\widetilde{\mathcal{F}}^k(t)=\mathcal{F}^k(t)
+ C_2 (1+\alpha^3) e^{-\alpha t}|\overline{\varphi_{0}}-c_0|.
\notag
\end{align}
Integrating in time on $[0,t]\subset[0,T]$, we obtain
\begin{align}
\widetilde{\mathcal{F}}^k(t)
 + \frac14\int_0^t\mathcal{D}^k(s)\,\mathrm{d}s
 \leq  \widetilde{\mathcal{F}}^k(0)+ (1+\chi^2) \int_0^t \|\bm{v}^k(s)\|^2\,\mathrm{d}s, \quad \forall\, t\in[0,T],
\notag
\end{align}
which together with \eqref{iniE1} and \eqref{low-bd1} yields
\begin{align}
&\|\varphi^k(t)\|_{H^1}^2+\|\sigma^k(t)\|^2  +\int_{0}^{t}\big(\|\nabla \mu^k(s)\|^2+ \|\nabla (\sigma^k-\chi\varphi^k)(s)\|^2\big)\,  \mathrm{d} s \notag \\
& \quad\le C + 4(1+\chi^2) \int_0^t \|\bm{v}^k(s)\|^2\,\mathrm{d}s,\quad \forall\, t\in [0,T],
\label{low-es2-app1}
\end{align}
where $C>0$ depends on $\overline{\varphi_{0}}$, $\overline{\sigma_{0}}$, $\max_{r\in[-1,1]}|\varPsi(r)|$, $\|\varphi_{0}\|_{H^1}$, $\|\sigma_{0}\|$, the coefficients of the system, $\Omega$, but it is independent of $k$ and $T$. Similarly to the estimates \cite[(3.15), (3.16)]{HW2022} (with $\gamma=0$ therein), we can control the integration of the chemical potential by
\begin{align}
\left|\int_\Omega \mu^k\,\mathrm{d}x\right|
& \le  C\big(\|\nabla \mu^k\| + \|\sigma^k- \overline{\sigma^k}\|+ \|\varphi^k- \overline{\varphi^k}\|\big)\|\varphi^k-\overline{\varphi^k}\|+C
\notag\\
& \le  C\big(\|\nabla \mu^k\| +\|\nabla(\sigma^k-\chi\varphi^k)\| +\|\nabla \varphi^k\|\big)\|\nabla \varphi^k\| + C,
	\label{q2}
\end{align}
where $C>0$ depends on the coefficients of the system, $\overline{\varphi_0}$, $\overline{\sigma_0}$, and $\Omega$. Then it follows from \eqref{low-es2-app1}, \eqref{q2} and the Poincar\'{e}--Wirtinger inequality that
\begin{align}
\sup_{t\in [0,T-1]}\int_{t}^{t+1} \| \mu^k(s)\|^{2}_{H^1} \, \mathrm{d} s \leq C\big(1+\|\bm{v}^k\|_{L^2(0,T;\bm{L}^2(\Omega))}^4\big),
\label{low-es3-app1a}\\
\sup_{t\in [0,T-1]}\int_{t}^{t+1} \|\sigma^k(s)\|_{H^1}^2 \, \mathrm{d} s \leq C\big(1+\|\bm{v}^k\|_{L^2(0,T;\bm{L}^2(\Omega))}^2\big),
\label{low-es3-app1b}
\end{align}
where the constant $C>0$ is independent of $k$, $T$ and $t$. In addition, using the interpolation $\|f\|_{L^3}\leq C\|f\|_{H^1}^\frac12\|f\|^\frac12$, we infer from \eqref{low-es2-app1} and \eqref{low-es3-app1b} that
\begin{align}
\sup_{t\in [0,T-1]}\int_{t}^{t+1} \|\sigma^k(s)\|_{L^3}^4 \, \mathrm{d} s \leq C\big(1+\|\bm{v}^k\|_{L^2(0,T;\bm{L}^2(\Omega))}^4\big),
\label{low-es3-app1c}
\end{align}

\textbf{Third estimate}. Consider the elliptic problem for $\varphi^k$:
\begin{align}
\begin{cases}
- \Delta \varphi^k +  \varPsi_0'(\varphi^k) = \mu^k + \theta_0 \varphi^k +\chi \sigma^k -\beta\mathcal{N}(\varphi^k-\overline{\varphi^k}),\quad \text{in}\ \Omega\times (0,T),\\
\partial_{\bm{n}}\varphi^k=0,\qquad \qquad\qquad\qquad\qquad\qquad \qquad\qquad\qquad \quad \ \ \text{on}\ \partial \Omega\times (0,T).
\end{cases}
\notag
\end{align}
It follows from \cite[Lemma 7.3]{GGW}, \eqref{low-es2-app1} and \eqref{low-es3-app1b} that
\begin{align}
\sup_{t\in [0,T-1]}\int_t^{t+1} \|\varphi^k(s)\|_{H^{2}}^4\,\mathrm{d}s
\leq C\big(1+\|\bm{v}^k\|_{L^2(0,T;\bm{L}^2(\Omega))}^4\big).
\label{L4H2-app1}
\end{align}
Moreover, applying \cite[Lemma 7.4]{GGW}, we can deduce from \eqref{low-es2-app1}, \eqref{low-es3-app1a} and \eqref{low-es3-app1b} that
\begin{align}
& \sup_{t\in [0,T-1]}\int_t^{t+1} \|\varphi^k(s)\|_{W^{2,6}}^2\,\mathrm{d}s
+ \sup_{t\in [0,T-1]}\int_t^{t+1} \|\varPsi'_0(\varphi^k(s))\|_{L^{6}}^2\,\mathrm{d}s
\notag \\
&\qquad \leq C\big(1+\|\bm{v}^k\|_{L^2(0,T;\bm{L}^2(\Omega))}^4\big).
\label{L6-app1}
\end{align}
The positive constant $C$ in \eqref{L4H2-app1} and \eqref{L6-app1} is independent of $k$, $T$ and $t$.
\medskip

\textbf{Fourth estimate}. Concerning the time derivative $\partial_t\varphi^k$, we infer from \eqref{1f1.a-app1}, \eqref{phi-app1}, \eqref{aver-app1} and \eqref{low-es2-app1} that
\begin{align}
\int_0^{T} \|\partial_t\varphi^k(s)\|_{(H^1)'}^2\,\mathrm{d}s
&  \leq C \int_0^{T} \big(\|\bm{v}^k(s)\|^2\| \varphi^k(s)\|_{L^\infty}^2
+\|\nabla \mu^k(s)\|^2 + \alpha^2\|\overline{\varphi^k}(s)-c_0\|^2\big)\,\mathrm{d}s
\notag\\
&  \leq  C \big(1+ \|\bm{v}^k\|_{L^2(0,T;\bm{L}^2(\Omega))}^2\big),
\notag
\end{align}
where $C>0$ is independent of $k$ and $T$.
Similarly, using \eqref{1f2.b-app1}, \eqref{low-es2-app1} and \eqref{low-es3-app1c},
we can obtain the following estimate for $\partial_t\sigma^k$:
\begin{align}
 & \int_t^{t+1} \|\partial_t\sigma^k(s)\|_{(H^1)'}^\frac{4}{3}\,\mathrm{d}s \notag \\
 & \quad  \leq C \int_t^{t+1} \big(\|\bm{v}^k(s)\|_{\bm{L}^6}^\frac43 \| \sigma^k(s)\|_{L^3}^\frac43
+\|\nabla (\sigma^k-\chi\varphi^k)(s)\|^\frac{4}{3}\big) \,\mathrm{d}s \notag\\
& \quad \leq  C\left(\int_0^T  \|\bm{v}^k(s)\|_{\bm{H}^1}^2\,\mathrm{d}s \right)^\frac23
\left(\int_t^{t+1}  \|\sigma^k(s)\|_{L^3}^4\,\mathrm{d}s \right)^\frac13
+  \left(\int_t^{t+1} \|\nabla (\sigma^k-\chi\varphi^k)(s)\|^2 \,\mathrm{d}s\right)^\frac23
\notag\\
&\quad \leq C\big(1+\|\bm{v}^k\|_{L^2(0,T;\bm{L}^2(\Omega))}^2\big) \|\bm{v}^k\|_{L^2(0,T;\bm{H}^1(\Omega))}^\frac43
+ C\big(1+\|\bm{v}^k\|_{L^2(0,T;\bm{L}^2(\Omega))}^2\big),
\notag
\end{align}
for all $t\in [0,T-1]$, where $C>0$ is independent of $k$, $T$ and $t$.

\subsubsection{Existence}
\begin{proof}[\textbf{Proof of Theorem \ref{vch-weak}-(1)}]
For any $k\in \mathbb{Z}^+$, due to the arbitrariness of $T>0$ and the uniqueness of the approximate solution $(\varphi^k, \mu^k, \sigma^k)$, we can extend $(\varphi^k, \mu^k, \sigma^k)$ from $[0,T]$ to the entire interval $[0,+\infty)$. To construct a global weak solution to the original problem \eqref{1f1}--\eqref{ini01} on $[0,+\infty)$, we apply the standard compactness argument combined with a diagonal extraction process as in \cite[Chapter V, Section 1.3.6]{Bo13} for the three-dimensional Navier--Stokes equations.

In what follows, the convergence is always understood in the sense of a suitable subsequence (not relabeled for simplicity). On the time interval $[0,1]$, from the uniform estimates obtained above, we can find a triple $(\varphi_{(1)},\mu_{(1)},\sigma_{(1)})$ such that
\begin{align*}
&\varphi^k\to \varphi_{(1)}\quad &&\text{weakly star in}\ L^\infty(0,1;H^1(\Omega))\cap L^\infty(\Omega\times(0,1)),
\\
&\varphi^k\to \varphi_{(1)} \quad &&\text{weakly in}\ L^4(0,1;H^2_N(\Omega))\cap L^2(0,1;W^{2,6}(\Omega)),
\\
&\partial_t\varphi^k\to \partial_t \varphi_{(1)} \quad &&\text{weakly in}\
L^2(0,1;(H^1(\Omega))'),
\\
&\mu^k\to \mu_{(1)}  \quad &&\text{weakly in}\ L^2(0,1;H^1(\Omega)),
\\
&\sigma^k\to \sigma_{(1)}\quad &&\text{weakly star in}\ L^\infty(0,1;L^2(\Omega)),
\\
&\sigma^k\to \sigma_{(1)}\quad &&\text{weakly in}\ L^2(0,1;H^1(\Omega))\cap L^4(0,1;L^3(\Omega)),
\\
&\partial_t\sigma^k\to \partial_t\sigma_{(1)} \quad &&\text{weakly in}\
L^\frac43(0,1;(H^1(\Omega))').
\end{align*}
By the Aubin--Lions compactness lemma, we have
\begin{align*}
&\varphi^k\to \varphi_{(1)} \quad && \text{strongly in}\ L^\infty(0,1;H^{1-r}(\Omega))\cap L^4(0,1;H^{2-r}(\Omega)),
\\
&\sigma^k\to \sigma_{(1)} \quad && \text{strongly in}\ L^2(0,1;H^{1-r}(\Omega)),
\end{align*}
for any $r\in (0,1/2)$, which further implies
$$
\varphi^k\to \varphi_{(1)},\quad \sigma^k\to \sigma_{(1)} \quad \text{almost everywhere in}\ \Omega\times(0,1).
$$
The convergence of singular nonlinear terms $\big\{\varPsi'(\varphi^k)\big\}$ can be achieved using the argument in \cite[Section 2]{A2022}. We sketch it here for the convenience of the readers.
From the pointwise convergence of $\varphi^k$ and the fact that $|\varphi^k(x,t)|<1$ almost everywhere in $\Omega\times (0,1)$, we find $\varphi_{(1)}\in L^\infty(\Omega\times (0,1))$ with $|\varphi_{(1)}(x,t)|\leq 1$ almost everywhere in $\Omega\times (0,1)$.
Due to (H2), i.e., the continuity of $\varPsi'$ in $(-1,1)$, we see that $\varPsi'(\varphi^k)\to \widetilde{\varPsi}'(\varphi_{(1)})$ almost everywhere in $\Omega\times (0,1)$, where
$\widetilde{\varPsi}'(s)=\varPsi'(s)$ if $s\in (-1,1)$ and $\widetilde{\varPsi}'(\pm 1)=\pm \infty$. Hence, an application of Fatou's lemma yields
$$\int_{\Omega\times(0,1)} |\widetilde{\varPsi}'(\varphi_{(1)})|^2\,\mathrm{d}x\mathrm{d}t\leq \liminf_{k\to +\infty} \int_{\Omega\times(0,1)} |\varPsi'(\varphi^k)|^2\,\mathrm{d}x\mathrm{d}t<+\infty,$$
so that $\widetilde{\varPsi}'(\varphi_{(1)})\in L^2(\Omega\times(0,1))$, which implies   $|\varphi_{(1)}(x,t)|< 1$ almost everywhere in $\Omega\times (0,1)$. As a result, $\widetilde{\varPsi}'(\varphi_{(1)})=\varPsi'(\varphi_{(1)})$ holds almost everywhere in $\Omega\times (0,1)$. This observation, together with the uniform boundedness of $\big\{\varPsi'(\varphi^k)\big\}$ in $L^2(0,1;L^6(\Omega))$ implies that
\begin{align*}
& \varPsi'(\varphi^k)\to \varPsi'(\varphi_{(1)})\quad \text{weakly in}\ L^2(0,1;L^6(\Omega)).
\end{align*}
Next, from the strong convergence \eqref{vk-con} for $\bm{v}^k$, the weak/strong convergence properties of $\sigma^k$ and the incompressibility condition for $\bm{v}^k$, $\bm{v}$, we find
\begin{align*}
&\left|\int_0^1\!\! \int_\Omega (\bm{v}^k\cdot\nabla \sigma^k-\bm{v}\cdot\nabla \sigma_{(1)})\xi\,\mathrm{d}x\mathrm{d}t\right|
\\
&\quad \leq \left| \int_0^1\!\! \int_\Omega \sigma^k  (\bm{v}^k-\bm{v})\cdot\nabla \xi\,\mathrm{d}x\mathrm{d}t\right|
+ \left|\int_0^1\!\! \int_\Omega (\sigma^k-\sigma_{(1)}) \bm{v}\cdot\nabla \xi\,\mathrm{d}x\mathrm{d}t\right|
\\
&\quad \leq \|\sigma^k\|_{L^\infty(0,1; L^2(\Omega))}^\frac12\|\sigma^k\|_{L^2(0,1;H^1(\Omega))}^\frac12\|\bm{v}^k-\bm{v}\|_{L^2(0,1;\bm{L}^6(\Omega))}\|\nabla \xi\|_{L^4(0,1;\bm{L}^2(\Omega))}
\\
&\qquad + \left|\int_0^1\!\! \int_\Omega (\sigma^k-\sigma_{(1)}) \bm{v}\cdot\nabla \xi\,\mathrm{d}x\mathrm{d}t\right|
\\
&\quad \to 0,\quad \text{as}\ k\to +\infty,
\end{align*}
for any $\xi \in L^4(0,1;H^1(\Omega))$. Here, for the convergence of the second integral we have used the weak convergence $\sigma^k\to \sigma_{(1)}$ in $L^4(0,1;L^3(\Omega))$ and the fact $\bm{v}\cdot\nabla \xi \in L^{\frac43}(0,1;L^\frac{3}{2}(\Omega))$. Similarly, one can check that
\begin{align*}
&\left|\int_0^1\!\! \int_\Omega (\bm{v}^k\cdot\nabla \varphi^k-\bm{v}\cdot\nabla \varphi_{(1)})\zeta\,\mathrm{d}x\mathrm{d}t\right|
\to 0,\quad \text{as}\ k\to +\infty,
\end{align*}
for any $\zeta \in L^2(0,1;H^1(\Omega))$.
Hence, we are able to pass the limit as $k\to +\infty$ in the weak formulation for $(\varphi^k, \mu^k, \sigma^k)$ and conclude that the limit triple $(\varphi_{(1)}, \mu_{(1)}, \sigma_{(1)})$ is a weak solution to problem \eqref{1f1}--\eqref{ini01} on the time interval $[0,1]$ with expected regularity properties. Verification of the initial conditions is straightforward (cf. also Remark \ref{rem:w-ini}).

Taking the above subsequence $\big\{(\varphi^k, \mu^k, \sigma^k)\big\}$ that is convergent on $[0,1]$, using now the uniform estimates on the time interval $[0,2]$, we can exact a further convergent subsequence with the limit $(\varphi_{(2)}, \mu_{(2)}, \sigma_{(2)})$, which yields a weak solution to problem \eqref{1f1}--\eqref{ini01} on $[0,2]$. From the construction, we find that
$$
(\varphi_{(2)}, \mu_{(2)}, \sigma_{(2)})=(\varphi_{(1)}, \mu_{(1)}, \sigma_{(1)})\quad \text{on}\ \ [0,1].
$$
By successive extractions, for any $N\in \mathbb{Z}^+$,  we are able to obtain a weak solution $(\varphi_{(N)}, \mu_{(N)}, \sigma_{(N)})$ to problem \eqref{1f1}--\eqref{ini01} on $[0,N]$ such that $$
(\varphi_{(N+1)}, \mu_{(N+1)}, \sigma_{(N+1)})=(\varphi_{(N)}, \mu_{(N)}, \sigma_{(N)})\quad \text{on}\ \ [0,N].
$$
Hence, applying the diagonal extraction process as in \cite{Bo13}, we can find a subsequence $\big\{(\varphi^k, \mu^k, \sigma^k)\big\}$ that converges towards a weak solution of problem \eqref{1f1}--\eqref{ini01} on any interval $[0, T]$, the limit of this subsequence is indeed a global weak solution defined on $[0,+\infty)$.

Let $\big\{(\varphi^k, \mu^k, \sigma^k)\big\}$ be the convergent subsequence constructed above.
We observed that $\sigma^k$ can be used as a test function in its weak formulation such that
\begin{align}
\frac12 \frac{\mathrm{d}}{\mathrm{d}t} \|\sigma^k\|^2
+\|\nabla \sigma^k\|^2 = \chi\int_\Omega \nabla \sigma^k\cdot\nabla \varphi^k\,\mathrm{d}x,
\notag
\end{align}
for almost all $t>0$. Integrating in time yields the energy equality
\begin{align}
\frac12\|\sigma^k(t)\|^2+\int_0^t \|\nabla \sigma^k(s)\|^2\,\mathrm{d}s
= \frac12\|\sigma_0\|^2+ \chi\int_0^t\! \int_\Omega \nabla \sigma^k\cdot\nabla \varphi^k\,\mathrm{d}x\mathrm{d}s,
\quad \forall\, t>0.
\notag
\end{align}
Applying the standard argument in \cite[Chapter V, Section 1.4]{Bo13} and passing to the limit as $k\to +\infty$, we can conclude the energy inequality \eqref{sig-ener-app1b}. This completes the proof of Theorem \ref{vch-weak}-(1).
\end{proof}

\subsubsection{Conditional uniqueness}
Unlike the single convective Cahn--Hilliard equation studied in \cite{A2009,A2022}, the uniqueness of weak solutions to problem \eqref{1f1}--\eqref{ini01} in three dimensions is not known due to the low regularity of $\bm{v}$ and $\sigma$. However, with the help of the energy inequality \eqref{sig-ener-app1b}, we are able to establish a result about conditional uniqueness.%

\begin{proof}[\textbf{Proof of Theorem \ref{vch-weak}-(2)}]
Let  $(\varphi_i,\mu_i,\sigma_i)$, $i=1,2$, be two weak solutions to problem \eqref{1f1}--\eqref{ini01} corresponding to the same initial datum $(\varphi_0, \sigma_0)$.
For convenience, we set their differences by
$$
\varphi=\varphi_1-\varphi_2,
\quad \mu=\mu_1-\mu_2,
\quad \sigma=\sigma_1-\sigma_2.
$$
Thanks to the additional regularity $\sigma_2\in L^8(0,T; L^4(\Omega))$ and the interpolation inequality
$\|\bm{v}\|_{\bm{L}^{4}}\leq C\|\bm{v}\|^{\frac{1}{4}}\|\nabla \bm{v}\|^\frac{3}{4}$,
for any $\xi\in L^2(0,T; H^1(\Omega))$, we have
\begin{align*}
& \left|\int_0^T\!\!\int_\Omega (\bm{v}\cdot \nabla \sigma_2)\xi\,\mathrm{d}x\mathrm{d}s\right|
=\left|\int_0^T\!\!\int_\Omega (\bm{v}\cdot \nabla \xi )\sigma_2\,\mathrm{d}x\mathrm{d}s\right|
\\
&\quad \leq \int_0^T \|\bm{v}\|_{\bm{L}^4}\|\sigma_2\|_{L^4}\|\nabla \xi\|\,\mathrm{d}s
\\
&\quad \leq C\int_0^T \|\bm{v}\|^{\frac{1}{4}} \|\nabla \bm{v}\|^\frac{3}{4} \|\sigma_2\|_{L^4}\|\nabla \xi\|\,\mathrm{d}s
\\
&\quad \leq C\|\bm{v}\|_{L^\infty(0,T;\bm{L}^2(\Omega))}^\frac{1}{4} \|\bm{v}\|_{L^2(0,T;\bm{H}^1(\Omega))}^{\frac{3}{4}} \|\sigma_2\|_{L^8(0,T;L^4(\Omega))} \|\xi\|_{L^2(0,T;H^1(\Omega))},
\end{align*}
which implies that $\partial_t\sigma_2\in L^2(0,T;(H^1(\Omega))')$ by comparison in the weak formulation for $\sigma_2$ (note that the regularity of $\varphi_2$ is sufficient for a test function in $L^2(0,T; H^1(\Omega))$). In a similar manner, for any $\zeta\in L^8(0,T; L^4(\Omega))$, it holds
\begin{align*}
& \left|\int_0^T\!\!\int_\Omega (\bm{v}\cdot \nabla \sigma_1)\zeta\,\mathrm{d}x\mathrm{d}s\right|
\\
&\quad \leq \int_0^T\|\bm{v}\|_{\bm{L}^4} \|\nabla \sigma_1\| \|\zeta\|_{L^4}\mathrm{d}s
\leq \int_0^T\|\bm{v}\|^{\frac{1}{4}}\|\nabla \bm{v}\|^\frac{3}{4}
\|\nabla \sigma_1\| \|\zeta\|_{L^4}\,\mathrm{d}s
\\
&\quad \leq
C\|\bm{v}\|_{L^\infty(0,T;\bm{L}^2(\Omega))}^{\frac{1}{4}} \|\bm{v}\|_{L^2(0,T;\bm{H}^1(\Omega))}^\frac{3}{4} \|\sigma_1\|_{L^2(0,T;H^1(\Omega))} \|\zeta\|_{L^8(0,T;L^4(\Omega))},
\end{align*}
which implies that
$$
\partial_t\sigma_1\in (L^8(0,T;L^4(\Omega)))' + L^2(0,T;(H^1(\Omega))').
$$

The above observations allow us to test the weak formulations for $\sigma_1$, $\sigma_2$ by $\sigma_2$, respectively. Meanwhile, $\sigma_1$ can be used as a test function in the weak formulation of $\sigma_2$ as well. Integrating the resultants on $[0,t]\subset[0,T]$, we obtain
\begin{align}
\frac12\|\sigma_2(t)\|^2
+\int_0^t \|\nabla \sigma_2(s)\|^2\,\mathrm{d}s
= \frac12\|\sigma_0\|^2+ \chi\int_0^t\!\int_\Omega \nabla \sigma_2 \cdot\nabla \varphi_2 \,\mathrm{d}x\mathrm{d}s,
\label{sig2-eq}
\end{align}
and
\begin{align}
& (\sigma_1(t),\sigma_2(t))
+ \,\underbrace{\int_0^t\! \int_\Omega \bm{v}\cdot \nabla (\sigma_1\sigma_2)\,\mathrm{d}x\mathrm{d}s}_{=0}\,
+ 2 \int_0^t\! \int_\Omega \nabla \sigma_1\cdot \nabla \sigma_2 \,\mathrm{d}x\mathrm{d}s
\notag \\
&\quad =\|\sigma_0\|^2 + \chi\int_0^t\!\int_\Omega \big( \nabla \sigma_1 \cdot\nabla \varphi_2 +  \nabla \sigma_2 \cdot\nabla \varphi_1\big) \,\mathrm{d}x\mathrm{d}s.
\label{sig12-eq}
\end{align}
On the other hand, we recall that the weak solution $(\varphi_1,\sigma_1)$ satisfies the energy inequality \eqref{sig-ener-app1b} on $[0,t]\subset [0,T]$. Adding it with \eqref{sig2-eq} and subtracting \eqref{sig12-eq}, we arrive at the following estimate on the difference $\sigma$:
\begin{align}
& \frac12\|\sigma(t)\|^2 + \int_0^t \|\nabla \sigma(s)\|^2\,\mathrm{d}s
\leq \chi \int_0^t\!\int_\Omega \nabla \sigma \cdot \nabla \varphi  \,\mathrm{d}x\mathrm{d}s
\notag\\
& \quad \leq  \frac12 \int_0^t \|\nabla \sigma (s)\|^2\,\mathrm{d}s
+ \frac{\chi^2}{2} \int_0^t \|\nabla \varphi(s)\|^2\,\mathrm{d}s,
\quad \forall\, t\in [0,T].
\label{diff-sig-app1}
\end{align}

Concerning the difference $\varphi$, since $\varphi(0)=0$, it holds
$$
\overline{\varphi}(t)=0,\quad \forall\, t\in [0,T].
$$
Taking advantage of this fact, testing the equation for $\varphi$ by $\mathcal{N}(\varphi-\overline{\varphi})=\mathcal{N} \varphi$, we get
\begin{align}
\frac12\frac{\mathrm{d}}{\mathrm{d}t} \|\varphi\|_{V_0^{-1}}^2
+ (\mu, \varphi) =
 (\bm{v}\varphi, \nabla  \mathcal{N} \varphi).
\notag
\end{align}
The term on the right-hand side can be estimated by Poincar\'{e}'s inequality, the Poincar\'e--Wirtinger inequality and Young's inequality such that
\begin{align*}
\left|(\bm{v}\varphi, \nabla  \mathcal{N}\varphi)\right|
&  \leq C\|\bm{v}\|_{\bm{L}^3}\| \varphi\|_{L^6}\| \nabla  \mathcal{N}\varphi\|
\leq \frac14\|\nabla \varphi\|^2 +C \|\bm{v}\|_{\bm{H}^1}^2\| \varphi\|_{V_0^{-1}}^2.
\end{align*}
Thanks to (H2), we can handle $(\mu, \varphi)$ by a similar argument like that in \cite[Section 4]{H}:
\begin{align*}
(\mu, \varphi) &\geq \frac34 \|\nabla \varphi\|^2- C\|\varphi\|_{V_0^{-1}}^2 - \frac12\|\sigma\|^2.
\end{align*}
Combining the above estimates, we get
\begin{align}
& \frac12\frac{\mathrm{d}}{\mathrm{d}t} \| \varphi\|_{V_0^{-1}}^2
+ \frac12\|\nabla \varphi\|^2
 \leq \frac12\|\sigma\|^2 + C\big(1+ \|\bm{v}\|_{\bm{H}^1}^2\big)\|\varphi\|_{V_0^{-1}}^2.
 \label{w-diff-vphi}
\end{align}
Integrating on $[0,t]\subset[0,T]$, we obtain
\begin{align}
\| \varphi(t)\|_{V_0^{-1}}^2
+  \int_0^t \|\nabla \varphi(s)\|^2 \,\mathrm{d}s
 \leq  \int_0^t \|\sigma(s)\|^2\,\mathrm{d}s
 + C\int_0^t \big(1+ \|\bm{v}(s)\|_{\bm{H}^1}^2\big)\|\varphi(s)\|_{V_0^{-1}}^2\,\mathrm{d}s. \label{diff-phi-app1}
\end{align}

If $\chi=0$, we infer from \eqref{diff-sig-app1} and Gronwall's lemma that $\sigma(t)=0$ for $t\in [0,T]$. Then by \eqref{diff-phi-app1} and Gronwall's lemma, we can further conclude $\varphi(t)=0$ for $t\in [0,T]$.
Otherwise, if $\chi\neq 0$, from \eqref{diff-sig-app1} and  \eqref{diff-phi-app1}, we find
\begin{align}
& \chi^2 \| \varphi(t)\|_{V_0^{-1}}^2 + \|\sigma(t)\|^2
\leq \max\{1,\chi^2\} \int_0^t \big(1+ \|\bm{v}(s)\|_{\bm{H}^1}^2\big)\big( \chi^2 \|\varphi(s)\|_{V_0^{-1}}^2+  \|\sigma(s)\|^2 \big)\,\mathrm{d}s,
\notag
\end{align}
which together with Gronwall's lemma again yields that $\varphi(t)=\sigma(t)=0$ for $t\in [0,T]$.

The proof of Theorem \ref{vch-weak}-(2) is complete.
\end{proof}

\subsection{Solutions with higher-order regularity}

In order to prove Theorem \ref{vch}, we first investigate two decoupled auxiliary problems for $\sigma$ and $\varphi$, respectively.

\subsubsection{Decoupled auxiliary problem for $\sigma$}
Let $\bm{v}, \varphi$ be two given functions. Consider the following initial boundary value problem of a convective reaction-diffusion equation for $\sigma$ subject to the homogeneous Neumann boundary condition:
 \begin{subequations} \label{1f1-sig}
	\begin{alignat}{3}
	&\partial_t \sigma+\bm{v} \cdot \nabla \sigma= \Delta (\sigma-\chi\varphi),  \qquad &\textrm{in}&\ \Omega\times(0,+\infty),
    \label{1f2.b-sig}\\
    &{\partial}_{\bm{n}}\sigma=0,\qquad\qquad &\textrm{on}& \ \partial\Omega\times(0,+\infty),
\label{boundary1-sig}\\
& \sigma|_{t=0}=\sigma_{0}(x), \qquad &\textrm{in}&\ \Omega. \label{ini01-sig}
	\end{alignat}
\end{subequations}

Our result reads as follows
\begin{proposition}
\label{sig-weak}
Let $\Omega$ be a bounded domain in $\mathbb{R}^3$, with boundary $\partial \Omega$ of class $C^2$. Suppose that $\chi\in \mathbb{R}$ and
$$
\bm{v} \in L^2(0, +\infty; \bm{H}_{0, \mathrm{div}}^1(\Omega)),\quad \varphi \in L^2_\mathrm{uloc}([0,+\infty);H^2_N(\Omega)).
$$
\begin{itemize}
\item[\emph{(1)}] \emph{Existence}. For any initial datum $\sigma_0 \in L^2(\Omega)$, problem \eqref{1f1-sig} admits a global weak solution
\begin{align}
  & \sigma\in L^{\infty}(0, +\infty ; L^2(\Omega))\cap L_{\mathrm{uloc}}^2([0, +\infty); H^1(\Omega))\cap   W^{1,\frac{4}{3}}_{\mathrm{uloc}}([0, +\infty); (H^1(\Omega))')
  \notag
\end{align}
that satisfies
\begin{align}
&\left \langle\partial_t \sigma,\xi\right \rangle_{(H^1)',H^1}
+ ( \bm{v}\cdot\nabla \sigma, \xi) + (\nabla \sigma,\nabla \xi)
= \chi ( \nabla \varphi,\nabla \xi),
\quad \forall\ \xi \in H^1(\Omega),
\label{test2.bw-sig}
\end{align}
almost everywhere in $(0,+\infty)$, and  the initial condition \eqref{ini01-sig} almost everywhere in $\Omega$. In addition, the solution $\sigma$ satisfies the following energy inequality:
\begin{align}
\frac12\|\sigma(t)\|^2
+\int_0^t \|\nabla \sigma(s)\|^2\,\mathrm{d}s
\leq \frac12\|\sigma_0\|^2
+ \chi\int_0^t\!\int_\Omega \nabla \sigma \cdot\nabla \varphi \,\mathrm{d}x\mathrm{d}s,
\quad \forall\, t \geq 0.
\label{asig-ener}
\end{align}
\item[\emph{(2)}] \emph{Conditional uniqueness}. For any $T\in (0,+\infty)$, let $\sigma_i$, $i=1,2$, be two weak solutions to problem \eqref{1f1-sig} on $[0,T]$ corresponding to the same initial datum $\sigma_0$ as described in \emph{(1)}.
Assume, in addition, $\bm{v}\in L^\infty(0,+\infty;\bm{L}^2_{0,\mathrm{div}}(\Omega))$ and $\sigma_2\in L^8(0,T;L^4(\Omega))$,
then
$$
\sigma_1(t)=\sigma_2(t),\quad \forall\, t\in [0,T].
$$
\item[\emph{(3)}] \emph{Regularity}.
Assume, in addition, $\bm{v}\in L^\infty(0,+\infty;\bm{L}^2_{0,\mathrm{div}}(\Omega))$,
$\varphi\in  L^2_\mathrm{uloc}([0,+\infty);W^{2,6}(\Omega))$ and $\sigma_0\in L^6(\Omega)$. Then problem \eqref{1f1-sig} admits a unique global weak solution satisfying
\begin{align}
  & \sigma\in L^{\infty}(0, +\infty ; L^6(\Omega))\cap L_{\mathrm{uloc}}^2([0, +\infty); H^1(\Omega))\cap   H^{1}_{\mathrm{uloc}}([0, +\infty); (H^1(\Omega))').
  \notag
\end{align}
\end{itemize}
\end{proposition}
\begin{proof}
The proof of conclusion (1) closely follows that of Theorem \ref{vch-weak}-(1). In what follows, we only derive \textit{a priori} estimates, neglecting the approximating procedure. First, we still have the mass conservation property
\be
\overline{\sigma}(t)=\overline{\sigma_0}, \quad\forall\,t\geq 0.
\label{1aver-app1-sig}
\ee
Next, testing \eqref{1f2.b-sig} by $\sigma-\overline{\sigma}$, integrating over $\Omega$, using \eqref{1aver-app1-sig}, H\"{o}lder's inequality, the Poincar\'{e}--Wirtinger inequality and Young's inequality, we obtain
\begin{align}
\frac12 \frac{\mathrm{d}}{\mathrm{d}t}\|\sigma-\overline{\sigma}\|^2
+ \|\nabla \sigma\|^2
&=-\chi\int_\Omega \Delta \varphi (\sigma-\overline{\sigma})\,\mathrm{d}x
\notag \\
&\leq |\chi|\|\Delta\varphi\|\|\sigma-\overline{\sigma}\|
\notag\\
&\leq C|\chi|\|\Delta\varphi\|\|\nabla \sigma\|
\notag\\
&\leq \frac12\|\nabla \sigma\|^2 +C\chi^2\|\Delta\varphi\|^2,
\label{aux-sig-1}
\end{align}
where $C>0$ depends only on $\Omega$. Applying the Poincar\'{e}--Wirtinger inequality again, we get
\begin{align}
 \frac{\mathrm{d}}{\mathrm{d}t}\|\sigma-\overline{\sigma}\|^2
+ C_1\|\sigma-\overline{\sigma}\|^2
& \leq C_2 \chi^2\|\Delta\varphi\|^2,
\label{aux-sig-2}
\end{align}
for $C_1,C_2>0$ depending only on $\Omega$. Applying Lemma \ref{GronW}, we obtain
\begin{align}
\|\sigma(t)-\overline{\sigma}(t)\|^2 \leq
2\|\sigma_0-\overline{\sigma_0}\|^2 e^{-C_1 t}
+ \frac{2e^{C_1}C_2\chi^2}{1-e^{-C_1}}\sup_{r\geq 0}\int_r^{r+1} \|\Delta \varphi(s)\|^2\,\mathrm{d}s,\quad \forall\, t\geq 0.
\label{aux-sig-3}
\end{align}
This estimate, together with \eqref{1aver-app1-sig}, yields
\begin{align}
\sup_{t\geq 0}\|\sigma(t)\|^2\leq C.
\label{aux-sig-3b}
\end{align}
Integrating \eqref{aux-sig-1} on $[t,t+1]$, using \eqref{aux-sig-3} and the Poincar\'{e}--Wirtinger inequality, we further get
\begin{align}
 \sup_{t\geq 0} \int_t^{t+1} \|\sigma(s)\|_{H^1}^2\, \mathrm{d}s \leq   C + C \sup_{t\geq 0}\int_t^{t+1} \|\Delta \varphi(s)\|^2\,\mathrm{d}s,
 \label{aux-sig-4}
\end{align}
where the constant $C>0$ depends on $\Omega$, $\|\sigma_0\|$, $\overline{\sigma_0}$ and $\chi$.
This implies $\sigma\in L^2_{\mathrm{uloc}}([0,+\infty);H^1(\Omega))$. An application of the interpolation inequality $\|\sigma\|_{L^3}\leq C\|\sigma\|_{H^1}^\frac12\|\sigma\|^\frac12$ leads to $\sigma\in L^4_{\mathrm{uloc}}([0,+\infty);L^3(\Omega))$. Besides, by a comparison in \eqref{test2.bw-sig}, we see that
\begin{align}
 \int_t^{t+1} \|\partial_t\sigma(s)\|_{(H^1)'}^\frac{4}{3}\,\mathrm{d}s
 &  \leq C\int_t^{t+1} \big(\|\bm{v}(s)\|_{\bm{L}^6}^\frac43 \| \sigma(s)\|_{L^3}^\frac43
+\|\nabla  \sigma(s)\|^\frac43 +\|\nabla \varphi(s)\|^\frac{4}{3}\big) \,\mathrm{d}s \notag\\
&  \leq  C\left(\int_0^\infty  \|\bm{v}(s)\|_{\bm{H}^1}^2\,\mathrm{d}s \right)^\frac23
\left(\int_t^{t+1}  \|\sigma(s)\|_{L^3}^4\,\mathrm{d}s \right)^\frac13\notag
\\
&\quad + C \int_t^{t+1} \|\nabla \sigma(s)\|^2\,\mathrm{d}s
+ C \int_t^{t+1} \|\nabla \varphi(s)\|^2\,\mathrm{d}s + C,
\notag
\end{align}
for all $t\geq 0$, which gives $\sigma\in W^{1,\frac43}_{\mathrm{uloc}}([0,+\infty);(H^1(\Omega))')$.
\smallskip

Due to the energy inequality \eqref{asig-ener}, the proof of conclusion (2) is the same as that of Theorem \ref{vch-weak}-(2) (cf. \eqref{diff-sig-app1} with $\varphi=0$ therein).
\smallskip

It remains to prove conclusion (3). First, we observe that, due to the stated regularity of the solution, its uniqueness is a direct consequence of the conclusion (2). In the following proof of existence, we only derive \textit{a priori} estimates, which can be justified by a suitable approximating procedure.

Testing \eqref{1f2.b-sig} by $\sigma^3$, integrating over $\Omega$, we get
\begin{align}
&\frac14 \frac{\mathrm{d}}{\mathrm{d} t}\|\sigma\|_{L^4}^4
+\frac34\|\nabla \sigma^2\|^2  =-\int_{\Omega} \chi\sigma^3\Delta\varphi \, \mathrm{d}x.
\label{energy e}
\end{align}
From the Sobolev embedding theorem and the Poincar\'{e}--Wirtinger inequality, it follows that
\be
\|\sigma\|_{L^{12}}^2=\|\sigma^2\|_{L^6}\le C\left(\|\nabla \sigma^2\|+\int_{\Omega} \sigma^2\,\mathrm{d}x\right).
\label{aux-sig-5}
\ee
The above estimate, \eqref{aux-sig-3b} and Young's inequality enable us to conclude
\begin{align*}
	\|\sigma^3\|_{L^\frac{6}{5}}^2 =\|\sigma\|_{L^\frac{18}{5}}^6
	&\le  \|\sigma\|^{\frac{14}{5}} \|\sigma\|_{L^{12}}^{\frac{16}{5}}
 \le  \|\nabla\sigma^2\|^2+C,
\end{align*}
which combined with H\"{o}lder's and Young's inequalities yields
\begin{align}
-\int_{\Omega}\chi\sigma^3\Delta\varphi \,\mathrm{d}x
& \le   |\chi|\|\sigma^3\|_{L^\frac{6}{5}}\|\Delta\varphi\|_{L^6}
\le \frac14\|\nabla \sigma^2\|^2 + C\chi^2 \|\Delta\varphi\|^2_{L^6} +C.
\label{l13}
\end{align}
From \eqref{energy e} and \eqref{l13}, we obtain
\begin{align}
&\frac14 \frac{\mathrm{d}}{\mathrm{d} t}\|\sigma\|_{L^4}^4
+\frac12\|\nabla \sigma^2\|^2
\le  C\chi^2 \|\Delta\varphi\|^2_{L^6}+C.
\label{energy e1}
\end{align}
On the other hand, using \eqref{aux-sig-3b}, \eqref{aux-sig-5}, H\"{o}lder's inequality, Young's inequality, the Sobolev embedding theorem and the Poincar\'{e}--Wirtinger inequality, we get
\begin{align}
\|\sigma\|_{L^4}^4 =\int_\Omega\sigma^4\,\mathrm{d}x
&\le \|\sigma\|_{L^{12}}^2\|\sigma\|_{L^3} \|\sigma\|
\notag\\
&\le \frac12\|\sigma\|_{L^{12}}^4 + \frac12 \|\sigma\|_{L^3}^2\|\sigma\|^2
\notag\\
&\le C_3\big(1+ \|\nabla \sigma^2\|^2 +\|\nabla\sigma\|^2 \big),
\label{q14}
\end{align}
where $C_3>0$ does not depend on time.
Combining \eqref{energy e1} and \eqref{q14}, we infer that
\begin{align}
& \frac{\mathrm{d}}{\mathrm{d} t}\|\sigma\|_{L^4}^4
+C_4\|\sigma\|_{L^4}^4
+ \|\nabla \sigma^2\|^2
\le C(\chi^2 \|\Delta\varphi\|^2_{L^6}
+  \|\nabla\sigma\|^2 +1),
\label{energy e2}
\end{align}
where $0<C_4< 1/C_3$, and the positive constant $C$ on the right-hand side does not depend on time. Since
$$
\sup_{t\geq 0}\int_t^{t+1}\left(\chi^2 \|\Delta\varphi(s)\|^2_{L^6}
+ \|\nabla\sigma(s)\|^2 +1 \right)\,\mathrm{d}s<+\infty,
$$
thanks to the assumption on $\varphi$ and the estimate \eqref{aux-sig-4}, we can apply Lemma \ref{GronW} and the same argument for
\eqref{aux-sig-1}, \eqref{aux-sig-2} to conclude from \eqref{energy e2} that
\begin{align}
\sup_{t\geq 0} \|\sigma(t)\|_{L^4}^4 +  \sup_{t\geq 0} \int_{t}^{t+1} \|\nabla \sigma^2(s)\|^2 \,\mathrm{d}s
 \leq  C,
\label{diss}
\end{align}
where $C>0$ depends on $\Omega$, $\chi$, $\|\sigma_0\|_{L^4}$, $\sup_{t\geq 0}\int_t^{t+1} \|\varphi(s)\|_{W^{2,6}}^2\,\mathrm{d}s$, but not on time. In addition, by a comparison in \eqref{test2.bw-sig}, we see that
\begin{align}
 \int_t^{t+1} \|\partial_t\sigma(s)\|_{(H^1)'}^2\,\mathrm{d}s
 &  \leq C\int_t^{t+1} \big(\|\bm{v}(s)\|_{\bm{L}^6}^2 \| \sigma(s)\|_{L^3}^2
+\|\nabla  \sigma(s)\|^2 +\|\nabla \varphi(s)\|^2\big) \,\mathrm{d}s \notag\\
&  \leq  C \sup_{t\geq 0}\|\sigma(s)\|_{L^3}^2 \int_0^\infty  \|\bm{v}(s)\|_{\bm{H}^1}^2\,\mathrm{d}s
+ C \int_t^{t+1} \|\nabla \sigma(s)\|^2\,\mathrm{d}s
 \notag
\\
&\quad
+ C \int_t^{t+1} \|\nabla \varphi(s)\|^2\,\mathrm{d}s + C,
\notag
\end{align}
for all $t\geq 0$, which implies $\sigma\in H^1_{\mathrm{uloc}}([0,+\infty);(H^1(\Omega))')$.

Next, testing \eqref{1f2.b-sig} by $\sigma^5$, integrating over $\Omega$, we obtain
\begin{align}
&\frac16 \frac{\mathrm{d}}{\mathrm{d} t}\|\sigma\|_{L^6}^6
+\frac59\|\nabla \sigma^3\|^2  =-\int_{\Omega} \chi\sigma^5 \Delta\varphi\, \mathrm{d}x.
\label{energy e-6}
\end{align}
Noticing that
\be
\|\sigma\|_{L^{18}}^3=\|\sigma^3\|_{L^6}\le C\left(\|\nabla \sigma^3\|+\left|\int_{\Omega} \sigma^3\,\mathrm{d}x\right|\right),
\label{l12-6}
\ee
using \eqref{diss} and Young's inequality, we get
\begin{align*}
	\|\sigma^5\|_{L^\frac{6}{5}}^2 =\|\sigma\|_{L^6}^{10}
	&\le  \|\sigma\|_{L^4}^{\frac{40}{7}} \|\sigma\|_{L^{18}}^{\frac{30}{7}}
 \le  \|\nabla \sigma^3\|^2+C.
\end{align*}
Thus, the term on the right-hand side of \eqref{energy e-6} can be estimated by
\begin{align}
-\int_{\Omega}\chi \sigma^5 \Delta\varphi \,\mathrm{d}x
& \le  |\chi|\| \sigma ^5\|_{L^\frac{6}{5}} \|\Delta\varphi \|_{L^6}
 \le \frac{2}{9}\|\nabla  \sigma^3\|^2 + C\chi^2 \|\Delta\varphi\|^2_{L^6} +C.
\label{l13-6}
\end{align}
It follows from \eqref{energy e-6} and \eqref{l13-6} that
\begin{align}
&\frac16 \frac{\mathrm{d}}{\mathrm{d} t}\|\sigma\|_{L^6}^6
+\frac13\|\nabla\sigma^3\|^2
\le  C \chi^2\|\Delta\varphi\|^2_{L^6}+C.
\label{energy e1-6}
\end{align}
From H\"{o}lder's and Young's inequalities, the Sobolev embedding theorem, the Poincar\'{e}--Wirtinger inequality, \eqref{diss} and \eqref{l12-6}, we find
\begin{align}
\|\sigma\|_{L^6}^6
= \int_\Omega \sigma^6\,\mathrm{d}x
&\le \|\sigma\|_{L^{18}}^3 \|\sigma\|_{L^3}^2\|\sigma\|_{L^6}
\notag\\
&\le \frac12 \|\sigma\|_{L^{18}}^6 +\frac12\|\sigma\|_{L^3}^4\|\sigma\|_{L^6}^2
\notag\\
&\le C_5\big(1+ \|\nabla \sigma^3\|^2 +\|\nabla \sigma\|^2 \big).
\label{q14-6}
\end{align}
Combining \eqref{energy e1-6} and \eqref{q14-6}, we get
\begin{align}
& \frac{\mathrm{d}}{\mathrm{d} t}\|\sigma\|_{L^6}^6
+C_6\|\sigma\|_{L^6}^6 +\|\nabla \sigma^3\|^2
\le  C\big( \chi^2\|\Delta\varphi\|^2_{L^6} +\|\nabla \sigma\|^2 +1\big),
\notag
\end{align}
where $0<C_6<1/C_5$. This together with Lemma \ref{GronW} yields
\begin{align}
\sup_{t\geq 0} \|\sigma(t)\|_{L^6}^6 +  \sup_{t\geq 0} \int_{t}^{t+1} \|\nabla \sigma^3(s)\|^2 \,\mathrm{d}s
 \leq  C,
\notag
\end{align}
where $C>0$ depends on $\Omega$, $\chi$, $\|\sigma_0\|_{L^6}$, $\sup_{t\geq 0}\int_t^{t+1} \|\varphi(s)\|_{W^{2,6}}^2\,\mathrm{d}s$, but not on time.

With the above estimates, we can prove the existence of a global weak solution to problem \eqref{1f1-sig} with the required regularity properties. The details are omitted.
\end{proof}

\subsubsection{Decoupled auxiliary problem for $(\varphi,\mu)$}

Let $\bm{v}, \sigma$ be two given functions. Consider the following initial boundary value problem of a convective Cahn--Hilliard type equation for $(\varphi, \mu)$ subject to homogeneous Neumann boundary conditions:
\begin{subequations}\label{con-CHs}
	\begin{alignat}{3}
	&\partial_t \varphi +\bm{v} \cdot \nabla \varphi=\Delta \mu -\alpha(\overline{\varphi}-c_0),
    \qquad \qquad &\text{in}&\ \Omega \times (0,+\infty),
\label{con-CHs-1} \\
	&\mu=- \Delta \varphi+  \varPsi'(\varphi)-\chi \sigma +\beta\mathcal{N}(\varphi-\overline{\varphi}),&\text{in}&\ \Omega \times (0,+\infty),
\label{con-CHs-2} \\
& {\partial}_{\bm{n}}\varphi={\partial}_{\bm{n}}\mu =0,\qquad\qquad &\textrm{on}& \   \partial\Omega\times(0,+\infty),
\label{con-CHs-bd}\\
& \varphi|_{t=0}=\varphi_0(x), \qquad &\textrm{in}&\ \Omega.
\label{con-CHs-ini}
	\end{alignat}
\end{subequations}

For problem \eqref{con-CHs}, we establish the following result:
\begin{proposition}\label{con-CHs-ws}
Let $\Omega$ be a bounded domain in $\mathbb{R}^3$, with boundary $\partial \Omega$ of class $C^2$. Suppose that hypotheses (H2)--(H3) are satisfied and
\begin{align*}
\bm{v} \in L^2(0, +\infty; \bm{H}_{0, \mathrm{div}}^1(\Omega)),
\quad \sigma\in L^2_{\mathrm{uloc}}([0,+\infty);H^1(\Omega)).
\end{align*}
\begin{itemize}
\item[\emph{(1)}] \emph{Global weak solution}.
For any initial datum $\varphi_{0}$ satisfying
$\varphi_{0} \in H^{1}(\Omega)\cap L^\infty(\Omega)$ with $\|\varphi_0\|_{L^\infty}\leq 1$, $|\overline{\varphi_0}|<1$, problem \eqref{con-CHs} admits a unique global weak solution $(\varphi, \mu)$ such that
$$
\begin{aligned}
 & \varphi\in L^{\infty}(0, +\infty ; H^1(\Omega))\cap L^{4}_{\mathrm{uloc}}([0,+\infty);H^2_{N}(\Omega)) \cap L^{2}_{\mathrm{uloc}}([0,+\infty);W^{2,6}(\Omega)),
 \\
 & \partial_t \varphi \in L^2(0, +\infty ; (H^1(\Omega))'),\quad \varPsi'(\varphi) \in L^{2}_{\mathrm{uloc}}([0, +\infty); L^6(\Omega))
 \\
 & \varphi\in L^{\infty}(\Omega \times(0, +\infty))\ \text { with }\ |\varphi(x, t)|<1\ \text { a.e. in } \Omega \times(0, +\infty),
 \\
 & \mu\in L_{\mathrm{uloc}}^{2}([0, +\infty); H^1(\Omega)),\quad \nabla \mu \in L^2(0, +\infty; \bm{L}^2(\Omega)).
\end{aligned}
$$
For almost all $t\in (0,+\infty)$, the solution $(\varphi, \mu)$ satisfies
%
\begin{align}
&\left \langle \partial_t \varphi,\xi\right \rangle_{(H^1)',H^1}
+({\bm{v} \cdot \nabla \varphi},\xi)=- (\nabla \mu,\nabla \xi)-\alpha(\overline{\varphi}-c_0,\xi),
\label{con-CHs.aw}
\end{align}
for all $\xi \in H^1(\Omega)$, where
\begin{align}
& \mu=-\Delta \varphi+\varPsi'(\varphi)-\chi \sigma+\beta\mathcal{N}(\varphi-\overline{\varphi}), \quad
\text{a.e. in}\ \Omega\times (0,+\infty).
\label{con-CHs.dw}
\end{align}
Moreover, the initial condition \eqref{con-CHs-ini} is fulfilled almost everywhere in $\Omega$.
\item[\emph{(2)}] \emph{Global strong solution}.
Assume, in addition, $\partial \Omega$ is of class $C^3$ and
\begin{align*}
\sigma\in L^\infty(0,+\infty;L^6(\Omega))\cap H^1_{\mathrm{uloc}}([0,+\infty);(H^1(\Omega))')\quad \text{with}\ \ \sigma(0)\in H^1(\Omega).
\end{align*}

For any initial datum $\varphi_{0}$ satisfying
$\varphi_{0} \in H^{2}_N(\Omega)$, $\|\varphi_0\|_{L^\infty}\leq 1$, $|\overline{\varphi_0}|<1$, $\mu_0 =-\Delta \varphi_0+ \varPsi'(\varphi_0)\in H^1(\Omega)$, problem \eqref{con-CHs} admits a unique global strong solution $(\varphi, \mu)$  such that
$$
\begin{aligned}
 & \varphi\in L^{\infty}(0, +\infty; W^{2, 6}(\Omega)\cap H^2_N(\Omega)),\\
 & \partial_t\varphi \in L^2_{\mathrm{uloc}}([0, +\infty); H^1(\Omega)),
 \\
 & \varphi\in L^{\infty}(\Omega \times(0, +\infty))\ \text { with }\ |\varphi(x, t)|<1\ \text { a.e. in } \Omega \times(0, +\infty),
 \\
 & \varPsi'(\varphi) \in L^{\infty}(0, +\infty; L^6(\Omega)),
 \\
 & \mu\in L^{\infty}(0, +\infty; H^1(\Omega)) \cap L_{\mathrm{uloc}}^2([0, +\infty); H^3(\Omega)\cap H^2_N(\Omega)).
\end{aligned}
$$
The solution $(\varphi, \mu)$ satisfies
\begin{subequations}
	\begin{alignat}{3}
	& \partial_t \varphi+\bm{v} \cdot \nabla \varphi=\Delta \mu-\alpha(\overline{\varphi}-c_0),&
    \notag \\
    &  \mu=-\Delta \varphi+\varPsi'(\varphi)-\chi \sigma+\beta\mathcal{N}(\varphi-\overline{\varphi}),&
    \notag
	\end{alignat}
\end{subequations}
almost everywhere in $\Omega \times (0,+\infty)$.
The initial condition \eqref{con-CHs-ini} is satisfied in $\Omega$. Moreover, if $\bm{v}$ also satisfies $\bm{v} \in L^\infty(0, +\infty; \bm{L}_{0, \mathrm{div}}^2(\Omega))$, then $\partial_t\varphi \in L^\infty(0,+\infty;(H^1(\Omega))')$.
\end{itemize}
\end{proposition}

\begin{proof}[\textbf{Proof of Proposition \ref{con-CHs-ws}-(1)}]
The proof of conclusion (1) closely follows the argument for Theorem \ref{vch-weak}-(1). Thus, for the existence part, we only derive necessary \textit{a priori} estimates below. Concerning the uniqueness, we can obtain an analogue of the inequality \eqref{w-diff-vphi} for the difference of two weak solutions (with $\sigma=0$ therein) and reach the conclusion.

In the following we shall work with the \textit{a priori} bound
\begin{align}
    \sup_{t \geq 0} \|\varphi(t)\|_{L^{\infty}} \leq 1,
    \label{con-CHs-Linfty}
\end{align}
which can be guaranteed by (H2) (cf. also \eqref{phi-app1}).

First, taking $\xi=1$ in \eqref{con-CHs.aw}, we have the mass relation
\be
\overline{\varphi}(t)
= c_0+e^{-\alpha t}\big(\overline{\varphi_{0}}-c_0\big),
\quad\forall\,t\geq 0.
\label{con-CHs-aver}
\ee
Due to the assumption $|\overline{\varphi_{0}}|<1$ and (H3), we infer from \eqref{con-CHs-aver} that for all $t\geq0$,
$$
\overline{\varphi}(t)\in [c_0,\overline{\varphi_{0}}]\ \ \text{if}\ \  \overline{\varphi_{0}}>c_0; \quad \overline{\varphi}(t)\in [\overline{\varphi_{0}},c_0]\ \ \text{if}\ \  \overline{\varphi_{0}}<c_0;\quad \overline{\varphi}(t)=c_0\ \ \text{if}\ \ \overline{\varphi_{0}}=c_0.
$$
Hence, the spatial average $\overline{\varphi}(t)$ is strictly separated from $\pm 1$ for all time.

Next, testing \eqref{con-CHs.aw} by $\mu+\chi\sigma$ and using \eqref{con-CHs.dw}, we obtain
\begin{align}
&\frac{\mathrm{d}}{\mathrm{d}t} \left(\frac12 \|\nabla \varphi\|^2
+ \int_\Omega \varPsi(\varphi)\,\mathrm{d}x
+ \frac{\beta}{2} \|\nabla \mathcal{N}(\varphi-\overline{\varphi})\|^2\right)
+ \|\nabla \mu\|^2
\notag \\
&\quad = -\int_\Omega (\bm{v}\cdot\nabla \varphi)(\mu+\chi\sigma)\,\mathrm{d}x
-\chi\int_\Omega \nabla \mu\cdot \nabla \sigma\,\mathrm{d}x -\alpha(\overline{\varphi}-c_0)\int_\Omega (\mu+ \chi \sigma)\,\mathrm{d}x.
\label{con-CHs-es1}
\end{align}
The first two terms on the right-hand side of \eqref{con-CHs-es1} can be estimated as follows
\begin{align}
& -\int_\Omega (\bm{v}\cdot\nabla \varphi)(\mu+\chi\sigma)\,\mathrm{d}x
 -\chi\int_\Omega \nabla \mu\cdot \nabla \sigma\,\mathrm{d}x\notag \\
&\quad  \leq \left|\int_\Omega (\bm{v}\cdot\nabla \mu)\varphi\,\mathrm{d}x\right|
+ \left|\chi\int_\Omega (\bm{v}\cdot\nabla \sigma) \varphi\,\mathrm{d}x\right| + \left|\chi\int_\Omega \nabla \mu\cdot \nabla \sigma\,\mathrm{d}x\right|\notag \\
& \quad \leq \|\varphi\|_{L^\infty}\|\bm{v}\|(\|\nabla \mu\|+ |\chi|\|\nabla \sigma\|) + |\chi|\|\nabla \mu\|\|\nabla \sigma\|\notag \\
&\quad \leq \frac{1}{8}\|\nabla \mu\|^2 + C(\|\bm{v}\|^2+ \chi^2\|\nabla \sigma\|^2).
\label{con-CHs-es2}
\end{align}
Concerning the third term, we can use (H2) and exploit the argument for \cite[(3.20)]{HW2022} (with $\gamma=0$ therein) to obtain
\begin{align}
& -\alpha(\overline{\varphi}-c_0)\int_\Omega \mu\,\mathrm{d}x
\leq \frac12 \|\nabla \mu\|^2+ \frac12 \|\nabla\sigma\|^2 + C (1+\alpha^3) \alpha e^{-\alpha t}|\overline{\varphi_{0}}-c_0|.
\notag
\end{align}
Besides,
\begin{align}
& -\alpha(\overline{\varphi}-c_0)\int_\Omega \chi \sigma\,\mathrm{d}x
\leq  2\alpha|\chi|\|\sigma\|_{L^1}\leq C(\|\sigma\|^2+1).\notag
\end{align}
Testing \eqref{con-CHs-2} by $\varphi-\overline{\varphi}$, and using the convexity of $\varPsi_0$ (see (H2)), (H3), \eqref{con-CHs-Linfty} as well as \eqref{con-CHs-aver}, we find
\begin{align}
& \int_\Omega (\mu-\overline{\mu}) (\varphi-\overline{\varphi})\,\mathrm{d}x +
\chi \int_\Omega (\sigma-\overline{\sigma}) (\varphi-\overline{\varphi})\,\mathrm{d}x \notag \\
&\quad = \|\nabla \varphi\|^2 + \int_\Omega \varPsi_0'(\varphi)(\varphi-\overline{\varphi})\,\mathrm{d}x - \theta_0\|\varphi-\overline{\varphi}\|^2 + \beta\|\nabla \mathcal{N}(\varphi-\overline{\varphi})\|^2\notag \\
&\quad \geq \|\nabla \varphi\|^2 + \int_\Omega \varPsi_0(\varphi)\,\mathrm{d}x - \int_\Omega \varPsi_0(\overline{\varphi})\,\mathrm{d}x
+ \frac{\beta}{2}\|\nabla \mathcal{N}(\varphi-\overline{\varphi})\|^2
-C\notag\\
&\quad \geq \|\nabla \varphi\|^2 + \int_\Omega \varPsi(\varphi)\,\mathrm{d}x + \frac{\beta}{2}\|\nabla \mathcal{N}(\varphi-\overline{\varphi})\|^2-C.\notag
\end{align}
The above inequality together with H\"{o}lder's inequality, the Poincar\'{e}--Wirtinger inequality, and Young's inequality yields
\begin{align}
& \|\nabla \varphi\|^2
+ \int_\Omega \varPsi(\varphi)\,\mathrm{d}x
+ \frac{\beta}{2}\|\nabla \mathcal{N}(\varphi-\overline{\varphi})\|^2
\notag \\
&\quad \leq \|\mu-\overline{\mu}\|\|\varphi-\overline{\varphi}\| +
|\chi| \|\sigma-\overline{\sigma}\|\|\varphi-\overline{\varphi}\| +C
\notag \\
&\quad \leq \frac12\|\nabla \varphi\|^2 + C(1+\chi^2)(\|\nabla \mu\|^2 +\|\nabla \sigma\|^2) + C,
\notag
\end{align}
where $C>0$ depends on $\Omega$, $\overline{\varphi_0}$, and the coefficients of the system. On the other hand, due to (H2) and \eqref{con-CHs-Linfty}, there exists $C_7>0$ depending only on $\Omega$, $\theta_0$ and $\beta$ such that
\begin{align}
\frac12\|\nabla \varphi\|^2
+ \int_\Omega \varPsi(\varphi)\,\mathrm{d}x
+ \frac{\beta}{2}\|\nabla \mathcal{N}(\varphi-\overline{\varphi})\|^2 +C_7\geq 0.
\label{low-bd-E}
\end{align}
Combining the above estimates, we can deduce from
\eqref{con-CHs-es1} that
\begin{align}
&\frac{\mathrm{d}}{\mathrm{d}t} \left(\frac12\|\nabla \varphi\|^2
+ \int_\Omega \varPsi(\varphi)\,\mathrm{d}x
+ \frac{\beta}{2} \|\nabla \mathcal{N}(\varphi-\overline{\varphi})\|^2+C_7\right) \notag \\
&\qquad + \frac14\|\nabla \mu\|^2 + C_{8}\left(\frac12\|\nabla \varphi\|^2
+ \int_\Omega \varPsi(\varphi)\,\mathrm{d}x
+ \frac{\beta}{2} \|\nabla \mathcal{N}(\varphi-\overline{\varphi})\|^2+C_7\right)
\notag \\
&\quad \leq C_9(\|\bm{v}\|^2+ \|\sigma\|_{H^1}^2+1),
\label{con-CHs-es5}
\end{align}
for some positive constants $C_8, C_9$ that are independent of time. Recalling the assumptions on $\bm{v}$, $\sigma$, we can apply Lemma \ref{GronW} to \eqref{con-CHs-es5} and conclude
\begin{align}
& \sup_{t\geq 0} \left(\frac12\|\nabla \varphi(t)\|^2
+ \int_\Omega \varPsi(\varphi(t))\,\mathrm{d}x
+ \frac{\beta}{2} \|\nabla \mathcal{N}(\varphi-\overline{\varphi})(t)\|^2\right) \notag \\
&\quad +  \sup_{t\geq 0} \int_{t}^{t+1} \|\nabla \mu(s)\|^2 \,\mathrm{d}s
 \leq  C,
\label{con-CHs-es6}
\end{align}
where the positive constant $C$ depends on $\Omega$,  the coefficients of the system,   $\max_{r\in[-1,1]}|\varPsi(r)|$, $\overline{\varphi_{0}}$, $\|\varphi_{0}\|_{H^1}$, $\|\bm{v}\|_{L^2(0,+\infty,\bm{L}^2(\Omega))}$, $\sup_{t\geq 0}\int_t^{t+1}\|\sigma(s)\|_{H^1(\Omega)}^2\,\mathrm{d}s$, but it is independent of the time.

With the aid of \eqref{con-CHs-es6}, the rest estimates can be obtained similarly to Section \ref{uni-es-appW}. We omit the details here.
\end{proof}

The proof of Proposition \ref{con-CHs-ws}-(2) is based on a suitable regularization of the original problem \eqref{con-CHs}, cf. \cite{GMT,Gio2022,H2,HW2022}.
\smallskip

\textbf{The regularized problem}. We first approximate the initial data. For any integer $k\geq 1$, let $h_k$ be a globally Lipschitz continuous function satisfying
\begin{equation*}
h_k(r)=\begin{cases}
k,\quad\ \ \ \, r>k,\\
r,\qquad   r\in[-k, k],\\
-k,\quad \, r<-k.
\end{cases}
\end{equation*}
Define
$$
\widetilde{\mu}_0=-\Delta \varphi_0+ \varPsi_0'(\varphi_0)=\mu_0+\theta_0\varphi_0.
$$
We take the cut-off $\widetilde{\mu}_{0,k}=h_k\circ \widetilde{\mu}_{0}$.
Since $\mu_0\in H^1(\Omega)$ and $\varphi_0\in H^2_N(\Omega)$, it follows from the superposition principle that
$$
\widetilde{\mu}_0,\ \widetilde{\mu}_{0,k} \in H^1(\Omega),\quad  \|\widetilde{\mu}_{0,k}\|_{H^1}\leq \|\widetilde{\mu}_{0}\|_{H^1}\quad \text{and}\quad  \|\widetilde{\mu}_{0,k}-\widetilde{\mu}_{0}\|\to 0\quad \text{as}\quad k\to+\infty.
$$
Given $\widetilde{\mu}_{0,k}\in H^1(\Omega)$, let us consider the nonlinear elliptic problem
$$
\begin{cases}
-\Delta \varphi_{0,k} +\varPsi_0'(\varphi_{0,k})=\widetilde{\mu}_{0,k}, \quad \text{in}\ \Omega, \\
\partial_{\bm{n}}\varphi_{0,k} =0,\qquad \qquad \qquad \quad \ \ \text{on}\ \partial \Omega.
\end{cases}
$$
Thanks to \cite[Lemma A.1]{GMT}, it admits a unique solution $\varphi_{0,k}\in H^2_N(\Omega)$ with $\varPsi_0'(\varphi_{0,k})\in L^2(\Omega)$. Furthermore, we have
$$
\|\varphi_{0,k}\|_{H^2}+\|\varPsi_0'(\varphi_{0,k})\|\leq C(1+\|\widetilde{\mu}_0\|)\quad \text{and}\quad  \|\varphi_{0,k}-\varphi_0\|_{H^1}\to 0\quad\text{as}\quad k\to +\infty.
$$
As a consequence, there exists a sufficiently large integer $\widehat{k}$ such that
$$
\|\varphi_{0,k}\|_{H^1}\leq 1+\|\varphi_{0}\|_{H^1} \quad
\text{and}\quad |\overline{\varphi_{0,k}}|\leq \frac{1+|\overline{\varphi_{0}}|}{2}<1,
\quad \forall\, k\geq \widehat{k}.
$$
Besides, we infer from \cite[Lemma A.1]{CG} that
$\|\varPsi_0'(\varphi_{0,k})\|_{L^\infty}\leq \|\widetilde{\mu}_{0,k}\|_{L^\infty}\leq k$,
which together with the assumption (H2) implies the separation property
$$
\|\varphi_{0,k}\|_{C(\overline{\Omega})}\leq 1-\delta_k,
$$
for some constant  $\delta_k\in (0,1)$ depending on $k$. This yields $\varPsi_0'(\varphi_{0,k})\in H^1(\Omega)$. Using the standard elliptic estimate, we further obtain $\varphi_{0,k}\in H^3(\Omega)$.

Next, we approximate the given velocity field $\bm{v}$ and the source term $\sigma$. Like in Section \ref{sec:CHS-app-pro},
we take a convergent sequence $\big\{\boldsymbol{v}^k\big\}\subset C^\infty_0((0, +\infty);D(\bm{S}))$ satisfying \eqref{vk-con}.
Concerning $\sigma$, for any $k\in \mathbb{Z}^+$, we denote by $\sigma^k$ the (unique) solution of the following linear elliptic problem
$$
\begin{cases}
\dfrac{1}{k} (-\Delta \sigma^k + \sigma^k) + \sigma^k =\sigma,\quad\ \ \text{in}\ \Omega, \\
\partial_{\bm{n}}\sigma^k =0,\qquad\qquad\qquad\qquad   \text{on}\ \partial \Omega.
\end{cases}
$$
The operator $-\Delta + I$ subject to the homogeneous Neumann boundary condition is an canonical
isomorphism from $H^1(\Omega)\to (H^1(\Omega))'$ due to the Riesz theorem. In view of the assumption on $\sigma$, we infer from the elliptic regularity theory that
\begin{align*}
\sigma^k \in L^2_{\mathrm{uloc}}([0,+\infty);H^3(\Omega)\cap H^2_N(\Omega))\cap L^\infty(0,+\infty;W^{2,6}(\Omega))\cap H^1_{\mathrm{uloc}}([0,+\infty);H^1(\Omega)).
\end{align*}
By interpolation, we also find $\sigma^k\in BUC([0,+\infty);H^2(\Omega))$, while from the assumption $\sigma(0)\in H^1(\Omega)$, we obtain $\sigma^k(0)\in H^3(\Omega)\cap H^2_N(\Omega)$. Applying Lemma \ref{sig-conv} with $V=H^1(\Omega)$, $H=L^2(\Omega)$, $\varepsilon^2=1/k$, $J=-\Delta + I$ for $u=\sigma$ and then $u=\partial_t \sigma$, we can conclude that
\begin{align*}
& \|\sigma^k\|_{L^2_{\mathrm{uloc}}([0,+\infty);H^1(\Omega))}\leq \|\sigma\|_{L^2_{\mathrm{uloc}}([0,+\infty);H^1(\Omega))},\\
&\|\partial_t\sigma^k\|_{L^2_{\mathrm{uloc}}([0,+\infty);(H^1(\Omega))')}\leq \|\partial_t\sigma\|_{L^2_{\mathrm{uloc}}([0,+\infty);(H^1(\Omega))')},\\
& \|\sigma^k(0)\|_{H^1}\leq \|\sigma(0)\|_{H^1},
\end{align*}
moreover, as $k\to +\infty$, it holds
\begin{align*}
 &\sigma^k\to \sigma \quad \text{in}\ \ L^2_{\mathrm{uloc}}([0,+\infty);H^1(\Omega)),\\
 &\partial_t\sigma^k\to \partial_t\sigma \quad \text{in}\ \ L^2_{\mathrm{uloc}}([0,+\infty);(H^1(\Omega))'),\\
 & \sigma^k(0)\to \sigma(0)\quad \text{in}\ \ H^1(\Omega).
\end{align*}
Testing the equation of $\sigma^k$ by $(\sigma^k)^q$, $q=1,3,5$, integrating over $\Omega$, using H\"{o}lder's inequality and Young's inequality, we obtain
\begin{align*}
\underbrace{\frac{1}{k} \int_\Omega \left(q(\sigma^k)^{q-1}|\nabla \sigma^k|^2+ (\sigma^k)^{q+1}\right)\,\mathrm{d}x}_{\geq 0}
+ \|\sigma^k\|_{L^{q+1}}^{q+1}
 & = \int_\Omega \sigma  (\sigma^k)^q\,\mathrm{d}x
 \\
 & \leq  \|\sigma\|_{L^{q+1}}\|\sigma^k\|_{L^{q+1}}^q \\
 & \leq  \frac{1}{q+1}\|\sigma\|_{L^{q+1}}^{q+1} +\frac{q}{q+1}\|\sigma^k\|_{L^{q+1}}^{q+1},
\end{align*}
which implies the uniform bounds
$$
\|\sigma^k\|_{L^\infty(0,+\infty;L^{q+1}(\Omega))}\leq \|\sigma\|_{L^\infty(0,+\infty;L^{q+1}(\Omega))},\quad q=1,3,5.
$$

For any given $\gamma\in(0,1)$ and integer $k\geq \widehat{k}$ (thus the approximated initial data are uniformly bounded with respect to $k$), we consider the following regularized problem
\begin{subequations}\label{rf2}
	\begin{alignat}{3}
	&\partial_t \varphi^{\gamma,k} +\bm{v}^k \cdot \nabla \varphi^{\gamma,k}=\Delta \mu^{\gamma,k} -\alpha(\overline{\varphi^{\gamma,k}}-c_0), &\qquad \textrm{in}& \ \Omega \times (0,+\infty),
\label{f2.c} \\
	&\mu^{\gamma,k}=- \Delta \varphi^{\gamma,k}+  \varPsi'(\varphi^{\gamma,k})-\chi \sigma^{k} \notag \\
    & \qquad \quad +\beta\mathcal{N}(\varphi^{\gamma,k}-\overline{\varphi^{\gamma,k}}) +\gamma\partial_t\varphi^{\gamma,k},
    &\qquad \textrm{in}& \ \Omega \times (0,+\infty),
\label{f2.d} \\
	& {\partial}_{\bm{n}}\varphi^{\gamma,k}={\partial}_{\bm{n}}\mu^{\gamma,k}=0,\qquad\qquad &\qquad \textrm{on}& \   \partial\Omega\times(0,+\infty),
\label{boundary123}\\
& \varphi^{\gamma,k}|_{t=0}=\varphi_{0,k}(x), \qquad &\qquad \textrm{in}&\ \Omega.
\label{ini0222}
	\end{alignat}
\end{subequations}
By extending the previous results \cite[Theorem A.1]{Gio2022} (for the single convective viscous Cahn--Hilliard equation) and \cite[Theorem 2.2]{H2} (for the coupled viscous Cahn--Hilliard--diffusion system without convection), we can establish the existence and uniqueness of a global strong solution $(\varphi^{\gamma,k}, \mu^{\gamma,k})$ to problem \eqref{rf2} on $[0,T]$ for any given $T\geq 1$, that is,

\bp
Assume that $\Omega$ is a bounded domain in $\mathbb{R}^3$ with boundary $\partial \Omega$ of class $C^3$, (H2)--(H3) are satisfied, $\gamma\in(0,1)$ and $k\in \mathbb{Z}^+$ is sufficiently large.
For any $T\geq 1$, the regularized problem \eqref{rf2}
admits a unique strong solution $(\varphi^{\gamma,k}, \mu^{\gamma,k})$ on $[0,T]$ such that
\begin{align*}
&\varphi^{\gamma,k} \in  L^\infty(0,T;H^3(\Omega)),
\\
&\partial_t\varphi^{\gamma,k}\in L^{\infty}(0,T;H^1(\Omega))\cap L^{2}(0,T;H^2(\Omega))\cap H^1(0,T;L^2(\Omega)),
\\
&\mu^{\gamma,k} \in   L^{\infty}(0,T;H^2(\Omega))\cap  H^1(0,T;L^2(\Omega)).
\end{align*}
The pair $(\varphi^{\gamma,k}, \mu^{\gamma,k})$ satisfies the equation \eqref{f2.c}--\eqref{f2.d} almost everywhere in $\Omega\times(0,T)$, the boundary condition \eqref{boundary123} almost everywhere on $\partial\Omega\times(0,T)$, and the initial condition \eqref{ini0222} in $\Omega$. Moreover, there exists some constant $\delta_{\gamma,k}\in (0,\delta_k)$ depending on $\gamma$ such that
\be
\|\varphi^{\gamma,k}(t)\|_{C(\overline{\Omega})} \leq 1-\delta_{\gamma,k}, \quad \forall\,t\in[0,T].
\label{vark-sep}
\ee
\ep

After the above preparations, we are now in a position to prove Proposition \ref{con-CHs-ws}-(2).

\begin{proof}[\textbf{Proof of Proposition \ref{con-CHs-ws}-(2)}] The uniqueness of strong solutions is a direct consequence of the uniqueness result for weak solutions established in Proposition \ref{con-CHs-ws}-(1). To establish the existence, we derive estimates for $(\varphi^{\gamma,k}, \mu^{\gamma,k})$ that are independent of the approximating parameters $\gamma$, $k$ and time.
\medskip

\textbf{First estimate}.
Analogously to \eqref{con-CHs-aver}, we have the mass relation
\be
\overline{\varphi^{\gamma,k}}(t)
= c_0 + e^{-\alpha t}\big(\overline{\varphi^k}(0)-c_0\big)
= c_0+e^{-\alpha t}\big(\overline{\varphi_{0,k}}-c_0\big),
\quad\forall\,t\in [0,T].
\label{aver}
\ee

\textbf{Second estimate}.
A similar argument for \eqref{con-CHs-es1} yields the following equality
\begin{align}
&\frac{\mathrm{d}}{\mathrm{d}t} \left(\frac12 \|\nabla \varphi^{\gamma,k}\|^2 + \int_\Omega \varPsi(\varphi^{\gamma,k})\,\mathrm{d}x
+ \frac{\beta}{2} \|\nabla \mathcal{N}(\varphi^{\gamma,k}-\overline{\varphi^{\gamma,k}})\|^2\right)
+ \|\nabla \mu^{\gamma,k}\|^2
+ \gamma\|\partial_t\varphi^{\gamma,k}\|^2
\notag \\
&\quad = -\int_\Omega (\bm{v}^k\cdot\nabla \varphi^{\gamma,k})(\mu^{\gamma,k}+\chi\sigma^k)\,\mathrm{d}x
-\chi\int_\Omega \nabla \mu^{\gamma,k}\cdot \nabla \sigma^k\,\mathrm{d}x
\notag\\
&\qquad -\alpha(\overline{\varphi^{\gamma,k}}-c_0)\int_\Omega (\mu^{\gamma,k}+ \chi \sigma^k)\,\mathrm{d}x,
\label{con-CHs-es7}
\end{align}
for almost all $t\in (0,T)$. We can estimate the right-hand side of \eqref{con-CHs-es7} like before, keeping in mind possible modifications due to the additional viscous term $\gamma \partial_t\varphi^{\gamma, k}$ in $\mu^{\gamma, k}$. The first two terms can be handled as in \eqref{con-CHs-es2}. For the last term,
recalling the estimates \cite[(3.15), (3.16)]{HW2022} and using the bound $\|\varphi^{\gamma, k}\|_{L^\infty}\leq 1$ (cf. \eqref{vark-sep}), we infer that
\begin{align}
\left|\int_\Omega \mu^{\gamma,k}\,\mathrm{d}x\right|
& \le  C\big(\|\nabla \mu^{\gamma,k}\| + \|\sigma^{k}- \overline{\sigma^{k}}\|+ \|\varphi^{\gamma,k}- \overline{\varphi^{\gamma,k}}\|+\gamma \|\partial_t\varphi^{\gamma,k}\| \big)\| \varphi^{\gamma,k}-\overline{\varphi^{\gamma,k}}\|+C
\notag\\
& \le  C\big(\|\nabla \mu^{\gamma,k}\|+\|\nabla \sigma^k\| + \gamma \|\partial_t\varphi^{\gamma,k}\|+ 1 \big).
\label{q2c}
\end{align}
Then from \eqref{con-CHs-es7} we can deduce that
\begin{align}
&\frac{\mathrm{d}}{\mathrm{d}t} \left(\frac12\|\nabla \varphi^{\gamma,k}\|^2
+ \int_\Omega \varPsi(\varphi^{\gamma,k})\,\mathrm{d}x
+ \frac{\beta}{2} \|\nabla \mathcal{N}(\varphi^{\gamma,k}-\overline{\varphi^{\gamma,k}})\|^2+C_7\right) + \frac14\|\nabla \mu^{\gamma,k}\|^2 + \frac{\gamma}{4}\|\partial_t\varphi ^{\gamma,k}\|^2 \notag \\
&\qquad+ C_{10}\left(\frac12\|\nabla \varphi^{\gamma,k}\|^2
+ \int_\Omega \varPsi(\varphi^{\gamma,k}) \,\mathrm{d}x
+ \frac{\beta}{2} \|\nabla \mathcal{N}(\varphi^{\gamma,k} -\overline{\varphi^{\gamma,k}})\|^2+C_7\right)
\notag \\
&\quad \leq C_{11}(\|\bm{v}^k\|^2+ \| \sigma^k\|_{H^1}^2+1),
\label{con-CHs-es8}
\end{align}
for some positive constants $C_{10}, C_{11}$ that are independent of $\gamma$, $k$ and time. Here, the constant $C_7>0$ is the same as in \eqref{low-bd-E}. Based on the construction of approximate data $\bm{v}^k$, $\sigma^k$, and $\varphi_{0,k}$, the differential inequality \eqref{con-CHs-es8} combined with Lemma \ref{GronW} yields the following estimate:
\begin{align}
& \sup_{t\in [0,T]} \left(\frac12\|\nabla \varphi^{\gamma,k}(t)\|^2
+ \int_\Omega \varPsi(\varphi^{\gamma,k}(t))\,\mathrm{d}x
+ \frac{\beta}{2} \|\nabla \mathcal{N}(\varphi^{\gamma,k}-\overline{\varphi^{\gamma,k}})(t)\|^2\right) \notag \\
&\quad +  \sup_{t\in [0,T-1]} \int_{t}^{t+1} \|\nabla \mu^{\gamma,k}(s)\|^2 \,\mathrm{d}s +   \sup_{t\in[0,T-1]} \int_{t}^{t+1} \gamma \|\partial_t\varphi^{\gamma,k}(s)\|^2 \,\mathrm{d}s
 \leq  C,
\label{low-es2-app1b}
\end{align}
where the positive constant $C$ depends on $\Omega$, the coefficients of the system, $\max_{r\in[-1,1]}|\varPsi(r)|$, $\overline{\varphi_{0}}$, $\|\varphi_{0}\|_{H^1}$, $\|\bm{v}\|_{L^2(0,+\infty,\bm{L}^2(\Omega))}$, $\sup_{t\geq 0}\int_t^{t+1}\|\sigma(s)\|_{H^1(\Omega)}^2\,\mathrm{d}s$, but it is independent of $\gamma$, $k$ and $T$.
\medskip

\textbf{Third estimate}.
Consider the following elliptic problem with singular nonlinearity:
\begin{equation}
\notag
\begin{cases}
-\Delta \varphi^{\gamma,k}+\varPsi_0'(\varphi^{\gamma,k})=\mu^{\gamma,k} +\theta_0\varphi^{\gamma,k} +\chi\sigma^{k} -\gamma\partial_t\varphi^{\gamma,k},\quad &\text{ in }\Omega,
\\
\partial_{\bm{n}} \varphi^{\gamma,k}=0,
\quad &\text{ on }\partial \Omega.
\end{cases}
\end{equation}
According to the argument in \cite[Section 4]{GGM2017} (see also \cite{A2009,CG}), it holds
\begin{align}
 & \|\varphi^{\gamma,k}\|_{W^{2,q}} +\|\varPsi_0'(\varphi^{\gamma,k})\|_{L^q}
 \notag \\
 &\quad  \leq  C\big(\|\mu^{\gamma,k}\|_{L^q}
 +\|\varphi^{\gamma,k}\|_{L^q} +\|\sigma^{k}\|_{L^q} +\gamma\|\partial_t\varphi^{\gamma,k}\|_{L^q}\big),
 \quad \forall\, q\in [2,6],
\label{q12}
\end{align}
where $C>0$ depends on $\Omega$ and $q$.
From \eqref{vark-sep}, \eqref{q2c}, \eqref{low-es2-app1b}, \eqref{q12}, the Sobolev embedding theorem and the Poincar\'{e}--Wirtinger inequality, we obtain
 \begin{align}
 & \|\varphi^{\gamma,k}(t)\|_{W^{2,q}} +\|\varPsi_0'(\varphi^{\gamma,k}(t))\|_{L^q}
 \notag \\
 &\quad  \leq C\left(\|\nabla\mu^{\gamma,k}(t)\|+ \left|\int_\Omega \mu^{\gamma,k}(t)\,\mathrm{d}x\right|
 +\| \sigma^{k}(t) \|_{L^q} +\gamma\|\partial_t\varphi^{\gamma,k}(t)\|_{L^q}+1\right)
 \notag\\
 &\quad  \leq C\big(\|\nabla\mu^{\gamma,k}(t)\|
 +\|\sigma^{k}(t)\|_{L^q} +\gamma\|\partial_t\varphi^{\gamma,k}(t)\|_{L^q}+1\big),
 \quad \forall\, q\in [2,6],
\label{q13}
\end{align}
for almost all $t\in [0,T-1]$.
The constant $C>0$ is independent of $\gamma$, $k$ and time. Taking $q=2$, from \eqref{low-es2-app1b} and \eqref{q13}, we find
\begin{align*}
   \sup_{t\in[0,T-1]}\int_{t}^{t+1} \left( \|\varphi^{\gamma,k}(s)\|_{H^2}^2  + \|\varPsi_0'(\varphi^{\gamma,k}(s))\|^2\right)\,\mathrm{d}s \leq C.
\end{align*}

\textbf{Fourth estimate}.
Testing \eqref{f2.c} by $\partial_{t} \mu^{\gamma,k}$ and integrating over $\Omega$, after integration by parts, we get
\begin{align*}
\frac{1}{2}\frac{\mathrm{d}}{\mathrm{d}t} \|\nabla\mu^{\gamma,k}\|^2
+ \int_\Omega \partial_t\varphi^{\gamma,k}\partial_{t} \mu^{\gamma,k}\,\mathrm{d}x
+ \int_\Omega (\bm{v}^k\cdot\nabla \varphi^{\gamma,k})\partial_{t} \mu^{\gamma,k}\,\mathrm{d}x
+\alpha(\overline{\varphi^{\gamma,k}}-c_0) \int_\Omega \partial_{t} \mu^{\gamma,k}\,\mathrm{d}x=0.
\end{align*}
By the definition of $\mu^{\gamma,k}$, a direct computation yields
\begin{align*}
\int_\Omega \partial_t\varphi^{\gamma,k}\partial_{t} \mu^{\gamma,k}\,\mathrm{d}x
& = \frac{\gamma}{2}\frac{\mathrm{d}}{\mathrm{d}t}\| \partial_{t} \varphi^{\gamma,k}\|^2 +
\|\nabla\partial_t\varphi^{\gamma,k}\|^2 +
\int_\Omega \varPsi_0''(\varphi^{\gamma,k})|\partial_t\varphi^{\gamma,k}|^2\,\mathrm{d}x
-\theta_0\|\partial_t\varphi^{\gamma,k}\|^2
\\
&\quad
-\chi\int_\Omega \partial_{t} \varphi^{\gamma,k}\partial_t\sigma^{k}  \,\mathrm{d}x
+\beta \int_\Omega \partial_{t} \varphi^{\gamma,k} \mathcal{N}\big(\partial_{t} \varphi^{\gamma,k} -\overline{\partial_{t} \varphi^{\gamma,k}}\big) \,\mathrm{d}x.
\end{align*}
Next, from the construction of $\bm{v}^k$, $\sigma^k$ and $(\varphi^{\gamma,k},\mu^{\gamma,k})$, using integration by parts, we observe that
\begin{align}
&\int_{\Omega} (\boldsymbol{v}^k \cdot \nabla \varphi^{\gamma,k}) \partial_t \mu^{\gamma,k} \,\mathrm{d} x\notag\\
&\quad =  \int_{\Omega} (\boldsymbol{v}^k \cdot \nabla \varphi^{\gamma,k}) \big(\gamma \partial_t^2 \varphi^{\gamma,k}-\Delta \partial_t \varphi^{\gamma,k}+\varPsi_0''(\varphi^{\gamma,k}) \partial_t \varphi^{\gamma,k}-\theta_0 \partial_t \varphi^{\gamma,k}-\chi\partial_t\sigma^{k}\big) \, \mathrm{d} x
\notag\\
&\qquad + \beta \int_{\Omega} (\boldsymbol{v}^k \cdot \nabla \varphi^{\gamma,k}) \mathcal{N}\big(\partial_{t} \varphi^{\gamma,k} -\overline{\partial_{t} \varphi^{\gamma,k}}\big)\, \mathrm{d} x
\notag\\
&\quad =  \frac{\mathrm{d}}{\mathrm{d} t}\left(\gamma \int_{\Omega} (\boldsymbol{v}^{k} \cdot \nabla \varphi^{\gamma,k}) \partial_t \varphi^{\gamma,k}\, \mathrm{d} x\right)
+\gamma \int_{\Omega} (\partial_t \boldsymbol{v}^k \cdot \nabla \partial_t \varphi^{\gamma,k}) \varphi^{\gamma,k}  \,\mathrm{d} x
\notag\\
&\qquad +\int_{\Omega}\big(\nabla^T \boldsymbol{v}^k \nabla \varphi^{\gamma,k}\big) \cdot \nabla \partial_t \varphi^{\gamma,k}  \,\mathrm{d} x
 +\int_{\Omega}\big(\nabla^2 \varphi^{\gamma,k} \boldsymbol{v}^k\big) \cdot \nabla \partial_t \varphi^{\gamma,k} \, \mathrm{d} x
 \notag\\
 &\qquad -\int_{\Omega} (\boldsymbol{v}^k \cdot \nabla \partial_t \varphi^{\gamma,k}) \varPsi_0'(\varphi^{\gamma,k}) \, \mathrm{d} x
 +\theta_0 \int_{\Omega} (\boldsymbol{v}^k \cdot \nabla \partial_t \varphi^{\gamma,k}) \varphi^{\gamma,k} \, \mathrm{d} x
 \notag\\
 &\qquad -\chi \int_{\Omega} (\boldsymbol{v}^k \cdot  \nabla \varphi^{\gamma,k}) \partial_t\sigma^{k} \, \mathrm{d} x
 - \beta \int_{\Omega} \big[\boldsymbol{v}^{k}\cdot \nabla \mathcal{N}\big(\partial_{t} \varphi^{\gamma,k}
 -\overline{\partial_{t} \varphi^{\gamma,k}}\big)\big] \varphi^{\gamma,k} \,\mathrm{d}x.
 \notag
\end{align}
Finally, using \eqref{aver}, we can get
\begin{align*}
\alpha\big(\overline{\varphi^{\gamma,k}}-c_0\big) \int_\Omega \partial_{t} \mu^{\gamma,k}\,\mathrm{d}x
&=   \frac{\mathrm{d}}{\mathrm{d}t} \left(\alpha\big(\overline{\varphi^{\gamma,k}}-c_0\big)  \int_\Omega \mu^{\gamma,k}\,\mathrm{d}x \right)
- \alpha \left(\frac{\mathrm{d}}{\mathrm{d}t}\overline{\varphi^{\gamma,k}} \right) \int_\Omega \mu^{\gamma,k}\,\mathrm{d}x
\notag\\
&=  \frac{\mathrm{d}}{\mathrm{d}t} \left(\alpha\big(\overline{\varphi^{\gamma,k}}-c_0\big)  \int_\Omega \mu^{\gamma,k}\,\mathrm{d}x \right)
 + \alpha^2 \big(\overline{\varphi^{\gamma,k}}-c_0\big) \int_\Omega \mu^{\gamma,k}\,\mathrm{d}x.
\end{align*}
Combining the above identities, we arrive at the following differential equality:
\begin{align}
&\frac{\mathrm{d}}{\mathrm{d}t} \left(\frac{1}{2}\|\nabla\mu^{\gamma,k}\|^2
+\frac{\gamma}{2}\|\partial_t \varphi^{\gamma,k}\|^2
+\gamma \int_{\Omega} (\boldsymbol{v}^k \cdot \nabla \varphi^{\gamma,k}) \partial_t \varphi^{\gamma,k}  \,\mathrm{d} x
+ \alpha\big(\overline{\varphi^{\gamma,k}}-c_0\big) \int_\Omega   \mu^{\gamma,k}\,\mathrm{d}x
\right)
\notag\\
&\qquad +\|\nabla\partial_t\varphi^{\gamma,k}\|^2 +\int_\Omega \varPsi_0''(\varphi^{\gamma,k})|\partial_t\varphi^{\gamma,k}|^2\,\mathrm{d}x
\notag \\
&\quad = \theta_0\|\partial_t\varphi^{\gamma,k}\|^2
  -\gamma \int_{\Omega} (\partial_t \boldsymbol{v}^k \cdot \nabla \partial_t \varphi^{\gamma,k}) \varphi^{\gamma,k}  \,\mathrm{d} x
-\int_{\Omega}\big(\nabla^T \boldsymbol{v}^k \nabla \varphi^{\gamma,k}\big) \cdot \nabla \partial_t \varphi^{\gamma,k}  \,\mathrm{d} x
\notag\\
&\qquad -\int_{\Omega}\big(\nabla^2 \varphi^{\gamma,k} \boldsymbol{v}^k\big) \cdot \nabla \partial_t \varphi^{\gamma,k} \, \mathrm{d} x
+\int_{\Omega} (\boldsymbol{v}^{k} \cdot \nabla \partial_t \varphi^{\gamma,k}) \varPsi'_0(\varphi^{\gamma,k}) \, \mathrm{d} x
\notag\\
&\qquad
-\theta_0 \int_{\Omega} (\boldsymbol{v}^k \cdot \nabla \partial_t \varphi^{\gamma,k}) \varphi^{\gamma,k} \, \mathrm{d} x
+\chi\int_\Omega \partial_{t} \varphi^{\gamma,k}\partial_t\sigma^{k}  \,\mathrm{d}x
\notag\\
&\qquad
-\beta \int_\Omega \partial_{t} \varphi^{\gamma,k}\mathcal{N}\big(\partial_{t} \varphi^{\gamma,k} -\overline{\partial_{t} \varphi^{\gamma,k}}\big) \,\mathrm{d}x
-\alpha^2 \big(\overline{\varphi^{\gamma,k}}(t)-c_0\big) \int_\Omega \mu^{\gamma,k}\,\mathrm{d}x
\notag\\
&\qquad + \chi \int_{\Omega} (\boldsymbol{v}^k \cdot \nabla \varphi^{\gamma,k}) \partial_t \sigma^{k} \, \mathrm{d} x
+ \beta \int_{\Omega}  \big[\boldsymbol{v}^k \cdot \nabla \mathcal{N}\big(\partial_{t} \varphi^{\gamma,k} -\overline{\partial_{t} \varphi^{\gamma,k}}\big) \big] \varphi^{\gamma,k} \,\mathrm{d}x
\notag\\
&\quad =: \sum_{j=1}^{11} I_j.
\label{high-mu}
\end{align}
The first six terms $I_1,\cdots, I_6$ on the right-hand side of \eqref{high-mu} can be estimated as in \cite[Section 2]{A2022} with suitable modifications. To this end, using \eqref{aver}, \eqref{low-es2-app1b} and the Poincar\'{e}--Wirtinger inequality, we infer that
\begin{align*}
I_1 &\leq 2\theta_0\big(\|\partial_t\varphi^{\gamma,k}
- \overline{\partial_t\varphi^{\gamma,k}}\|^2
+ |\Omega| |\overline{\partial_t\varphi^{\gamma,k}}|^2\big)
\\
&\leq C\|\partial_t\varphi^{\gamma,k}
- \overline{\partial_t\varphi^{\gamma,k}}\|_{H^1} \|\partial_t\varphi^{\gamma,k}
- \overline{\partial_t\varphi^{\gamma,k}}\|_{(H^1)'}
+ C|\overline{\partial_t\varphi^{\gamma,k}}|^2
\\
&\leq C\|\nabla \partial_t\varphi^{\gamma,k}\|\big(
\|\nabla \mu^{\gamma,k}\|+ \|\bm{v}^k  \varphi^{\gamma,k}\|\big)
+ C|\overline{\partial_t\varphi^{\gamma,k}}|^2
\\
&\leq  \frac{1}{18} \|\nabla \partial_t\varphi^{\gamma,k}\|^2
+ C\|\nabla \mu^{\gamma,k}\|^2+C\|\bm{v}^k\|^2\|  \varphi^{\gamma,k}\|_{L^\infty}^2
+ C|\overline{\partial_t\varphi^{\gamma,k}}|^2
\\
&\leq  \frac{1}{18} \|\nabla \partial_t\varphi^{\gamma,k}\|^2
+ C \|\nabla \mu^{\gamma,k}\|^2
+ C\alpha^2|\overline{\varphi^{\gamma,k}}-c_0|^2
+ C\|\bm{v}^k\|^2.
\end{align*}
Using the $L^\infty$-bound $\|\varphi^{\gamma,k}(t)\|_{L^\infty(0,T;L^\infty(\Omega))}\leq 1$ again, we have
\begin{align*}
I_2
\leq \gamma\|\partial_t \boldsymbol{v}^k\|\|\nabla \partial_t \varphi^{\gamma,k}\|\|\varphi^{\gamma,k}\|_{L^{\infty}}
\leq \frac{1}{18}\|\nabla \partial_t \varphi^{\gamma,k}\|^2
 +\gamma^2 C\|\partial_t \boldsymbol{v}^{k}\|^2,
\end{align*}
and
\begin{align*}
I_{6}
&\leq \theta_0 \|\boldsymbol{v}^k\| \|\nabla \partial_t \varphi^{\gamma,k} \|\|\varphi^{\gamma,k}\|_{L^{\infty}}
\leq \frac{1}{18} \|\nabla \partial_t \varphi^{\gamma,k} \|^2 +C\|\boldsymbol{v}^k\|^2.
\end{align*}
Then exploiting the Gagliardo--Nirenberg inequality and the Sobolev embedding theorem, we infer from \eqref{aver}, \eqref{low-es2-app1b}, \eqref{q12} and \eqref{q13} that
\begin{align*}
I_3 & \leq\|\nabla \boldsymbol{v}^k\|\|\nabla \varphi^{\gamma,k}\|_{L^{\infty}}\|\nabla \partial_t \varphi^{\gamma,k}\|
\\
& \leq \frac{1}{36}\|\nabla \partial_t \varphi^{\gamma,k}\|^2
+C\|\nabla \boldsymbol{v}^k\|^2\|\varphi^{\gamma,k}\|_{W^{2,4}}^2
\\
& \leq \frac{1}{36}\|\nabla \partial_t \varphi^{\gamma,k}\|^2
+ C\|\nabla \boldsymbol{v}^k\|^2\big(1+\|\nabla\mu^{\gamma,k}\|^2 +\| \sigma^{k} \|_{L^4}^2
+\gamma^2\|\partial_t\varphi^{\gamma,k}\|^2_{L^4}\big)
\\
& \leq \frac{1}{36}\|\nabla \partial_t \varphi^{\gamma,k}\|^2
+\gamma^2 C\|\nabla \boldsymbol{v}^k\|^2\|\partial_t \varphi^{\gamma,k}\|^{\frac{1}{2}}\big(\|\nabla \partial_t \varphi^{\gamma,k}\|
+|\overline{\partial_t \varphi^{\gamma,k}}|\big)^{\frac{3}{2}}
\\
& \quad+C\|\nabla \boldsymbol{v}^k\|^2\big(1
+\|\nabla \mu^{\gamma,k}\|^2+\| \sigma^{k} \|_{L^6}^2\big)
\\
&\leq  \frac{1}{18}\|\nabla \partial_t \varphi^{\gamma,k}\|^2
+\gamma^8 C\|\nabla \boldsymbol{v}^k\|^8\|\partial_t \varphi^{\gamma,k}\|^2
+C\|\nabla \boldsymbol{v}^k\|^2 \|\nabla \mu^{\gamma,k}\|^2
\\
&\quad + C\|\nabla \boldsymbol{v}^k\|^2\big(1
+ \| \sigma^{k} \|_{L^6}^2\big) + C\alpha^2|\overline{\varphi^{\gamma,k}}-c_0|^2,
\end{align*}
and
\begin{align*}
I_4+I_5 & \leq \big( \|\varphi^{\gamma,k}\|_{W^{2,3}}+ \|\varPsi_0'(\varphi^{\gamma,k}) \|_{L^3}\big) \|\boldsymbol{v}^k\|_{\bm{L}^6} \|\nabla \partial_t \varphi^{\gamma,k} \|
\\
& \leq C\big(1+\|\nabla \mu^{\gamma,k}\| + \|\sigma^{k}\|_{L^3}
+\gamma \|\partial_t \varphi^{\gamma,k} \|_{L^3} \big)\|\nabla \boldsymbol{v}^k\| \|\nabla \partial_t \varphi^{\gamma,k} \|
\\
& \leq \frac{1}{36} \|\nabla \partial_t \varphi^{\gamma,k} \|^2
+\gamma^2 C \|\partial_t \varphi^{\gamma,k} \| \big(\|\nabla \partial_t \varphi^{\gamma,k} \| +|\overline{\partial_t \varphi^{\gamma,k}}|\big)\|\nabla \boldsymbol{v}^k\|^2
\\
&\quad +C\|\nabla \boldsymbol{v}^k\|^2\big(1+\|\nabla \mu^{\gamma,k}\|^2 + \| \sigma^{k} \|_{L^6}^2\big)
\\
& \leq \frac{1}{18} \|\nabla \partial_t \varphi^{\gamma,k} \|^2
+\gamma^4 C \|\nabla \boldsymbol{v}^k\|^4\|\partial_t \varphi^{\gamma,k} \|^2
+C\|\nabla \boldsymbol{v}^k\|^2 \|\nabla \mu^{\gamma,k}\|^2
\\
&\quad + C\|\nabla \boldsymbol{v}^k\|^2\big(1
+ \| \sigma^{k} \|_{L^6}^2\big)
+ C\alpha^2|\overline{\varphi^{\gamma,k}}-c_0|^2.
\end{align*}
Next, using \eqref{aver}, \eqref{low-es2-app1b}, H\"{o}lder's inequality and the Poincar\'{e}--Wirtinger inequality, we get
\begin{align*}
I_7&= \chi\int_\Omega \big(\partial_{t} \varphi^{\gamma,k} -\overline{\partial_{t} \varphi^{\gamma,k}}\big)\partial_t\sigma^{k} \,\mathrm{d}x + \chi \overline{\partial_{t} \varphi^{\gamma,k}} \int_\Omega   \partial_t\sigma^{k} \,\mathrm{d}x
\notag\\
&\leq C\|\partial_t\varphi^{\gamma,k} -\overline{\partial_{t} \varphi^{\gamma,k}}\|_{H^1} \|\partial_t\sigma^{k}\|_{(H^1)'}
+ |\chi||\overline{\partial_{t} \varphi^{\gamma,k}}| \|1\|_{H^1}\|\partial_t\sigma^{k}\|_{(H^1)'}
\notag\\
&\leq C\|\nabla \partial_t\varphi^{\gamma,k}\| \|\partial_t\sigma^{k}\|_{(H^1)'}
+ C\|\partial_t\sigma^{k}\|_{(H^1)'}
\\
&\leq \frac{1}{18} \|\nabla \partial_t\varphi^{\gamma,k}\|^2
+ C(1+\|\partial_t\sigma^{k}\|_{(H^1)'}^2).
\end{align*}
In a similar manner, we can treat $I_8$ as follows
\begin{align*}
I_8& =-\beta \int_\Omega \big(\partial_{t} \varphi^{\gamma,k} - \overline{\partial_{t} \varphi^{\gamma,k}}\big) \mathcal{N}\big(\partial_{t} \varphi^{\gamma,k} -\overline{\partial_{t} \varphi^{\gamma,k}}\big)\,\mathrm{d}x
\\
&\leq |\beta| \|\partial_{t} \varphi^{\gamma,k} - \overline{\partial_{t} \varphi^{\gamma,k}}\| \|\mathcal{N}(\partial_{t} \varphi^{\gamma,k} -\overline{\partial_{t}\varphi^{\gamma,k}}) \|
\\
&\leq C\|\nabla \partial_t \varphi^{\gamma,k}\|\big(\|\mu^{\gamma,k}-\overline{\mu^{\gamma,k}}\|+ \|\bm{v}^k \cdot \nabla \varphi^{\gamma,k}\|_{L^\frac65}\big)
\\
&\leq \frac{1}{18} \|\nabla \partial_t\varphi^{\gamma,k}\|^2
+ C\|\nabla \mu^{\gamma,k}\|^2 + C\|\nabla \bm{v}^k\|^2.
\end{align*}
Recalling \eqref{q2c}, \eqref{low-es2-app1b}, \eqref{q12}, we estimate $I_9$ by
\begin{align*}
I_9 & \leq \alpha^2 |\overline{\varphi^{\gamma,k}}-c_0|\left|\int_\Omega \mu^{\gamma,k}\,\mathrm{d}x\right|
\\
& \leq   C\alpha^2|\overline{\varphi^{\gamma,k}}-c_0|\big(1+\|\nabla \mu^{\gamma,k}\|  +\|\sigma^{k}\|  +\gamma \|\partial_t \varphi^{\gamma,k} \| \big)
\\
&\leq  \|\nabla \mu^{\gamma,k}\|^2
+ \gamma^2  \|\partial_t \varphi^{\gamma,k} \|^2
+ C (1+\|\sigma^{k}\|^2)
+ C\alpha^4 |\overline{\varphi^{\gamma,k}}-c_0|^2.
\end{align*}
Applying a similar argument for $I_3$ and $I_7$, we get
\begin{align}
I_{10}
&\le C\|\boldsymbol{v}^k \cdot \nabla \varphi^{\gamma,k}\|_{H^1}\|\partial_t \sigma^{k} \|_{(H^1)'}
\notag\\
&\le C\|\nabla\boldsymbol{v}^k\|^2 \| \nabla \varphi^{\gamma,k}\|_{\bm{L}^\infty}^2
+ C\|\boldsymbol{v}^{k}\|_{\bm{L}^6}^2\|\varphi^{\gamma,k}\|_{W^{2,3}}^2
+ C\|\partial_t \sigma^{k} \|_{(H^1)'}^2\notag\\
&\leq C\|\nabla\boldsymbol{v}^k\|^2 \| \varphi^{\gamma,k}\|_{W^{2,4}}^2
+ C \|\partial_t \sigma^{k} \|_{(H^1)'}^2
\notag\\
&\leq  \frac{1}{18}\|\nabla \partial_t \varphi^{\gamma,k}\|^2
+\gamma^8 C\|\nabla \boldsymbol{v}^k\|^8\|\partial_t \varphi^{\gamma,k}\|^2
+C\|\nabla \boldsymbol{v}^k\|^2 \|\nabla \mu^{\gamma,k}\|^2
\notag \\
&\quad + C\|\nabla \boldsymbol{v}^k\|^2\big(1
+ \| \sigma^{k} \|_{L^6}^2\big) + C\alpha^2|\overline{\varphi^{\gamma,k}}-c_0|^2
+ C \|\partial_t \sigma^{k} \|_{(H^1)'}^2,
\notag
\end{align}
while a similar argument for $I_6$ yields
\begin{align*}
I_{11} &\leq C\|\boldsymbol{v}^k \|\| \nabla \mathcal{N}(\partial_{t} \varphi^{\gamma,k} -\overline{\partial_{t} \varphi^{\gamma,k}})\|\| \varphi^{\gamma,k}\|_{L^\infty}\\
&\leq \frac{1}{18} \|\nabla \partial_t \varphi^{\gamma,k} \|^2 +C\|\boldsymbol{v}^k\|^2.
\end{align*}
Define
\begin{align*}
\mathcal{G}^{\gamma,k}(t)
&= \frac{1}{2}\|\nabla \mu^{\gamma,k}(t)\|^2
+\frac{\gamma}{2} \|\partial_t \varphi^{\gamma,k} (t)\|^2
+\gamma \int_{\Omega} \big(\boldsymbol{v}^k(t) \cdot \nabla \varphi^{\gamma,k}(t)\big) \partial_t \varphi^{\gamma,k}(t) \,\mathrm{d} x
\\
&\quad
+ \alpha(\overline{\varphi^{\gamma,k}(t)}-c_0) \int_\Omega   \mu^{\gamma,k}(t)\,\mathrm{d}x.
\end{align*}
We need to control the last two terms in $\mathcal{G}^{\gamma,k}$ without a definite sign.
It follows from \eqref{low-es2-app1b} that
\begin{align}
\left|\gamma \int_{\Omega} (\boldsymbol{v}^k \cdot \nabla \varphi^{\gamma,k}) \partial_t \varphi^{\gamma,k} \,\mathrm{d} x\right|
& \leq \gamma\|\boldsymbol{v}^k\|_{\bm{L}^{\infty}}\|\nabla \varphi^{\gamma,k}\| \|\partial_t \varphi^{\gamma,k} \|
\notag \\
& \leq \frac{\gamma}{8} \|\partial_t \varphi^{\gamma,k} \|^2
    +\gamma C_{12}\|\boldsymbol{v}^k\|_{\bm{L}^{\infty}}^2.
    \label{cont-1}
\end{align}
On the other hand, for $\gamma\in (0,1)$, a similar argument to $I_9$ yields
\begin{align}
& \alpha |\overline{\varphi^{\gamma,k}}-c_0|\left|\int_\Omega \mu^{\gamma,k}\,\mathrm{d}x\right|
\notag\\
&\quad  \leq   C\alpha |\overline{\varphi^{\gamma,k}}-c_0|\big(1+\|\nabla \mu^{\gamma,k}\|  +\|\sigma^{k}\|  +\gamma \|\partial_t \varphi^{\gamma,k} \| \big)
\notag \\
&\quad \leq \frac14\|\nabla \mu^{\gamma,k}\|^2 +
\frac{\gamma}{8} \|\partial_t \varphi^{\gamma,k} \|^2
+ C_{13}(1+\|\sigma^k\|^2)+ C_{13}\alpha^2 |\overline{\varphi^{\gamma,k}}-c_0|^2.
\label{cont-2}
\end{align}
In \eqref{cont-1}, \eqref{cont-2}, the positive constants $C_{12}$, $C_{13}$ are independent of $\gamma$, $k$ and time.
Recalling the construction of $\bm{v}^k$, $\sigma^k$, we choose the constants
\begin{align*}
C_{14}(k,T) & = C_{12}\sup_{t\in[0,T]}\|\bm{v}^k(t)\|_{\bm{L}^\infty}, \\
C_{15} & = C_{13}(1+\|\sigma\|_{L^\infty(0,+\infty;L^2(\Omega))}^2) + 4C_{13}\alpha^2,
\end{align*}
so that
\begin{align}
\gamma C_{12}\|\boldsymbol{v}^k(t)\|_{\bm{L}^{\infty}}^2
 + C_{13}(1+\|\sigma^k(t)\|^2)+ C_{13}\alpha^2 |\overline{\varphi^{\gamma,k}(t)}-c_0|^2
 \leq \gamma C_{14}(k,T)+ C_{15},
 \notag
\end{align}
for all $t\in [0,T]$. We note that the constant $C_{14}\geq 0$ can depend on $k$ and $T$, while $C_{15}>0$ is independent of $\gamma$, $k$ and time. Define
\begin{align*}
 \widehat{\mathcal{G}^{\gamma,k}}(t)
 & = \mathcal{G}^{\gamma,k}(t) + \gamma C_{14}(k,T)+ C_{15},
 \end{align*}
 which satisfies
\begin{align*}
&\widehat{\mathcal{G}^{\gamma,k}}(t)
\geq \frac{1}{4}\|\nabla \mu^{\gamma,k}(t)\|^2
+\frac{1}{4} \gamma\|\partial_t \varphi^{\gamma,k} (t)\|^2,\\
&\widehat{\mathcal{G}^{\gamma,k}}(t) \leq
\frac{3}{4}\|\nabla \mu^{\gamma,k}(t)\|^2
+\frac{3}{4} \gamma \|\partial_t \varphi^{\gamma,k} (t)\|^2+ 2\gamma C_{14}(k,T)+ 2C_{15},
\end{align*}
for all $t\in[0,T]$.
Combining the above estimates for $I_1,\cdots,I_{11}$, we can deduce from \eqref{high-mu} the following differential inequality:
\begin{align}
& \frac{\mathrm{d}}{\mathrm{d} t}\widehat{\mathcal{G}^{\gamma,k}}
+ \frac12 \|\nabla\partial_t\varphi^{\gamma,k}\|^2  \leq \mathcal{H}_1^{\gamma,k}\widehat{\mathcal{G}^{\gamma,k}}
+\mathcal{H}^{\gamma,k}_2,\
\label{apphigh}
\end{align}
with
 \begin{align*}
 \mathcal{H}^{\gamma,k}_1(t)
 & =C\big(\|\nabla \boldsymbol{v}^k(t)\|^2  + \gamma^3 \|\nabla \boldsymbol{v}^k(t)\|^4 +\gamma^7\|\nabla \boldsymbol{v}^k(t)\|^8\big),
 \\
 \mathcal{H}^{\gamma,k}_2(t)
 & =C(\|\nabla \mu^{\gamma,k}(t)\|^2
+ \gamma^2 \|\partial_t \varphi^{\gamma,k} (t)\|^2)
+ C \gamma^2  \|\partial_t \boldsymbol{v}^k(t) \|^2
 \\
 &\quad  + C(1+\|\nabla \bm{v}^k(t)\|^2)(1+\|\sigma^k(t)\|_{L^6}^2) + C \|\partial_t\sigma^k\|_{(H^1)'}^2.
\end{align*}
Here, the generic constant $C>0$ is independent of $\gamma$, $k$ and time.
From the construction of $\boldsymbol{v}^k$, $\sigma^k$ and the lower order estimate \eqref{low-es2-app1b}, for any given $T\in [1,+\infty)$, we infer that the non-negative functions $\widehat{\mathcal{G}^{\gamma,k}}$, $\mathcal{H}^{\gamma,k}_1$, and $\mathcal{H}^{\gamma,k}_2$ belong to $L^1(0, T)$ with upper bounds possibly depending on $k$, $T$, but not on $\gamma$ due to the fact $\gamma \in (0,1)$.

Let us first consider the estimate on the time interval $[0,1]$. Applying Gronwall's lemma, we can deduce from \eqref{apphigh} that
\begin{align*}
	& \sup_{t\in [0,1]} \widehat{\mathcal{G}^{\gamma,k}}(t)
    \leq \left(\widehat{\mathcal{G}^{\gamma,k}}(0)
+\int_0^1 \mathcal{H}^{\gamma,k}_2(s)\, \mathrm{d} s\right) \exp \left(\int_0^1 \mathcal{H}^{\gamma,k}_1(s)\, \mathrm{d} s\right),
\end{align*}
with the initial datum (noticing that $\bm{v}^k(0)=\mathbf{0}$)
\begin{align*}
\widehat{\mathcal{G}^{\gamma,k}}(0)
& = \frac12 \|\nabla \mu^{\gamma,k}(0)\|^2
+\frac{\gamma}{2} \|\partial_t \varphi^{\gamma,k}(0) \|^2
+ \alpha(\overline{\varphi_{0,k}}-c_0)
\int_\Omega   \mu^{\gamma,k}(0)\,\mathrm{d}x\\
&\quad + \gamma C_{14}(k,T)+ C_{15}.
\end{align*}
As a consequence, we get
\begin{align}
	& \sup_{t\in[0,1]}\left(\|\nabla \mu^{\gamma,k}(t)\|^2
+\gamma \|\partial_t \varphi^{\gamma,k}(t) \|^2\right)
\leq 4 \sup_{t\in [0,1]} \widehat{\mathcal{G}^{\gamma,k}}(t)
\notag \\
& \qquad \leq 4\left(\widehat{\mathcal{G}^{\gamma,k}}(0)
+  \int_0^1 \mathcal{H}^{\gamma,k}_2(s)\, \mathrm{d} s\right)
\exp \left(\int_0^1 \mathcal{H}^{\gamma,k}_1(s)\, \mathrm{d} s\right).
\label{energys}
\end{align}
Integrating \eqref{apphigh} with respect to time, we further obtain
\begin{align}
\int_0^1 \|\nabla\partial_t\varphi^{\gamma,k}(s)\|^2\,\mathrm{d}s
& \leq 2\widehat{\mathcal{G}^{\gamma,k}}(0) + 2\sup _{t \in[0, 1]}\widehat{\mathcal{G}^{\gamma,k}}(t) \int_0^1 \mathcal{H}^{\gamma,k}_1(s)\, \mathrm{d} s
+ 2\int_0^1 \mathcal{H}^{\gamma,k}_2(s)\, \mathrm{d} s.
\label{energys-b}
\end{align}
Then it is left to estimate the initial datum $\widehat{\mathcal{G}^{\gamma,k}}(0)$. Define
\begin{align*}
\widehat{\mu}_{0,k}
& =- \Delta \varphi_{0,k} +  \varPsi'(\varphi_{0,k}) -\chi\sigma^k(0) +\beta\mathcal{N}(\varphi_{0,k}-\overline{\varphi_{0,k}})\\
& = \widetilde{\mu}_{0,k} - \theta_0 \varphi_{0,k} -\chi\sigma^k(0) +\beta\mathcal{N}(\varphi_{0,k}-\overline{\varphi_{0,k}}).
\end{align*}
By the definition of $\widetilde{\mu}_{0,k}$ and the construction of $\varphi_{0,k}$, $\sigma^k$, we find
\begin{align}
  \|\nabla \widehat{\mu}_{0,k} \|
\leq C_{16}\big(\|\nabla \widetilde{\mu}_{0}\| + \|\varphi_{0}\|_{H^1}  +\|\sigma(0)\|_{H^1}\big),
\label{ini-mu-hat}
\end{align}
where $C_{16}>0$ is independent of $\gamma$, $k$ and time.
Noticing that $\partial_t\varphi^{\gamma,k}$, $\mu^{\gamma,k}\in C([0,T];H^1(\Omega))$ and $\sigma^{k}\in C([0,T];H^1(\Omega))$, we infer from equation \eqref{f2.d} that $- \Delta \varphi^{\gamma,k}+  \varPsi'(\varphi^{\gamma,k})\in C([0,T];H^1(\Omega))$.
Following a similar argument for \cite[(3.46)]{HW2022}, we can deduce that (using again $\bm{v}^k(0)=\mathbf{0}$):
\begin{align}
\|\nabla \mu^{\gamma,k}(0) \|^{2}
+ \gamma\|\partial_t \varphi^{\gamma,k}(0)\|^2
& = \gamma \alpha^2 |\Omega| |\overline{\varphi_{0,k}}-c_0|^2
+ \int_\Omega \nabla \mu^{\gamma,k}(0)\cdot \nabla\widehat{\mu}_{0,k}\,\mathrm{d}x
\notag \\
& \leq 4 \alpha^2 |\Omega| +
\frac12\|\nabla \mu^{\gamma,k}(0) \|^{2} + \frac12\|\nabla \widehat{\mu}_{0,k}\|^2.
\notag
\end{align}
From the above estimate, \eqref{q2c}, $\gamma\in (0,1)$, $\|\varphi_{0,k}\|_{L^\infty}\leq 1$ and Young's inequality, we also obtain
\begin{align}
\left|\int_\Omega \mu^{\gamma,k}(0)\,\mathrm{d}x\right|
& \leq
C\big(1
+ \|\nabla \mu^{\gamma,k}(0) \| + \gamma\|\partial_t \varphi^{\gamma,k}(0)\| +\|\sigma^k(0)\|_{H^1} \big)
\notag \\
&\leq \frac12 \|\nabla \widehat{\mu}_{0,k}\|^2 +C(1+\|\sigma(0)\|_{H^1}),
\notag
\end{align}
which is again independent of $\gamma$ and $k$ (cf. \eqref{ini-mu-hat}). In summary, we get
\begin{align}
\widehat{\mathcal{G}^{\gamma,k}}(0)
\leq \|\nabla \widehat{\mu}_{0,k}\|^2 + 8 \alpha^2 |\Omega| + 2\gamma C_{14}(k,T)+ 2C_{15},
\label{ini-Ggk}
\end{align}
which is uniformly bounded with respect to $\gamma\in(0,1)$.

Next, we consider the estimate on $[1,T]$. An application of the uniform Gronwall inequality (see, e.g., \cite[Chapter III, Lemma 1.1]{T}) to \eqref{apphigh} yields
\begin{align}
\widehat{\mathcal{G}^{\gamma,k}}(t)
& \leq \left(\sup_{r\in[0,T-1]}\int_r^{r+1} \widehat{\mathcal{G}^{\gamma,k}}(s)\,\mathrm{d}s
+ \sup_{r\in[0,T-1]}\int_r^{r+1} \mathcal{H}^{\gamma,k}_2(s)\, \mathrm{d} s\right)
\notag\\
&\quad \ \ \times \exp\left(\sup_{r\in[0,T-1]}\int_r^{r+1} \mathcal{H}^{\gamma,k}_1(s)\, \mathrm{d} s\right),\quad \forall\,t\in [1,T].
\label{energys-uG}
\end{align}
Then integrating \eqref{apphigh} with respect to time, we get
\begin{align}
\int_t^{t+1} \|\nabla\partial_t\varphi^{\gamma,k}(s)\|^2\,\mathrm{d}s
& \leq 2\widehat{\mathcal{G}^{\gamma,k}}(t) + 2\sup _{r \in[t, t+1]}\widehat{\mathcal{G}^{\gamma,k}}(r) \int_t^{t+1} \mathcal{H}^{\gamma,k}_1(s)\, \mathrm{d} s
\notag\\
&\quad
+ 2\int_t^{t+1} \mathcal{H}^{\gamma,k}_2(s)\, \mathrm{d} s, \quad \forall\, t\in [0,T-1].
\label{energys-b-uG}
\end{align}
Combining \eqref{low-es2-app1b}, \eqref{energys}--\eqref{energys-b-uG}, we arrive at higher-order estimates for $(\varphi^{\gamma,k},\mu^{\gamma,k})$ on $[0,T]$ (note that $T\geq 1$ is arbitrary but fixed), which are uniform with respect to the parameter $\gamma\in (0,1)$. Recalling \eqref{q13}, we further obtain
\begin{align}
\sup_{t\in[0,T]} \|\varphi^{\gamma,k}(t)\|_{H^2}+
   \sup_{t\in[0,T-1]}\int_{t}^{t+1} \left( \|\varphi^{\gamma,k}(s)\|_{W^{2,6}}^2  + \|\varPsi_0'(\varphi^{\gamma,k}(s))\|_{L^6}^2\right)\,\mathrm{d}s \leq C.
   \label{energys-c-uG}
\end{align}

\textbf{The vanishing viscosity limit as $\gamma \to 0^+$.} Thanks to the estimates \eqref{aver}, \eqref{low-es2-app1b}, \eqref{q13} and \eqref{energys}--\eqref{energys-c-uG}, we can pass to the limit as $\gamma\to 0^+$ (keeping $k$, $T$ fixed) such that there exists a convergent subsequence $\{(\varphi^{\gamma,k}, \mu^{\gamma,k})\}$ (not relabeled for simplicity) with the limit denoted by $(\varphi^k, \mu^k)$. The procedure follows a standard compactness argument (see e.g., \cite{A2022,Gio2022}) and its details are omitted here. The limit triple $(\varphi^k, \mu^k)$ satisfies
\begin{align*}
&\varphi^k\in L^\infty(0,T;H^2(\Omega))\cap L^2(0,T;W^{2,6}(\Omega))
\cap  H^1(0,T;H^1(\Omega)),\\
&\mu^k\in L^\infty(0,T;H^1(\Omega)),\\
&\varphi^k\in   L^{\infty}(\Omega\times(0,T))\ \ \textrm{with}\ \ |\varphi^{k}(x,t)|<1\quad \textrm{a.e.\ in}\ \Omega\times(0,T),
\end{align*}
for any $T>0$. Noticing that $(\varphi^k, \mu^k)$ satisfies
$$
(\partial_t \varphi^k, \xi) + (\bm{v}^k \cdot \nabla \varphi^k,\xi) =-(\nabla \mu^k, \nabla \xi) -\alpha \big((\overline{\varphi^k}-c_0), \xi\big), \quad \forall\, \xi \in H^1(\Omega),
$$
almost everywhere in $(0,T)$, and
\begin{align}
\|\boldsymbol{v}^k \cdot \nabla \varphi^k\|
&\leq C \|\boldsymbol{v}^k\|_{\bm{L}^6}
\|\varphi^k\|_{W^{1,3}}
 \leq C\|\boldsymbol{v}^k\|_{\bm{H}^1}
\|\varphi^k\|_{H^{2}},
\notag
\end{align}
we can infer from the elliptic regularity theory that
$$
\mu^k\in L^2(0,T;H^2(\Omega)).
$$
With the above estimates, it is straightforward to check that the strong solution $(\varphi^k, \mu^k)$ is unique. In addition, it solves the following system
\begin{equation}
    \begin{cases}
	&\partial_t \varphi^k + \bm{v}^k \cdot \nabla \varphi^k=\Delta \mu^k-\alpha(\overline{\varphi^k}-c_0),  \\
	&\mu^k=- \Delta \varphi^k+  \varPsi'(\varphi^k)-\chi \sigma^k+\beta\mathcal{N}(\varphi^k-\overline{\varphi^k}),
	\end{cases}
    \label{1f1.a-app1-b}
\end{equation}
almost everywhere in $\Omega \times (0,T)$, and fulfills the boundary and initial conditions:
\begin{equation}
    \begin{cases}
{\partial}_{\bm{n}}\varphi^k={\partial}_{\bm{n}}\mu^k =0,\quad\textrm{a.e. on} \ \partial\Omega\times(0,T),
\\
 \varphi^k|_{t=0}=\varphi_{0,k}(x),  \quad \ \ \, \textrm{a.e. in}\ \Omega.
\end{cases}
\label{1f1.a-app1-bi}
\end{equation}
Since $T>0$ is arbitrary, the solution $(\varphi^k, \mu^k)$ can be uniquely extended to $[0,+\infty)$. However, at this stage, we only have estimates in an arbitrary but fixed interval $[0,T]$.
\medskip

\textbf{Passage to the limit as $k \to +\infty$.}
For every solution $(\varphi^k, \mu^k)$, we can repeat the previous argument without the viscous term (i.e., $\gamma=0$) to show that lower-order estimates corresponding to \eqref{aver}, \eqref{low-es2-app1b}, \eqref{q13} are still valid, being uniform with respect to the parameter $k$ and the final time $T$.

In what follows, we derive higher-order estimates for $(\varphi^k, \mu^k)$. We shall work on the convergent subsequence $\big\{(\varphi^{\gamma,k}, \mu^{\gamma,k})\big\}$ (not relabeled for simplicity) with the corresponding limit $(\varphi^k, \mu^k)$. All convergence results mentioned in the following should be understood in the sense of a subsequence.

Integrating \eqref{con-CHs-es7} with respect to time, the approximate solution $(\varphi^{\gamma,k},\mu^{\gamma,k})$ satisfies the following energy identity:
\begin{align}
&  \frac12 \|\nabla \varphi^{\gamma,k}(t_2)\|^2  + \int_\Omega \varPsi(\varphi^{\gamma,k}(t_2)) \,\mathrm{d}x
+ \frac{\beta}{2} \|\nabla \mathcal{N}(\varphi^{\gamma,k}(t_2) -\overline{\varphi^{\gamma,k}(t_2)})\|^2 \notag\\
&\qquad
+\int_{t_1}^{t_2} (\|\nabla \mu^{\gamma,k}(s)\|^2 + \gamma\|\partial_t\varphi^{\gamma,k}(s)\|^2)\,\mathrm{d}s
\notag \\
&\quad = \frac12 \|\nabla \varphi^{\gamma,k}(t_1)\|^2 + \int_\Omega \varPsi(\varphi^{\gamma,k}(t_1)) \,\mathrm{d}x
+ \frac{\beta}{2} \|\nabla \mathcal{N}(\varphi^{\gamma,k}(t_1)-\overline{\varphi^{\gamma,k}(t_1)})\|^2
\notag\\
& \qquad  -\int_{t_1}^{t_2}\!\! \int_\Omega (\bm{v}^k(s)\cdot\nabla \varphi^{\gamma,k}(s))(\mu^{\gamma,k}(s) +\chi\sigma^k(s))\,\mathrm{d}x\mathrm{d}s
-\chi\int_{t_1}^{t_2} \!\!\int_\Omega \nabla \mu^{\gamma,k}(s)\cdot \nabla \sigma^k(s)\,\mathrm{d}x\mathrm{d}s
\notag\\
&\qquad
-\int_{t_1}^{t_2} \alpha(\overline{\varphi^{\gamma,k}(s)}-c_0)\int_\Omega (\mu^{\gamma,k}(s)+ \chi \sigma^k(s))\,\mathrm{d}x\mathrm{d}s,
\notag
\end{align}
for any $0\leq t_1\leq t_2\leq T$. Due to the regularity of $(\varphi^k, \mu^k)$, we can check that it also satisfies an energy identity:
\begin{align}
&  \frac12 \|\nabla \varphi^{k}(t_2)\|^2
+ \int_\Omega \varPsi(\varphi^{k}(t_2))\,\mathrm{d}x
+ \frac{\beta}{2} \|\nabla \mathcal{N}(\varphi^{k}(t_2)-\overline{\varphi^{k}(t_2)})\|^2
+\int_{t_1}^{t_2} \|\nabla \mu^{k}(s)\|^2\,\mathrm{d}s
\notag \\
&\quad
= \frac12 \|\nabla \varphi^{k}(t_1)\|^2
+ \int_\Omega \varPsi(\varphi^{k}(t_1))\,\mathrm{d}x
+ \frac{\beta}{2} \|\nabla \mathcal{N}(\varphi^{k}(t_1)-\overline{\varphi^{k}(t_1)})\|^2
\notag\\
& \qquad  -\int_{t_1}^{t_2} \!\! \int_\Omega (\bm{v}^k(s)\cdot\nabla \varphi^{k}(s))(\mu^{k}(s)+\chi\sigma^k(s))\,\mathrm{d}x\mathrm{d}s
-\chi\int_{t_1}^{t_2}\!\! \int_\Omega \nabla \mu^{k}(s)\cdot \nabla \sigma^k(s)\,\mathrm{d}x\mathrm{d}s
\notag\\
&\qquad -\int_{t_1}^{t_2} \alpha(\overline{\varphi^{k}(s)}-c_0)\int_\Omega (\mu^{k}(s) + \chi \sigma^k(s))\,\mathrm{d}x\mathrm{d}s,
\notag
\end{align}
for any $0\leq t_1\leq t_2\leq T$.
Using \eqref{aver}, \eqref{energys-b}, \eqref{ini-mu-hat}, \eqref{ini-Ggk}, \eqref{energys-b-uG} and the Poincar\'{e}--Wirtinger inequality, we observe that
\begin{align}
   \int_{t_1}^{t_2} \gamma\|\partial_t\varphi^{\gamma,k}(s)\|^2 \,\mathrm{d}s
  & \leq  \gamma C\int_{t_1}^{t_2}\|\nabla\partial_t\varphi^{\gamma,k}(s)\|^2 \,\mathrm{d}s + \gamma C |\Omega| \int_{t_1}^{t_2}  |\overline{\partial_t\varphi^{\gamma,k}(s)}|^2 \,\mathrm{d}s
  \notag\\
  &\to 0, \quad \text{as}\ \gamma\to 0^+.  \notag
\end{align}
By the argument in \cite[Section 2]{A2022}, comparing the two energy identities above and taking the limit as $\gamma\to 0^+$, we obtain
\begin{align*}
 \lim_{\gamma\to 0^+}\int_{t_1}^{t_2} \|\nabla \mu^{\gamma,k}(s)\|^2 \,\mathrm{d}s
=\int_{t_1}^{t_2}\|\nabla \mu^{k}(s)\|^2\,\mathrm{d}s,
\end{align*}
for any $0\leq t_1\leq t_2\leq T$.
With this observation, by the lower semi-continuity of norms with respect to the weak convergence, we can pass to the limit as $\gamma\to 0^+$ in the inequalities \eqref{energys}, \eqref{energys-b} to get (with the aid of \eqref{ini-mu-hat}, \eqref{ini-Ggk})
\begin{align}
	& \sup_{t\in[0,1]} \|\nabla \mu^{k}(t)\|^2
  \leq 4\left( C_{17}
+ \int_0^1 \mathcal{H}^{k}_2(s)\, \mathrm{d} s\right)
\exp \left(C\int_0^1 \|\nabla \boldsymbol{v}^k(s)\|^2 \, \mathrm{d} s\right),
\notag
\end{align}
and
\begin{align}
\int_0^1 \|\nabla\partial_t\varphi^{k}(s)\|^2\,\mathrm{d}s
& \leq 2 C_{17}
+ 2C \left( C_{17}
+ \int_0^1 \mathcal{H}^{k}_2(s)\, \mathrm{d} s\right)\int_0^1 \|\nabla \boldsymbol{v}^k(s)\|^2 \, \mathrm{d} s \notag\\
&\qquad\  \times
\exp \left(C\int_0^1 \|\nabla \boldsymbol{v}^k(s)\|^2 \, \mathrm{d} s\right)+ 2\int_0^1 \mathcal{H}^{k}_2(s)\, \mathrm{d} s,
\notag
\end{align}
where
\begin{align*}
C_{17} & = C_{16}\big(\|\nabla \widetilde{\mu}_{0}\| + \|\varphi_{0}\|_{H^1}  +\|\sigma(0)\|_{H^1}\big) + 8\alpha^2|\Omega| + 2C_{15},
\\
\mathcal{H}^{k}_2(t)
 & =C\|\nabla \mu^{k}(t)\|^2
 + C(1+\|\nabla \bm{v}^k(t)\|^2)(1+\|\sigma^k(t)\|_{L^6}^2) + C \|\partial_t\sigma^k\|_{(H^1)'}^2.
\end{align*}
Similarly, we can pass to the limit as $\gamma\to 0^+$ in the inequalities \eqref{energys-uG}, \eqref{energys-b-uG} to obtain (with the aid of \eqref{ini-mu-hat}, \eqref{ini-Ggk})
\begin{align}
\sup_{t\in[1,T]} \|\nabla \mu^{k}(t)\|^2
& \leq 4\left(\sup_{r\in[0,T-1]}A(r) + \sup_{r\in[0,T-1]}\int_r^{r+1} \mathcal{H}^{k}_2(s)\, \mathrm{d} s\right)
\notag\\
&\quad \ \ \ \times \exp\left(C\sup_{r\in[0,T-1]}\int_r^{r+1} \|\nabla \bm{v}^k(s)\|^2\, \mathrm{d} s\right),
\notag
\end{align}
and
\begin{align}
& \sup_{t\in[0,T-1]}\int_t^{t+1} \|\nabla\partial_t\varphi^{k}(s)\|^2\,\mathrm{d}s \notag\\
&\quad  \leq 2\left(C_{17}+ \sup_{r\in[0,T-1]}A(r) + \sup_{r\in[0,T-1]}\int_r^{r+1} \mathcal{H}^{k}_2(s)\, \mathrm{d} s\right)
\left(1  + C\int_t^{t+1} \|\nabla \bm{v}^k(s)\|^2\, \mathrm{d} s\right)
\notag\\
&\qquad \ \times \exp\left(C\sup_{r\in[0,T-1]}\int_r^{r+1} \|\nabla \bm{v}^k(s)\|^2\, \mathrm{d} s\right)
+ 2\int_t^{t+1} \mathcal{H}^{k}_2(s)\, \mathrm{d} s,
\notag
\end{align}
where
$$
A(r)=\int_r^{r+1}  \|\nabla \mu^k(s)\|^2 \,\mathrm{d}s + 2C_{15},\quad r\in [0,T-1].
$$
Hence, we have derived estimates for $\mu^k$, $\partial_t\varphi^k$ on $[0,T]$, which are independent of $k$ and $T$.

Collecting the above estimates, we can conclude
\begin{align}
& \|\mu^k\|_{L^{\infty}(0, +\infty ; H^1(\Omega))}
+ \|\partial_t \varphi^k\|_{L^2_{\mathrm{uloc}}(0, +\infty ; H^1(\Omega))}
\notag \\
& \quad + \|\varphi^k\|_{L^{\infty}(0, +\infty ; W^{2, 6}(\Omega))}
+  \|\varPsi'(\varphi^k)\|_{L^{\infty}(0, +\infty; L^6(\Omega))} \leq C.
\label{uni-es}
\end{align}
Noticing that
\begin{align}
\|\boldsymbol{v}^k \cdot \nabla \varphi^k\|_{ H^1}
&\leq C(\|\boldsymbol{v}^k\|_{\bm{H}^1}
\|\nabla \varphi^k\|_{ \bm{L}^{\infty}}
+ \|\boldsymbol{v}^k\|_{\bm{L}^6}
\|\varphi^k\|_{W^{2,3}})
\notag \\
&\leq C\|\boldsymbol{v}^k\|_{\bm{H}^1}
\|\varphi^k\|_{W^{2,6}},
\notag
\end{align}
from \eqref{uni-es} and the elliptic regularity theory applied to the first equation in \eqref{1f1.a-app1-b}, we also get
\begin{align}
\|\mu^k\|_{L^2_{\mathrm{uloc}}([0, +\infty) ; H^3(\Omega))} \leq C.
\label{uni-es-mu3}
\end{align}
The positive constant $C$ appearing in \eqref{uni-es}, \eqref{uni-es-mu3} depends on $\max_{r\in[-1,1]}|\varPsi(r)|$, $\|\widetilde{\mu}_{0}\|_{H^1}$, $\|\varphi_0\|_{H^1}$, $\overline{\varphi_0}$, $\|\sigma(0)\|_{H^1}$, the coefficients of the system, $\Omega$ and norms of the given data $\|\boldsymbol{v}\|_{L^2(0, +\infty ; \bm{H}_{0, \mathrm{div}}^1(\Omega))}$, $\|\sigma\|_{L^2_{\mathrm{uloc}}([0,+\infty);H^1(\Omega))}$, $\|\sigma\|_{L^\infty(0,+\infty;L^6(\Omega))}$ and $\|\partial_t\sigma\|_{L^2_{\mathrm{uloc}}([0,+\infty);(H^1(\Omega))')}$, but it is independent of $k$ and $T$.

With the above uniform estimates, we can apply the standard compactness argument and pass to the limit as $k\to+\infty$. There exists a convergent subsequence $\big\{(\varphi^k, \mu^k)\}$ whose limit is denoted by $(\varphi, \mu)$. It is straightforward to check that $(\varphi, \mu)$ is a solution to the original problem \eqref{con-CHs} on $[0,+\infty)$ with the corresponding regularity properties as stated in Proposition \ref{con-CHs-ws}-(2).

By comparison in \eqref{1f1.a},
we find
$$
\|\partial_t\varphi\|_{(H^1)'}\leq C\big(\|\bm{v}\|\|\varphi\|_{L^\infty}+\|\nabla \mu\|+ \alpha \|\overline{\varphi}-c_0\|\big).
$$
Thus, if in addition, $\bm{v}\in L^\infty(0,+\infty; \bm{L}^2_{0, \mathrm{div}}(\Omega))$, then $\partial_t\varphi\in L^\infty(0,+\infty;(H^1(\Omega))')$.

The proof of Proposition \ref{con-CHs-ws}-(2) is complete.
\end{proof}

\subsubsection{Proof of Theorem \ref{vch}}
We are ready to prove Theorem \ref{vch}.

\begin{proof}[\textbf{Proof of Theorem \ref{vch}}]
Let the assumptions of Theorem \ref{vch} be satisfied.
\smallskip

\textbf{Step 1.} We can apply Theorem \ref{vch-weak} to conclude that problem \eqref{1f1}--\eqref{ini01} with the given velocity field $\bm{v}$ admits a global weak solution $(\varphi^*, \mu^*, \sigma^*)$ on $[0,+\infty)$.
\smallskip

\textbf{Step 2.} According to Proposition \ref{sig-weak}-(3), the decoupled auxiliary problem \eqref{1f1-sig} with $\varphi=\varphi^*$, the same velocity field $\bm{v}$ and the same initial datum $\sigma_0\in H^1(\Omega)\hookrightarrow L^6(\Omega)$ as for problem \eqref{1f1}--\eqref{ini01}, admits a unique global weak solution $\sigma^\sharp$ on $[0,+\infty)$ such that
\begin{align}
  & \sigma^\sharp\in L^{\infty}(0, +\infty ; L^6(\Omega))\cap L_{\mathrm{uloc}}^2([0, +\infty); H^1(\Omega))\cap   H^{1}_{\mathrm{uloc}}([0, +\infty); (H^1(\Omega))'),
  \notag
\end{align}
and $\sigma^\sharp(0)=\sigma_0\in H^1(\Omega)$.
On the other hand, we easily check that $\sigma^*$ is actually a global weak solution to the same problem \eqref{1f1-sig} with properties described as in Proposition \ref{sig-weak}-(1). Hence, by the conditional uniqueness result (that is, Proposition \ref{sig-weak}-(2)), we find $\sigma^*=\sigma^\sharp$ on $[0,+\infty)$.
\smallskip

\textbf{Step 3.} Consider now the decoupled auxiliary problem \eqref{con-CHs} with $\sigma=\sigma^\sharp$, the same velocity field $\bm{v}$ and the same initial datum $\varphi_0$ as for problem \eqref{1f1}--\eqref{ini01}. According to Proposition \ref{con-CHs-ws}-(2), problem \eqref{con-CHs} admits a unique global strong solution $(\varphi^\sharp, \mu^\sharp)$ on $[0,+\infty)$. On the other hand, we can check that $(\varphi^*,\mu^*)$ is indeed a global weak solution to the same problem \eqref{con-CHs} with properties described as in Proposition \ref{con-CHs-ws}-(1). Due to the uniqueness of weak solutions, we have $(\varphi^*,\mu^*)=(\varphi^\sharp,\mu^\sharp)$ on $[0,+\infty)$.

In summary, we have shown $(\varphi^*,\mu^*, \sigma^*)=(\varphi^\sharp,\mu^\sharp, \sigma^\sharp)$ on $[0,+\infty)$. In addition, by a comparison in \eqref{1f2.b}, we see that
\begin{align}
 \int_0^{+\infty} \|\partial_t\sigma^*(s)\|_{(H^1)'}^2\,\mathrm{d}s
 &  \leq C\int_0^{+\infty} \big(\|\bm{v}(s)\|_{\bm{L}^6}^2 \| \sigma^*(s)\|_{L^3}^2
+\|\nabla (\sigma^*(s)-\chi\varphi^*(s))\|^2\big) \,\mathrm{d}s \notag\\
&  \leq  C \sup_{t\geq 0}\|\sigma^*(s)\|_{L^3}^2 \int_0^{+\infty}  \|\bm{v}(s)\|_{\bm{H}^1}^2\,\mathrm{d}s
\notag\\
&\quad + C \int_0^{+\infty}\|\nabla (\sigma^*(s)-\chi\varphi^*(s))\|^2\,\mathrm{d}s,
\notag
\end{align}
which implies that $\sigma^*\in H^1([0,+\infty);(H^1(\Omega))')$. Thanks to the regularity properties of $\sigma^\sharp$ and $(\varphi^\sharp,\mu^\sharp)$ guaranteed by Proposition \ref{sig-weak}-(3) and Proposition \ref{con-CHs-ws}-(2), respectively, we observe that $(\varphi^*,\mu^*, \sigma^*)$ is actually a global solution of the target problem \eqref{1f1}--\eqref{ini01} on $[0,+\infty)$ with the regularity properties described in Theorem \ref{vch}.
Its uniqueness follows from the conditional uniqueness result established in Theorem \ref{vch-weak}-(2). Finally, the regularity of $(\varphi^*,\mu^*, \sigma^*)$ enables us to derive the energy identity \eqref{a-wBEL} by testing \eqref{1f1.a} with $\mu$ and \eqref{1f2.b} with $\sigma-\chi\varphi$, respectively, adding the resultants together, integrating over $\Omega$ and then $[t_1,t_2]$.

The proof of Theorem \ref{vch} is complete.
\end{proof}

\section{Proof of Main Results}
\label{proof-main}
\setcounter{equation}{0}

Let $(\bm{v}, \varphi,\mu,\sigma)$ be a global weak solution to problem \eqref{nsch}--\eqref{ini0} with initial datum $(\bm{v}_0,\varphi_0, \sigma_0)$ as given in Proposition \ref{main}.
By its construction, we have
\begin{align}
\sup_{t\geq 0}\|\varphi(t)\|_{L^\infty}\leq 1,
\label{vp-Linf}
\end{align}
and the mass relations:
\begin{align}
&\overline{\varphi}(t)-c_0=(\overline{\varphi_0}-c_0)e^{-\alpha t},\quad
\overline{\sigma}(t)=\overline{\sigma_0},\quad \forall\, t\geq 0.
\label{mass-ps}
\end{align}
Since $(\varphi,\mu,\sigma)$ satisfies an estimate similar to \eqref{R1es-app1}, we can infer from the energy inequality \eqref{wBEL} that
\begin{align}
& E(\bm{v}(t),\varphi(t), \sigma(t))
+  \int_{0}^{t}\!\int_{\Omega} \left( 2\nu(\varphi)|D\bm{v}|^2 + \frac12 |\nabla \mu|^2+\frac12 |\nabla(\sigma-\chi\varphi)|^2\right) \mathrm{d}x\mathrm{d}s
\nonumber\\
&\quad \leq E(\bm{v}_0,\varphi_0, \sigma_0) + C,
\notag
\end{align}
where $C>0$ depends on $\overline{\varphi_0}$, $\overline{\sigma_0}$, the coefficients of the system, and $\Omega$. This together with \eqref{vp-Linf}, (H1), Korn's inequality yields the uniform-in-time estimate:
\begin{align}
& \sup_{t\geq 0}\left(\|\bm{v}(t)\|^2+ \|\varphi(t)\|_{H^1}^2
+\|\sigma(t)\|^2\right)\notag \\
&\quad + \int_{0}^{+\infty}\left(\|\nabla \bm{v}(t)\|^2 + \|\nabla \mu(s)\|^2+ \|\nabla(\sigma(s)-\chi\varphi(s))\|^2\right)\mathrm{d}s
\leq C.
\label{wBEL-2}
\end{align}

In what follows, we characterize the regularizing effect of the global weak solution at different time stages.

\subsection{Instantaneous regularity of $(\varphi,\mu,\sigma)$}
\label{ins-reg}
We begin by proving the first part of Theorem \ref{2main}.

\begin{proof}[\textbf{Proof of Theorem \ref{2main}-(1)}]
Due to the regularity properties $\mu\in L^2(0,1;H^1(\Omega))$, $\varphi\in L^4(0,1;H^2_N(\Omega))$ and $\sigma\in L^2(0,1;H^1(\Omega))$, for any $\tau \in (0,1)$, there exists $\widetilde{\tau}\in (0,\tau)$ such that
$$
\varphi(\widetilde{\tau})\in H^2_N(\Omega),\quad
\mu(\widetilde{\tau})\in H^1(\Omega),\quad \sigma(\widetilde{\tau}) \in H^1(\Omega),
$$
with $\|\varphi(\widetilde{\tau})\|_{L^\infty}\leq 1$ and $\overline{\varphi(\widetilde{\tau})}\in (-1,1)$. Taking $\widetilde{\tau}$ as the initial time, we apply Theorem \ref{vch} on the interval $[\widetilde{\tau},+\infty)$. This, combined with the conditional uniqueness result established in Theorem \ref{vch-weak}-(2), allows us to achieve the conclusion of Theorem \ref{2main}-(1).
\end{proof}

\subsection{The $\omega$-limit set and eventual separation property of $\varphi$}
Like in Section \ref{ins-reg}, there exists $\widetilde{\tau} \in (0,1)$ such that we can apply Theorem \ref{vch} on $[\widetilde{\tau},+\infty)$ and  conclude
\begin{align}
 \|\varphi\|_{L^\infty(1,+\infty;W^{2,6}(\Omega))}+\|\sigma\|_{L^\infty(1,+\infty;L^6(\Omega))} \leq C.
\label{uni-es2}
\end{align}
In addition, we have
\begin{align}
\int_1^{+\infty} \|\partial_t\varphi(s)\|_{(H^1)'}^2\,\mathrm{d}s
+\int_1^{+\infty} \|\partial_t\sigma(s)\|_{(H^1)'}^2\,\mathrm{d}s \leq C.
\label{uni-es2-tt}
\end{align}
\begin{remark}\rm
According to \eqref{mass-ps}, if $\alpha=0$, we have the mass conservation $\overline{\varphi(t)}=\overline{\varphi_0}$ for all $t\geq 0$, if $\alpha>0$, then $\overline{\varphi(t)}$ converges exponentially fast to $c_0$ as $t\to +\infty$. In what follows, we focus on the more involved case with $\alpha>0$. The case $\alpha=0$ can be treated in a similar manner with minor modifications.
\end{remark}

In view of \eqref{mass-ps} (assuming $\alpha >0$),
\eqref{wBEL-2} and \eqref{uni-es2}, we introduce the $\omega$-limit set for the global weak solution $(\bm{v}, \varphi, \sigma)$:
\begin{align}
\omega(\bm{v},\varphi, \sigma)
&=\big\{(\bm{v}_\infty,\varphi_\infty, \sigma_\infty)  \in \mathcal{Z} :\  \exists\,\left\{t_{n}\right\} \nearrow +\infty\text { such that }\notag\\
& \qquad  \left(\bm{v}(t_{n}),\varphi\left(t_{n}\right),\sigma \left(t_{n}\right)\right) \rightarrow (\bm{v}_\infty,\varphi_\infty, \sigma_\infty)\text{ weakly in } \bm{L}^2(\Omega)\times W^{2,6}(\Omega)\times L^{6}(\Omega)\big\},
\notag
\end{align}
where
$$
\mathcal{Z}=\left\{(\bm{z}_1,z_2,z_3) \in \bm{L}^2_{0,\mathrm{div}}(\Omega)\times (W^{2,6}(\Omega)\cap H^2_N(\Omega))\times L^{6}(\Omega):\|  z_2 \|_{L^{\infty}} \le 1,\
 \overline{ z_2}=c_0,\ \overline{ z_3}=\overline{\sigma_0}\right\}.
$$

Next, we provide a characterization of $\omega(\bm{v},\varphi, \sigma)$.

\begin{proposition} \label{sta}
Let $(\bm{v}, \varphi,\mu,\sigma)$ be a global weak solution to problem \eqref{nsch}--\eqref{ini0} with initial datum $(\bm{v}_0,\varphi_0, \sigma_0)$ as given in Proposition \ref{main}.
The corresponding $\omega$-limit set $\omega(\bm{v},\varphi, \sigma)$ is a non-empty and compact subset of $\bm{L}^2(\Omega)\times W^{2-\epsilon,6}(\Omega)\times W^{2-\epsilon,6}(\Omega)$ for any $\epsilon\in (0,1/2)$. The total energy $E(\bm{v},\varphi,\sigma)$ and the free energy $\mathcal{F}(\varphi, \sigma)$ are constant on $\omega(\bm{v},\varphi, \sigma)$.
For every $(\bm{v}_\infty,\varphi_\infty, \sigma_\infty)\in \omega(\bm{v},\varphi, \sigma)$, we have
\begin{align}
\lim_{t\to +\infty}\|\boldsymbol{v}(t)\|= 0,
\label{decay-v}
\end{align}
so that $\bm{v}_\infty=\mathbf{0}$. Moreover, $(\varphi_{\infty}, \sigma_{\infty}) \in (W^{2,6}(\Omega)\cap H^2_N(\Omega))\times (W^{2,6}(\Omega)\cap H^2_N(\Omega))$ is a strong solution to the stationary problem \eqref{5bchv}--\eqref{5cchv} satisfying $(\overline{\varphi_\infty},\overline{\sigma_\infty}) =(c_0,\overline{\sigma_0})$. There exists a constant $\delta_0\in (0,1)$ such that
	\be
	\|\varphi_{\infty}\|_{C(\overline{\Omega})}\le 1-\delta_0,
    \label{ps1}
	\ee
	where $\delta_0$ is independent of $\varphi_\infty$.
\end{proposition}

\begin{proof} Thanks to \eqref{mass-ps}, \eqref{wBEL-2}, \eqref{uni-es2} and the Banach--Alaoglu theorem, we see that
$\omega(\bm{v},\varphi, \sigma)$ is a non-empty bounded set in $\mathcal{Z}$.

Next, we prove \eqref{decay-v} by adapting the argument for \cite[Lemma 3.2]{A2022}. Using the energy inequality \eqref{wBEL} for the full system and the energy equality \eqref{a-wBEL} for the subsystem of $(\varphi, \sigma)$ (for $t\geq 1$), we can deduce that
\begin{align}
& \frac{1}{2}\|\bm{v}(t_2)\|^2
+  \int_{t_1}^{t_2}\!\!\int_{\Omega} 2\nu(\varphi(s))|D\bm{v}(s)|^2 \, \mathrm{d}x\mathrm{d}s
\notag\\
&\quad \leq \frac{1}{2}\|\bm{v}(t_1)\|^2
-\int_{t_1}^{t_2}\!\!\int_\Omega \big[\bm{v}(s)\cdot \nabla (\mu(s) +\chi \sigma(s))\big] \varphi(s)\,\mathrm{d}x \mathrm{d}s,
\notag
\end{align}
for almost all $t_1\geq 1$ and all $t_2\in [t_1,+\infty)$. Using the incompressibility condition $\mathrm{div}\bm{v}=0$ and \eqref{vp-Linf}, the second term on the right-hand side can be controlled as follows:
\begin{align}
&-\int_{t_1}^{t_2}\!\!\int_\Omega  \big[\bm{v}(s)\cdot \nabla (\mu(s) +\chi \sigma(s))\big] \varphi(s)\,\mathrm{d}x \mathrm{d}s
\notag\\
&\quad = -\int_{t_1}^{t_2}\!\!\int_\Omega  (\bm{v}(s)\cdot \nabla \mu(s)) \varphi(s)\,\mathrm{d}x \mathrm{d}s - \chi \int_{t_1}^{t_2}\!\!\int_\Omega  \big[\bm{v}(s)\cdot \nabla (\sigma(s)-\chi\varphi(s))\big] \varphi(s)\,\mathrm{d}x \mathrm{d}s
\notag\\
&\quad\le \big(\|\nabla\mu\|_{L^2(t_1,+\infty;\bm{L}^2(\Omega))} +|\chi|\|\nabla(\sigma-\chi\varphi)\|_{L^2(t_1,+\infty;\bm{L}^2(\Omega))}\big) \|\bm{v}\|_{L^2(t_1,+\infty;\bm{L}^2(\Omega))},
\label{L1vmup}
\end{align}
Thus, for arbitrary $\varepsilon\in (0,1)$, we infer from \eqref{wBEL-2} and Korn's inequality that there exists $t_1=t_1(\varepsilon)$ sufficiently large such that
$$
\|\nabla\mu\|_{L^2(t_1,+\infty;\bm{L}^2(\Omega))} \leq \varepsilon,
\quad \|\nabla(\sigma-\chi\varphi)\|_{L^2(t_1,+\infty;L^2(\Omega))} \leq \varepsilon,\quad
\|\bm{v}\|_{L^2(t_1,+\infty;\bm{H}^1(\Omega))}\leq \varepsilon,
$$
as well as $
\|\bm{v}(t_1)\|\leq \varepsilon$.
Hence, we have
$$
\begin{aligned}
& \sup _{t\geq t_1(\varepsilon)} \frac{1}{2}\|\bm{v}(t)\|^2
+ \int_{t_1(\varepsilon)}^{+\infty} \!\int_\Omega 2\nu(\varphi(s))|D\bm{v}(s)|^2\, \mathrm{d}x\mathrm{d}s
 \leq \frac12\varepsilon^2+ (1+|\chi|)\varepsilon^2,
\end{aligned}
$$
which together with the Cauchy--Schwarz, Korn and Young inequalities yields
$$
\|\boldsymbol{v}\|_{L^{\infty}(t_1,+ \infty ; \bm{L}^2(\Omega))}^2+\|\boldsymbol{v}\|_{L^2(t_1, +\infty ; \bm{H}^1(\Omega))}^2  \leq C (1+|\chi|)  \varepsilon^2.
$$
Here, $C>0$ is independent of $t_1$ and $\varepsilon$. As a result, the decay property \eqref{decay-v} holds and $\bm{v}_\infty=\mathbf{0}$.

Thanks to \eqref{mass-ps} and \eqref{wBEL-2}, it holds
\begin{align}
\alpha \int_0^{+\infty}\!\!\int_\Omega |\overline{\varphi(s)}-c_0||\mu(s)| \,\mathrm{d}x\mathrm{d}s
& \leq C\alpha\int_0^{+\infty} e^{-\alpha s} (1+\|\nabla \mu(s)\|)\,\mathrm{d}s \notag\\
& \leq C\alpha \left(1+ \int_0^{+\infty} \|\nabla \mu(s)\|^2\,\mathrm{d}s\right)
<+\infty,
\notag
\end{align}
that is, $\alpha (\overline{\varphi}-c_0)\mu\in L^1(\Omega\times (0,+\infty))$. Recalling the energy equality \eqref{a-wBEL} for any $t_2\geq t_1=1$, we infer from \eqref{wBEL-2}, \eqref{decay-v}, \eqref{L1vmup} and the above fact that
\begin{align}
& \lim_{t\to +\infty} E(\bm{v}(t),\varphi(t),\sigma(t))=\lim_{t\to +\infty} \mathcal{F}(\varphi(t), \sigma(t))=E_0,
\notag
\end{align}
for some constant $E_0\in \mathbb{R}$. Hence, the total energy $E(\bm{v},\varphi,\sigma)$ and the free energy $\mathcal{F}(\varphi, \sigma)$ are equal constant on $\omega(\bm{v},\varphi, \sigma)$.

Let $(\mathbf{0},\varphi_\infty,\sigma_\infty)$ be an arbitrarily given element in $\omega(\bm{v}, \varphi, \sigma)$. We denote its corresponding convergent subsequence by $\big\{(\bm{v}(t_n),\varphi(t_n),\sigma(t_n)\big\}$.

By the Aubin--Lions compactness lemma, we infer from \eqref{uni-es2} and $\varphi\in H^1_{\mathrm{uloc}}([0,+\infty);H^1(\Omega))$ that $\varphi\in BUC([1,+\infty);W^{2-\epsilon,6}(\Omega))$ for any $\epsilon\in (0,1/2)$. In addition, due to the compactness embedding and uniqueness of the weak and strong limits, it holds $\varphi(t_n)\to \varphi_\infty$ strongly in $W^{2-\epsilon,6}(\Omega)$ as $t_n\to +\infty$.
Recalling the instantaneous regularity of $(\varphi,\mu,\sigma)$, for $t\geq 1$, we are allowed to test \eqref{test1.a} by $-\chi(\sigma-\chi\varphi)$ and \eqref{test2.b} by $\sigma-\chi\varphi$, respectively. Adding the resultants together yields
\begin{align}
&\frac12 \frac{\mathrm{d}}{\mathrm{d}t} \|\sigma-\chi\varphi\|^2 +\|\nabla(\sigma-\chi\varphi)\|^2
\notag \\
&\quad=\chi\int_{\Omega}\nabla\mu\cdot\nabla(\sigma-\chi\varphi)\, \mathrm{d}x
+\alpha\chi(\overline{\varphi}-c_0)\int_{\Omega}(\sigma-\chi\varphi)
\,\mathrm{d}x,
\label{energy n}
\end{align}
for almost all $t\geq 1$.
On the other hand, it follows from \eqref{mass-ps} that
\begin{align}
& \frac{\mathrm{d}}{\mathrm{d}t} \Big(\int_{\Omega}-(\sigma-\chi\varphi) (\overline{\sigma}-\chi\overline{\varphi}) +\frac12(\overline{\sigma}-\chi\overline{\varphi})^2\,\mathrm{d}x\Big)
=-\alpha\chi(\overline{\varphi}-c_0) \int_{\Omega}(\overline{\sigma}-\chi\overline{\varphi})\,\mathrm{d}x.
\label{energy n1}
\end{align}
Combining \eqref{energy n}, \eqref{energy n1} and Young's inequality, we obtain
\begin{align}
& \frac{\mathrm{d}}{\mathrm{d}t}\|(\sigma-\chi\varphi) -(\overline{\sigma}-\chi\overline{\varphi})\|^2
+\|\nabla(\sigma-\chi\varphi)\|^2
\le \chi^2 \|\nabla\mu\|^2.
\label{energy n2}
\end{align}
Then by \eqref{mass-ps}, \eqref{wBEL-2}, the Poincar\'{e}--Wirtinger inequality and a similar argument for the decay of $\bm{v}$, we can deduce from \eqref{energy n2} that
\begin{align}
\lim_{t\to +\infty} \|(\sigma(t)-\chi\varphi(t))-(\overline{\sigma_0}-\chi c_0)\| =0. \label{con-aves}
\end{align}
This together with the strong convergence of $\{\varphi(t_n)\}$ in $W^{2-\epsilon,6}(\Omega)$ implies
\begin{align}
\lim_{t_n\to +\infty} \|\sigma(t_n)-[\chi\varphi_\infty+(\overline{\sigma_0}-\chi c_0)]\|=0,
\notag
\end{align}
and thus the limit $\sigma_\infty=\chi\varphi_\infty+(\overline{\sigma_0}-\chi c_0)\in W^{2,6}(\Omega)\cap H^2_N(\Omega)$.

As a consequence, $\omega(\bm{v},\varphi, \sigma)$ is a non-empty compact set in $\bm{L}^2(\Omega)\times W^{2-\epsilon,6}(\Omega)\times W^{2-\epsilon,6}(\Omega)$ for any $\epsilon\in (0,1/2)$.

Below we show that $(\varphi_{\infty}, \sigma_{\infty}) \in (W^{2,6}(\Omega)\cap H^2_N(\Omega))\times (W^{2,6}(\Omega)\cap H^2_N(\Omega))$ is a strong solution to the stationary problem \eqref{5bchv}--\eqref{5cchv}. The boundary condition \eqref{5cchv} is obvious. Besides, \eqref{5f2.b} is a consequence of \eqref{con-aves}, that is, $\sigma_\infty-\chi\varphi_\infty$ is indeed a constant. Thanks to the estimate \eqref{uni-es2} on $\partial_t\varphi$ and $\partial_t\sigma$, by an argument similar to that in \cite[Section 5]{H2} (see also \cite[Section 6]{JWZ}, \cite[Section 4]{WuXu24}), we can verify \eqref{5bchv}.

Concerning the strict separation property of $\varphi_\infty$, we apply a  dynamic approach inspired by \cite{MZ04}. The key observation is that every $(\varphi_\infty,\sigma_\infty)$ can be viewed as (at least) a
global weak solution to the evolution problem \eqref{1f1}--\eqref{ini01} with $\gamma\in (0,1)$, $\bm{v}=\mathbf{0}$ and the initial datum given
by $(\varphi_\infty,\sigma_\infty)$ itself. Then it follows from \cite[Proposition 5.1]{H2} that \eqref{ps1} holds with some $\delta_0\in (0,1)$. Moreover, since the set $\{\varphi_\infty\}$ is compact in $W^{2-\epsilon,6}(\Omega)\hookrightarrow C(\overline{\Omega})$ for $\epsilon\in (0,1/2)$, we can find  $\delta_0\in (0,1)$ independent of $\varphi_\infty$ such that \eqref{ps1} holds (cf. \cite{A2007}). The proof of Proposition \ref{sta} is complete.
\end{proof}

Now we are in a position to prove the second part of Theorem \ref{2main}.

\begin{proof}[\textbf{Proof of Theorem \ref{2main}-(2)}]
Define $\omega(\varphi)=\{\varphi_\infty: (\mathbf{0},\varphi_\infty,\sigma_\infty)\in \omega(\bm{v},\varphi, \sigma)\}$. It follows from the relative compactness of the trajectory $\varphi(t)$ in $W^{2-\epsilon,6}(\Omega)$ and Proposition \ref{sta} that
$$
\lim_{t\to+\infty}\mathrm{dist}_{W^{2-\epsilon,6}} \big(\varphi(t),\omega(\varphi)\big)=0,
$$
for $\epsilon\in (0,1/2)$. Since $W^{2-\epsilon, 6}(\Omega) \hookrightarrow C(\overline{\Omega})$, we infer from \eqref{ps1} that, for every $\delta \in\left(0, \delta_0\right)$, there is some $T_{\mathrm{SP}}\gg 1$ such that
\begin{align}
\|\varphi(t)\|_{C(\overline{\Omega})} \leq 1-\delta, \quad \forall\, t \geq T_{\mathrm{SP}}.
\label{sep}
\end{align}
This completes the proof.
\end{proof}

\subsection{Weak-strong uniqueness and eventual regularity of $(\bm{v},\sigma)$}

To establish the eventual regularity of the velocity field $\bm{v}$, as for the classical Navier--Stokes equations in three dimensions, we need a weak-strong uniqueness result for the full system \eqref{nsch}--\eqref{ini0}.

\begin{proposition} \label{ls}
Let $\Omega$ be a bounded domain in $\mathbb{R}^3$, with boundary $\partial \Omega$ of class $C^3$.
Suppose that hypotheses (H1)--(H3) are satisfied. Consider the initial datum $(\boldsymbol{v}_{0}, \varphi_{0},\sigma_{0})$ that satisfies
$\boldsymbol{v}_{0} \in \bm{H}^1_{0,\mathrm{div}}(\Omega)$, $\varphi_{0} \in H^{2}_N(\Omega)$, $\|\varphi_{0}\|_{L^\infty}\leq 1$, $|\overline{\varphi_{0}}|<1$, $\mu_{0}=-\Delta \varphi_{0}+ \varPsi'(\varphi_{0})\in H^1(\Omega)$ and $\sigma_{0} \in H^{1}(\Omega)$. Furthermore, we suppose that:

	\emph{(1)} $(\bm{v}_1,\varphi_1,\mu_1,\sigma_1)$ is a weak solution to problem \eqref{nsch}--\eqref{ini0} on $[0,T]$ such that
	\begin{align*}
	&\bm{v}_1\in L^{\infty}(0,T;\bm{L}^2_{0,\mathrm{div}}(\Omega))
\cap L^{2}(0,T;\bm{H}^1_{0,\mathrm{div}}(\Omega))
\cap  W^{1,\frac{4}{3}}(0,T;(\bm{H}^1_{0,\mathrm{div}}(\Omega))'),
\\
	&\varphi_1 \in L^\infty(0,T ; W^{2,6}(\Omega)\cap H^2_N(\Omega)),
\quad \partial_t \varphi_1 \in L^{\infty}(0,T; (H^{1}(\Omega))')
\cap L^2(0,T ; H^{1}(\Omega)),
\\
	&  \varphi_1 \in C(\overline{\Omega}\times[0,T]):\max _{t \in[0, T]}\|\varphi_1(t)\|_{C(\overline{\Omega})}<1,
\\
	&\mu_1 \in L^\infty(0,T; H^{1}(\Omega)) \cap L^{2}(0,T ; H^{3}(\Omega)\cap H^2_N(\Omega)),
\\
	&\sigma_1\in L^{\infty}(0, T ; L^6(\Omega))\cap L^{2}(0, T ; H^1(\Omega)),\quad \partial_t\sigma_1\in L^2(0, T ; (H^1(\Omega))'),
	\end{align*}
which satisfies
	\begin{subequations}
				\begin{alignat}{3}
				&\left \langle\partial_t  \bm{ v}_1,\bm{\zeta}\right \rangle_{(\bm{H}^1_{0,\mathrm{div}})',\bm{H}^1_{0,\mathrm{div}}}
+(( \bm{ v}_1 \cdot \nabla)  \bm {v}_1,\bm{ \zeta})+(  2\nu(\varphi_1) D\bm{v}_1,D\bm{ \zeta}) =((\mu_1+\chi \sigma_1)\nabla \varphi_1,\bm {\zeta}), &
\notag \\
				&\left \langle\partial_t \sigma_1,\xi\right \rangle_{(H^1)',H^1}+({\bm{v}_1 \cdot \nabla \sigma_1},\xi) + (\nabla \sigma_1,\nabla \xi)= \chi ( \nabla \varphi_1,\nabla \xi),& 
                \notag
				\end{alignat}
	\end{subequations}
 for any $\bm {\zeta} \in \bm{H}^1_{0,\mathrm{div}}(\Omega)$, $\xi \in H^1(\Omega)$ and almost all $t\in (0,T)$,
 	\begin{subequations}
				\begin{alignat}{3}
					& \partial_t \varphi_1+\bm{v}_1 \cdot \nabla \varphi_1=\Delta \mu_1-\alpha(\overline{\varphi_1}-c_0),&
                    \notag \\
				& \mu_1=-\Delta \varphi_1+\varPsi'(\varphi_1)-\chi \sigma_1+\beta\mathcal{N}(\varphi_1-\overline{\varphi_1}),&
                \notag
				\end{alignat}
	\end{subequations}
almost everywhere in $\Omega \times (0,T)$. Moreover, the initial conditions are fulfilled
$$
\bm {v}_1|_{t=0}=\bm{v}_{0},\quad \varphi_1|_{t=0}=\varphi_{0},\quad  \sigma_1|_{t=0}=\sigma_{0}.
$$
	
\emph{(2)} $(\bm{v}_2,p_2,\varphi_2,\mu_2,\sigma_2)$ is a more regular solution to problem \eqref{nsch}--\eqref{ini0} on $[0,T]$ such that
\begin{align*}
&\bm{v}_2\in C([0,T] ;\bm{H}^1_{0,\mathrm{div}}(\Omega) ) \cap L^{2}(0,T ; \bm{H}^2(\Omega))
\cap H^{1}(0,T ; \bm{L}^2_{0,\mathrm{div}}(\Omega)),
\\
& p_2\in L^2(0,T;H^1(\Omega)),
\\
&\varphi_2 \in L^\infty(0,T ; W^{2,6}(\Omega)\cap H^2_N(\Omega)),
\quad \partial_t \varphi_2 \in L^{\infty}(0,T; (H^{1}(\Omega))')
\cap L^2(0,T ; H^{1}(\Omega)),
\\
&   \varphi_2 \in C(\overline{\Omega}\times[0,T]):\max _{t \in[0, T]}\|\varphi_2(t)\|_{C(\overline{\Omega})}<1,
\\
&\mu_2 \in L^\infty(0,T; H^{1}(\Omega)) \cap L^{2}(0,T ; H^{3}(\Omega)\cap H^2_N(\Omega)),
\\
&\sigma_2\in L^\infty(0,T ; L^6(\Omega)) \cap L^{2}(0,T ; H^{1}(\Omega)),
\quad \partial_t\sigma_2\in L^2(0, T ; (H^1(\Omega))').
\end{align*}
The solution $(\bm{v}_2,\varphi_2,\mu_2,\sigma_2)$  fulfills the equations  \eqref{f3.c}--\eqref{f4.d} almost everywhere in $\Omega\times (0,T)$, and
\begin{align}
&\left \langle\partial_t \sigma_2,\xi\right \rangle_{(H^1)',H^1}+({\bm{v}_2 \cdot \nabla \sigma_2},\xi) + (\nabla \sigma_2,\nabla \xi)= \chi ( \nabla \varphi_2,\nabla \xi),&
\notag
\end{align}
for any $\xi \in H^1(\Omega)$ and almost all $t\in (0,T)$. Moreover, the initial conditions are fulfilled
$$
\bm {v}_2|_{t=0}=\bm{v}_{0},\quad \varphi_2|_{t=0}=\varphi_{0},\quad  \sigma_2|_{t=0}=\sigma_{0}.
$$
Then we have
$$
(\bm{v}_{1},\varphi_{1},\mu_{1},\sigma_{1})=(\bm{v}_{2},\varphi_{2},\mu_{2},\sigma_{2})\quad
\text{on}\ \ [0,T].
$$
\end{proposition}
\begin{proof}
Denote the differences of two solutions by
$$
(\bm{v},\varphi,\mu,\sigma) =(\bm{v}_{1}-\bm{v}_{2},\varphi_{1}-\varphi_{2}, \mu_1-\mu_2,\sigma_{1}-\sigma_{2}).
$$
By assumption, we have
$$\bm{v}(0)=\mathbf{0},\quad
\varphi(0)=\sigma(0)=0, \quad
\overline{\varphi(t)}=\overline{\sigma(t)}=0,
\quad \forall\, t\in [0,T].
$$
Moreover, there exists $\delta_*\in (0,1)$ such that
$$
\|\varphi_1(t)\|_{C(\overline{\Omega})}\leq 1-\delta_*,\quad \|\varphi_2(t)\|_{C(\overline{\Omega})}\leq 1-\delta_*,\quad \forall\, t\in[0,T].
$$
The strict separation property of $\varphi_1$, $\varphi_2$ will play a crucial role in the subsequent proof.

The solution $(\bm{v}_1,\varphi_1,\mu_1,\sigma_1)$ satisfies the energy inequality \eqref{wBEL} on $[0,T]$. Since $(\bm{v}_2,\varphi_2,\mu_2,\sigma_2)$ is sufficiently regular, it satisfies the corresponding energy equality. In addition, we see that $\bm{v}_1$ can be used as a test function for the equation of $\bm{v}_2$ and vice versa. Adding the resultants together and integrating on $[0,t]\subset [0,T]$, we obtain
\begin{align}
&(\bm{v}_1(t),\bm{v}_2(t)) - \|\bm{v}_0\|^2
\notag \\
&\quad =-\int_0^t\big((\bm{v}_1 \cdot \nabla) \bm{v}_1,\bm{v}_2\big)\,\mathrm{d} s
-\int_0^t\big((\bm{v}_2 \cdot \nabla) \bm{v}_2,\bm{v}_1\big)\,\mathrm{d} s
\notag \\
&\qquad -\int_0^t\!\int_\Omega 2\nu(\varphi_1)D\bm{v}_1: \nabla \bm{v}_2\,\mathrm{d}x\mathrm{d}s
  -\int_0^t\!\int_\Omega 2\nu(\varphi_2)D\bm{v}_2: \nabla \bm{v}_1\,\mathrm{d}x\mathrm{d}s
\notag \\
&\qquad + \int_0^t\!\int_\Omega (\mu_1+\chi\sigma_1)\nabla \varphi_1 \cdot \bm{v}_2 \,\mathrm{d}x\mathrm{d}s
  + \int_0^t\!\int_\Omega (\mu_2+\chi\sigma_2)\nabla \varphi_2 \cdot \bm{v}_1 \,\mathrm{d}x\mathrm{d}s.
\label{diff-u1}
\end{align}
Subtracting
\eqref{diff-u1} from the sum of the energy inequality \eqref{wBEL} for $(\bm{v}_1,\varphi_1,\mu_1,\sigma_1)$ with the corresponding energy equality for $(\bm{v}_2,\varphi_2,\mu_2,\sigma_2)$, using the incompressibility condition for $\bm{v}_1$, $\bm{v}_2$, and integration by parts, we find
\begin{align}
&\frac12\|\bm{v}(t)\|^2 + \mathcal{F}(\varphi_1(t),\sigma_1(t))+ \mathcal{F}(\varphi_2(t),\sigma_2(t))
+ \int_0^t\!\int_\Omega 2\nu(\varphi_1) |D\bm{v}|^2\,\mathrm{d}x\mathrm{d}s
\notag\\
&\qquad + \int_0^{t} \!\int_\Omega \left(|\nabla \mu_1|^2+ |\nabla \mu_2|^2\right)\,\mathrm{d}x\mathrm{d}s
+\int_0^t\!\int_\Omega \left(|\nabla (\sigma_1-\chi\varphi_1)|^2+ |\nabla (\sigma_2-\chi\varphi_2)|^2\right)\,\mathrm{d}x\mathrm{d}s
 \notag\\
&\quad \leq  2\mathcal{F}(\varphi_0,\sigma_0)
-\int_0^t\!\int_\Omega (\bm{v}\cdot\nabla ) \bm{v}_2 \cdot \bm{v}\,\mathrm{d}x\mathrm{d}s
 -\int_0^t\!\int_\Omega (2(\nu(\varphi_1)-\nu(\varphi_2))D \bm{v}_2): D \bm{v}\,\mathrm{d}x \mathrm{d}s
 \notag\\
&\qquad - \int_0^t\!\int_\Omega (\mu_1+\chi\sigma_1)\nabla \varphi_1 \cdot \bm{v}_2 \,\mathrm{d}x\mathrm{d}s
- \int_0^t\!\int_\Omega (\mu_2+\chi\sigma_2)\nabla \varphi_2 \cdot \bm{v}_1 \,\mathrm{d}x\mathrm{d}s
\notag \\
&\qquad -\alpha\int_{0}^{t}\!\int_\Omega (\overline{\varphi}_1-c_0)\mu_1\, \mathrm{d}x\mathrm{d}s
-\alpha\int_{0}^{t}\!\int_\Omega (\overline{\varphi}_2-c_0)\mu_2\, \mathrm{d}x\mathrm{d}s.
\label{diff-u2}
\end{align}
On the other hand, due to Theorem \ref{vch}, both solutions $(\varphi_1, \mu_1,\sigma_1)$, $(\varphi_2, \mu_2,\sigma_2)$ satisfy the energy equality \eqref{a-wBEL} on $[0,T]$ with $\bm{v}=\bm{v}_1, \bm{v}_2$, respectively. Subtracting them from the inequality \eqref{diff-u2} yields
\begin{align}
&\frac12\|\bm{v}(t)\|^2 + \int_0^t\int_\Omega 2\nu(\varphi_1) |D\bm{v}|^2\,\mathrm{d}x\mathrm{d}s
\notag\\
&\quad \leq
-\int_0^t\!\int_\Omega (\bm{v}\cdot\nabla ) \bm{v}_2 \cdot \bm{v}\,\mathrm{d}x\mathrm{d}s
 -\int_0^t\!\int_\Omega \big(2(\nu(\varphi_1)-\nu(\varphi_2))D \bm{v}_2\big): D \bm{v}\,\mathrm{d}x \mathrm{d}s
 \notag\\
&\qquad + \int_0^t\!\int_\Omega (\mu+\chi\sigma)\nabla \varphi_1 \cdot \bm{v}  \,\mathrm{d}x\mathrm{d}s
+ \int_0^t\!\int_\Omega (\mu_2+\chi\sigma_2)\nabla \varphi  \cdot \bm{v} \,\mathrm{d}x\mathrm{d}s
=:\sum_{j=1}^4J_i.
\label{diff-u3}
\end{align}
Testing the equation for the difference $\varphi$ by $-\Delta \varphi$, integrating over $\Omega\times [0,t]$ gives
\begin{align}
&\frac12\|\nabla \varphi(t)\|^2 + \int_0^t \|\nabla \Delta \varphi\|^2\,\mathrm{d}s
\notag \\
& \quad =
 \int_0^t\!\int_\Omega (\bm{v}\cdot\nabla\varphi_1+ \bm{v}_2\cdot\nabla \varphi)\Delta \varphi \,\mathrm{d}x\mathrm{d}s
+ \int_0^t\!\int_\Omega \nabla (\varPsi'(\varphi_1)-\varPsi'(\varphi_2))\cdot \nabla \Delta \varphi\,\mathrm{d}x\mathrm{d}s \notag\\
&\qquad -\chi \int_0^t\!\int_\Omega \nabla \sigma\cdot \nabla \Delta \varphi\,\mathrm{d}x\mathrm{d}s
-\beta \int_0^t \|\nabla \varphi\|^2\,\mathrm{d}s
=:\sum_{j=5}^8J_i.
\label{diff-p1}
\end{align}
Next, we are allowed to test the weak formulation for the difference $\sigma$ by $ \sigma$. Integrating the resultant over $[0,t]$ yields
\begin{align}
&\frac12\|\sigma(t)\|^2 + \int_0^t\|\nabla \sigma\|^2\,\mathrm{d}s
\notag \\
&\quad  =
- \int_0^t\!\int_\Omega (\bm{v}\cdot\nabla\sigma_2)\sigma \,\mathrm{d}x\mathrm{d}s
+\chi \int_0^t\!\int_\Omega \nabla \varphi \cdot \nabla\sigma \,\mathrm{d}x\mathrm{d}s
=:\sum_{i=9}^{10}J_i,
\label{diff-s1}
\end{align}
Adding \eqref{diff-u3}, \eqref{diff-p1} and \eqref{diff-s1} times by the factor $(1+\chi^2)$ together, using (H1) and Korn's inequality, we obtain
\begin{align}
&\frac12\|\bm{v}(t)\|^2 +\frac12\|\nabla \varphi(t)\|^2
+ \frac{1}{2}(1+\chi^2)\|\sigma(t)\|^2
+ \nu_*\int_0^t\|\nabla \bm{v}\|^2\,\mathrm{d}s
\notag\\
&\quad + \int_0^t \|\nabla \Delta \varphi\|^2\,\mathrm{d}s
+ (1+\chi^2)\int_0^t\|\nabla \sigma\|^2\,\mathrm{d}s
 \leq  \sum_{i=1}^8J_i +(1+\chi^2) \sum_{j=9}^{10}J_i.
 \label{diff-ups}
\end{align}
Let us estimate the right-hand side of \eqref{diff-ups} term by term.
\begin{align*}
J_1&\leq \|\nabla \bm{v}_2\|_{\bm{L}^3}\|\bm{v}\|_{\bm{L}^3}^2
\leq \frac{\nu_*}{8}\int_0^t\|\nabla \bm{v}\|^2\,\mathrm{d}s
+C \int_0^t \|\nabla \bm{v}_2\|_{\bm{L}^3}^2\|\bm{v}\|^2\,\mathrm{d}s,
\end{align*}
\begin{align}
J_2
&\leq C\int_0^t\left\|\int_0^1 \nu'(r\varphi_1+(1-r)\varphi_2)\varphi\,\mathrm{d}r\right\|_{L^6}
\|D\bm{v}_2\|_{\bm{L}^3}\|\nabla  \bm{v}\|\,\mathrm{d}s
\notag\\
&\leq C\int_0^t\|\nabla \varphi\| \|D\bm{v}_2\|_{\bm{L}^3} \|\nabla  \bm{v}\|\,\mathrm{d}s
 \notag\\
&\leq \frac{\nu_*}{8} \int_0^t \|\nabla \bm{v}\|^2\,\mathrm{d}s
+C \int_0^t\|D\bm{v}_2\|_{\bm{L}^3}^2\|\nabla \varphi\|^2\,\mathrm{d}s.
\notag
\end{align}
Using the expression of $\mu$, the incompressibility condition and integration by parts, we estimate $J_3$ and $J_5$ as follows
\begin{align*}
J_3 +J_5
& = \int_0^t\!\int_\Omega ( \varPsi'(\varphi_1)-\varPsi'(\varphi_2) + \beta \mathcal{N}\varphi )\nabla \varphi_1 \cdot \bm{v}  \,\mathrm{d}x\mathrm{d}s - \int_0^t\!\int_\Omega  \varphi\bm{v}_2 \cdot \nabla \Delta \varphi \,\mathrm{d}x\mathrm{d}s
\notag \\
&\leq \int_0^t\left\|\int_0^1 \varPsi''(r\varphi_1+(1-r)\varphi_2)\varphi\,\mathrm{d}r\right\|_{L^6}
\|\nabla \varphi_1 \|_{\bm{L}^3}\|\bm{v}\|\,\mathrm{d}s
\notag\\
&\quad + C\int_0^t \|\mathcal{N}\varphi\|_{L^6}
\|\nabla \varphi_1 \|_{\bm{L}^3}\|\bm{v}\|\,\mathrm{d}s
+ \int_0^t\|\bm{v}_2\|_{\bm{L}^3}\|\varphi\|_{L^6} \|\nabla \Delta \varphi \| \,\mathrm{d}s
\notag \\
&\leq \frac{1}{8} \int_0^t \|\nabla \Delta \varphi \|^2\, \mathrm{d}s
+ C\int_0^t (1+\|\bm{v}_2\|_{\bm{L}^3}^2)\|\nabla \varphi\|^2\,\mathrm{d}s + \int_0^t \|\nabla \varphi_1 \|_{\bm{L}^3}^2\|\bm{v}\|^2\,\mathrm{d}s.
\end{align*}
By H\"{o}lder's and Young's inequalities, we find
\begin{align*}
J_4&\leq C \int_0^t  (\|\mu_2\|_{L^3}+\|\sigma_2\|_{L^3})\|\nabla \varphi\|\|\bm{v}\|_{L^6} \,\mathrm{d}s
\notag\\
&\leq \frac{\nu_*}{8} \int_0^t \|\nabla \bm{v}\|^2\,\mathrm{d}s
+C \int_0^t  (\|\mu_2\|_{L^3}^2+\|\sigma_2\|_{L^3}^2)\|\nabla \varphi\|^2\, \mathrm{d}s,
\end{align*}
\begin{align*}
J_6 &\leq \int_0^t\left\|\nabla \left(\int_0^1 \varPsi''(r\varphi_1+(1-r)\varphi_2)\varphi\,\mathrm{d}r\right)\right\|
\|\nabla \Delta \varphi \|\,\mathrm{d}s
\notag \\
&\leq C \int_0^t \max_{r\in[-1+\delta_*,1-\delta_*]}|\varPsi''(r)|\|\nabla \varphi\|
 \|\nabla \Delta \varphi \|\,\mathrm{d}s
\notag\\
&\quad + C \int_0^t  \max_{r\in[-1+\delta_*,1-\delta_*]}|\varPsi'''(r)|
(\|\nabla \varphi_1\|_{\bm{L}^3}+\|\nabla \varphi_2\|_{\bm{L}^3})\|\varphi\|_{L^6}\|\nabla \Delta \varphi \|\,\mathrm{d}s
\notag \\
&\leq \frac{1}{8} \int_0^t \|\nabla \Delta \varphi \|^2\, \mathrm{d}s
+ \int_0^t (\|\nabla \varphi_1 \|_{\bm{L}^3}^2
+\|\nabla \varphi_2 \|_{\bm{L}^3}^2)\|\nabla \varphi\|^2\,\mathrm{d}s,
\end{align*}
\begin{align*}
J_7&\leq \frac{\chi^2}{2} \int_0^t  \|\nabla \sigma\|^2\,\mathrm{d}s
+ \frac12 \int_0^t \|\nabla \Delta \varphi\|^2\,\mathrm{d}s,
\end{align*}
\begin{align*}
(1+\chi^2)J_{10}&\leq \frac{(1+\chi^2)}{4} \int_0^t  \|\nabla \sigma\|^2\,\mathrm{d}s + C(1+\chi^2)\chi^2 \int_0^t  \|\nabla \varphi\|^2\,\mathrm{d}s.
\end{align*}
Finally, using the incompressibility condition and integration by parts, we get
\begin{align*}
(1+\chi^2)J_9& = (1+\chi^2)  \int_0^t\!\int_\Omega (\bm{v}\cdot\nabla\sigma)\sigma_2 \,\mathrm{d}x\mathrm{d}s
\notag \\
& \leq (1+\chi^2) \int_0^t \|\bm{v}\|_{\bm{L}^3}\|\nabla \sigma\|\|\sigma_2\|_{L^6}\,\mathrm{d}s\\
&\leq \frac{\nu_*}{8} \int_0^t \|\nabla \bm{v}\|^2\,\mathrm{d}s
+ \frac{1+\chi^2}{4} \int_0^t  \|\nabla \sigma\|^2\,\mathrm{d}s
+ C \int_0^t \|\sigma_2\|_{L^6}^4 \|\bm{v}\|^2\,\mathrm{d}s.
\end{align*}
Collecting the above estimates and using the Sobolev embedding theorem, we can deduce from \eqref{diff-ups} that
\begin{align}
&\mathcal{Y}(t)
+ \nu_* \int_0^t\|\nabla \bm{v}\|^2\,\mathrm{d}s
+ \frac{1}{2}\int_0^t \|\nabla \Delta \varphi\|^2\,\mathrm{d}s
+ \int_0^t\|\nabla \sigma\|^2\,\mathrm{d}s
 \leq  C \int_0^t\mathcal{Z}(s)\mathcal{Y}(s)\,\mathrm{d}s,
 \label{diff-ups1}
\end{align}
where
\begin{align*}
& \mathcal{Y}(t)=  \|\bm{v}(t)\|^2 +  \|\nabla \varphi(t)\|^2
+  (1+\chi^2) \|\sigma(t)\|^2,\notag\\
& \mathcal{Z}(t)=1+  \|\bm{v}_2(t)\|_{\bm{W}^{1,3}}^2 + \|\varphi_1(t) \|_{W^{1,3}}^2+ \|\varphi_2(t) \|_{W^{1,3}}^2 + \|\mu_2(t)\|_{L^3}^2+\|\sigma_2(t)\|_{L^3}^2+ \|\sigma_2(t)\|_{L^6}^4.
\end{align*}
Since $\mathcal{Z}\in L^1(0,T)$, we infer from \eqref{diff-ups1} and Gronwall's lemma that $$\mathcal{Y}(t)=0,\quad \forall\,t\in [0,T].$$
This combined with the Poincar\'{e}--Wirtinger inequality leads to the desired conclusion.

The proof of Proposition \ref{ls} is complete.
\end{proof}

We are ready to prove the third part of Theorem \ref{2main}. The proof is based on the following result for the Navier–Stokes system with variable viscosity (see \cite[Theorem 8]{A2009}):
\begin{lemma}
\label{NSglo}
Let $\Omega$ be a bounded domain in $\mathbb{R}^3$, with boundary $\partial \Omega$ of class $C^2$. Consider
\begin{align}
&\partial_t  \bm{ u}+\bm{ u} \cdot \nabla  \bm {u}-\mathrm{div} \big(2  \nu(c) D\bm{u} \big)+\nabla p=\boldsymbol{f}, &&\text { in } \Omega \times(0, T),
\label{ns1}
\\
&\mathrm{div}\,\bm{u}=0,&&\text { in } \Omega \times(0, T),
\\
&\bm{u}=\mathbf{0},&&\text { on } \partial\Omega \times(0, T),
\\
&\bm{u}|_{t=0}=\bm{u}_0,&&\text { in } \Omega,
\label{ns2}
\end{align}
for given data $c$, $\bm{u}_0$, $\bm{f}$.
Suppose that the assumption (H1) is satisfied,
$c\in BUC([0,+\infty);W^{1,6}(\Omega))$, $\bm{u}_0\in \bm{H}^1_{0,\mathrm{div}}(\Omega)$ and $\bm{f}\in L^2(0,+\infty;\bm{L}^2_{0,\mathrm{div}}(\Omega))$. There is some $\varepsilon_0>0$ such that, if
$$
\|\bm{u}_0\|_{\bm{H}^1_{0,\mathrm{div}}} +\|\bm{f}\|_{L^2(0,+\infty;\bm{L}^2_{0,\mathrm{div}}(\Omega))}\leq \varepsilon_0,
$$
then there is a unique solution $\bm{u}\in L^2(0,T;D(\bm{S}))\cap H^1(0,T;\bm{L}^2_{0,\mathrm{div}}(\Omega))$ of problem \eqref{ns1}--\eqref{ns2} with $T=+\infty$.
\end{lemma}
For its application, we refer to \cite{A2009} for the eventual regularity of global weak solutions to the Navier--Stokes--Cahn--Hilliard system with constant density and non-constant viscosity (an alternative approach can be found in \cite{ZWH}). Recent progress on a general Navier--Stokes--Cahn--Hilliard system with non-constant density and viscosity has been made in \cite{A2022}. Here, we confine ourselves to the case with constant density and non-constant viscosity, but take into account extra effects due to chemotaxis, active transport and nonlocal interactions.
\begin{proof}[\textbf{Proof of Theorem \ref{2main}-(3)}]
Let $(\boldsymbol{v},\varphi,\mu,\sigma)$ be a global weak solution to problem \eqref{nsch}--\eqref{ini0} given by Proposition \ref{main}. We have already shown that it
satisfies \eqref{vp-Linf}--\eqref{uni-es2} as well as the strict separation property \eqref{sep}.
Now for the small positive constant $\varepsilon_0$ determined by Lemma \ref{NSglo}, thanks to \eqref{mass-ps} and \eqref{wBEL-2}, there exists some time $T_{\mathrm{R}}=T_{\mathrm{R}}(\varepsilon_0)$ (can be taken larger than $T_{\mathrm{SP}}$, without loss of generality) such that
\begin{align*}
& \|\boldsymbol{v}\|_{L^2(T_{\mathrm{R}},+\infty; \bm{H}^1_{0,\mathrm{div}}(\Omega) )} \leq \frac{\varepsilon_0}{4(1+|\chi|)},
 &&\|\boldsymbol{v}(T_{\mathrm{R}})\|_{\bm{H}^1_{0,\mathrm{div}}} \leq \frac{\varepsilon_0}{2},
\\
& \|\nabla \mu\|_{L^2(T_{\mathrm{R}},+\infty;\bm{L}^2(\Omega))} \leq \frac{\varepsilon_0}{4(1+|\chi|)},
&&\|\nabla \mu(T_{\mathrm{R}})\| \leq \frac{\varepsilon_0}{2},
\\
& \|\nabla (\sigma-\chi\varphi)\|_{L^2(T_{\mathrm{R}},+\infty;\bm{L}^2(\Omega))} \leq \frac{\varepsilon_0}{4(1+|\chi|)},
&&\|\nabla (\sigma(T_{\mathrm{R}})-\chi\varphi(T_{\mathrm{R}}))\| \leq \frac{\varepsilon_0}{2}.
\end{align*}
From the above properties, we further observe that
$$
\varphi(T_{\mathrm{R}}) \in H_N^2(\Omega),\quad \|\varphi(T_{\mathrm{R}})\|_{C(\overline{\Omega})}<1,\quad
\sigma(T_{\mathrm{R}})\in H^1(\Omega),\quad
-\Delta \varphi(T_{\mathrm{R}})+\varPsi'(\varphi(T_{\mathrm{R}})) \in H^1(\Omega).
$$
Now taking $(\bm{v}(T_{\mathrm{R}}), \varphi(T_{\mathrm{R}}), \sigma(T_{\mathrm{R}}))$ as the initial data, according to \cite[Theorem 3.1]{HW2022},  problem \eqref{nsch}--\eqref{ini0} admits a unique local strong solution $(\widehat{\bm{v}},\widehat{p},\widehat{\varphi},\widehat{\mu},\widehat{\sigma})$ defined on the maximal interval $\left[T_{\mathrm{R}}, T_{\max }\right)$ such that for any $\widetilde{T}\in[T_{\mathrm{R}},T_{\max })$, it fulfills
\begin{align*}
&\widehat{\boldsymbol{v}} \in C([T_{\mathrm{R}}, \widetilde{T}] ;\bm{H}^1_{0,\mathrm{div}}(\Omega) ) \cap L^{2}(T_{\mathrm{R}}, \widetilde{T}; D(\bm{S})) \cap H^{1}(T_{\mathrm{R}}, \widetilde{T} ; \bm{L}^2_{0,\mathrm{div}}(\Omega)),
\\
&\widehat{\varphi} \in L^\infty(T_{\mathrm{R}}, \widetilde{T}; W^{2,6}(\Omega)\cap H^2_N(\Omega)),\quad \partial_t \varphi \in L^{\infty}(T_{\mathrm{R}}, \widetilde{T} ; (H^{1}(\Omega))') \cap L^2(T_{\mathrm{R}}, \widetilde{T} ; H^{1}(\Omega)),
\\
& \widehat{\varphi}\in C(\overline{\Omega}\times[T_{\mathrm{R}}, \widetilde{T}])\quad \text{with}\quad  \max_{t\in[T_{\mathrm{R}},\widetilde{T}]} \|\widehat{\varphi}(t)\|_{C(\overline{\Omega})}<1,
\\
&\widehat{\mu} \in L^\infty(T_{\mathrm{R}}, \widetilde{T}; H^{1}(\Omega)) \cap L^{2}(T_{\mathrm{R}}, \widetilde{T} ; H^{3}(\Omega)\cap H^2_N(\Omega)),
\\
&\widehat{\sigma }\in C([T_{\mathrm{R}}, \widetilde{T}] ; H^{1}(\Omega)) \cap L^{2}(T_{\mathrm{R}}, \widetilde{T}; H^2_N(\Omega)) \cap H^{1}(T_{\mathrm{R}}, \widetilde{T}; L^{2}(\Omega)),
\\
&\widehat{p} \in L^{2}(T_{\mathrm{R}}, \widetilde{T} ; V_0).
\end{align*}
The solution $(\widehat{\bm{v}},\widehat{p},\widehat{\varphi},\widehat{\mu},\widehat{\sigma})$ satisfies the system \eqref{nsch} almost everywhere in $\Omega\times (T_{\mathrm{R}}, T_{\max })$, the boundary condition \eqref{boundary}
almost everywhere on $\partial\Omega\times (T_{\mathrm{R}}, T_{\max })$, and the initial conditions $\widehat{\bm{v}}(T_{\mathrm{R}})=\bm{v}(T_{\mathrm{R}})$, $\widehat{\varphi}(T_{\mathrm{R}})=\varphi(T_{\mathrm{R}})$, $\widehat{\sigma}(T_{\mathrm{R}})=\sigma(T_{\mathrm{R}})$ almost everywhere in $\Omega$.
Moreover, it satisfies the energy equality:
\begin{align}
& E(\widehat{\bm{v}}(t_2),\widehat{\varphi}(t_2), \widehat{\sigma}(t_2))
+  \int_{t_1}^{t_2}\!\int_{\Omega} \left( 2\nu(\widehat{\varphi})|D\widehat{\bm{v}}|^2 + |\nabla \widehat{\mu}|^2+|\nabla(\widehat{\sigma}-\chi\widehat{\varphi})|^2\right) \, \mathrm{d}x\mathrm{d}s
\notag\\
&\quad = E(\widehat{\bm{v}}(t_1),\widehat{\varphi}(t_1), \widehat{\sigma}(t_1)) -\alpha\int_{t_1}^{t_2}\!\int_\Omega (\overline{\widehat{\varphi}}-c_0)\widehat{\mu}\, \mathrm{d}x\mathrm{d}s,
\notag
\end{align}
for all $T_{\mathrm{R}}\leq t_1\leq t_2\leq \widetilde{T}$.
Thus, we can apply Theorem \ref{2main}-(1), (2) and Proposition \ref{ls}
to conclude that
\begin{align}
(\bm{v},\varphi,\mu,\sigma)=(\widehat{\bm{v}},\widehat{\varphi}, \widehat{\mu},\widehat{\sigma}),\quad \text{in }[T_{\mathrm{R}}, T_{\max }),
\label{ws-uni}
\end{align}
which implies, for any $\widetilde{T}\in[T_{\mathrm{R}},T_{\max })$, it holds
\begin{align}
& \boldsymbol{v}  \in C([T_{\mathrm{R}}, \widetilde{T}] ;\bm{H}^1_{0,\mathrm{div}}(\Omega) )
\cap L^{2}(T_{\mathrm{R}}, \widetilde{T}; D(\bm{S}))
\cap H^{1}(T_{\mathrm{R}}, \widetilde{T}; \bm{L}^2_{0,\mathrm{div}}(\Omega)),
\notag
\\
& \sigma  \in C([T_{\mathrm{R}}, \widetilde{T}] ; H^{1}(\Omega))
\cap L^{2}(T_{\mathrm{R}}, \widetilde{T}; H^{2}_N(\Omega))
\cap H^{1}(T_{\mathrm{R}}, \widetilde{T}; L^{2}(\Omega)).
\notag
\end{align}

As a consequence, to complete the proof of Theorem \ref{2main}-(3), it remains to show that $T_{\max}=+\infty$.
To this end, we notice that \eqref{wBEL-2}, \eqref{uni-es2} and the separation property \eqref{sep} hold globally in time, and they are independent of $T_{\max}$. Thanks to \eqref{ws-uni}, below we will not distinguish
$(\bm{v},\varphi,\mu,\sigma)$, $(\widehat{\bm{v}},\widehat{\varphi}, \widehat{\mu},\widehat{\sigma})$ since they coincide as long as the latter exists.

Rewrite now the equation \eqref{f3.c} for $\bm{v}$ as follows
\be
\partial_t  \bm{ v}+\bm{ v} \cdot \nabla  \bm {v}-\mathrm{div} \big(2  \nu(\varphi) D\bm{v} \big)+\nabla p^*=\boldsymbol{g}, \quad \text { in } \Omega \times\left(T_{\mathrm{R}}, T_{\max }\right),
\label{veq}
\ee
where
$$
p^*=p-(\mu+\chi\sigma) \varphi +\frac{\chi^2}{2}\varphi^2-z,\quad  \bm{g}=\bm{P}{\bm{f}},
$$
and
$$
\bm{f}=-\varphi \nabla \mu-\chi\varphi \nabla (\sigma-\chi\varphi),\quad
\nabla z=\bm{f}-\bm{g}.
$$
Then it follows from \eqref{vp-Linf}, \eqref{wBEL-2} and the property of the Leray projection $\bm{P}$ that
\begin{align*}
\|\bm{g}\|_{L^2(T_{\mathrm{R}},+\infty;\bm{L}^2_{0,\mathrm{div}}(\Omega))}
&\leq
\|\varphi\|_{L^\infty(T_{\mathrm{R}},+\infty;L^\infty(\Omega))}
 \|\nabla \mu\|_{L^2(T_{\mathrm{R}},+\infty;\bm{L}^2(\Omega))}
 \\
&\quad  + |\chi|\|\varphi\|_{L^\infty(T_{\mathrm{R}},+\infty;L^\infty(\Omega))}\| \nabla (\sigma-\chi\varphi)\|_{L^2(T_{\mathrm{R}},+\infty;\bm{L}^2(\Omega))}
\\
&\leq \frac{\varepsilon_0}{2}.
\end{align*}
Since
$$
\varphi\in BUC([T_{\mathrm{R}},+\infty);W^{1,6}(\Omega)),\quad
\|\bm{v}(T_{\mathrm{R}})\|_{\bm{H}^1_{0,\mathrm{div}}(\Omega))}+\|\bm{g}\|_{L^2(T_{\mathrm{R}},+\infty; \bm{L}^2_{0,\mathrm{div}}(\Omega))}\leq \varepsilon_0,
$$
then we can apply Lemma \ref{NSglo} to conclude that equation \eqref{veq} subject to the no-slip boundary condition and the initial condition $\bm{v}(T_{\mathrm{R}})$ admits a unique global strong solution such that
$$
\bm{v}\in L^2(T_{\mathrm{R}},+\infty;D(\bm{S}))
\cap H^1(T_{\mathrm{R}},+\infty;\bm{L}^2_{0,\mathrm{div}}(\Omega)).
$$
By interpolation, we also find $\bm{v}\in BUC([T_{\mathrm{R}},+\infty);\bm{H}^1_{0,\mathrm{div}}(\Omega))$.

Next, testing equation \eqref{f2.b} for $\sigma$ by $-\Delta\sigma$ and integrating over $\Omega$, we get
\begin{align}
\frac{1}{2} \frac{\mathrm{d}}{\mathrm{d} t} \|\nabla\sigma \|^{2}
+\|\Delta \sigma\|^2
&= (\bm{v} \cdot \nabla \sigma,\Delta\sigma)
+ \chi (\Delta \varphi,\Delta\sigma)
\notag\\
&\leq \|\bm{v}\|_{\bm{L}^\infty} \| \nabla \sigma\|\| \Delta \sigma\|+
|\chi|\|\Delta \varphi\|\|\Delta \sigma\|
\notag \\
&\leq \frac12\|\Delta \sigma\|^2 + \|\bm{v}^k\|_{\bm{L}^\infty}^2 \| \nabla \sigma\|^2 + \chi^2\|\Delta \varphi\|^2.
\notag
\end{align}
Since $\sigma(T_{\mathrm{R}})\in H^1(\Omega)$,
$\bm{v}\in L^2(T_{\mathrm{R}},+\infty;D(\bm{S}))$
and $\varphi\in L^4_{\mathrm{uloc}}([T_{\mathrm{R}},+\infty); H^2(\Omega))$ (thanks to \eqref{wBEL-2} and a similar estimate for \eqref{L6-app1}), keeping in mind the continuous embedding  $\bm{H}^2(\Omega)\hookrightarrow \bm{L}^\infty(\Omega)$,
we can apply the classical Gronwall's lemma on $[T_{\mathrm{R}},T_{\mathrm{R}}+1]$ and the uniform Gronwall lemma on $[T_{\mathrm{R}}+1,+\infty)$ to conclude that
\be
\sup_{t\geq T_{\mathrm{R}}}\left(\|\nabla \sigma(t)\|^2 + \int_{t}^{t+1}\|\Delta \sigma(s)\|^2\,\mathrm{d}s\right)\leq C.
\notag
\ee
As a consequence,
$$
\sigma\in L^\infty(T_{\mathrm{R}},+\infty; H^1(\Omega))\cap L^2_{\mathrm{uloc}}([T_{\mathrm{R}},+\infty); H^2_N(\Omega)).
$$
By comparison in equation \eqref{f2.b}, we further  obtain $\partial_t\sigma\in L^2_{\mathrm{uloc}}([T_{\mathrm{R}},+\infty); L^2(\Omega))$.

In conclusion, we have shown that $T_{\max}=+\infty$ and the solution $(\bm{v},\varphi, \mu,\sigma)$ is  uniformly bounded in the corresponding spaces.

Finally, returning to equation \eqref{f4.d}, due to the improved regularity of $\sigma$, we can apply the elliptic estimate for the Neumann problem, \eqref{uni-es2}, \eqref{sep} and (H2) to conclude that $\varphi\in L^\infty (T_{\mathrm{R}},+\infty; H^3(\Omega))$.

The proof of Theorem \ref{2main}-(3) is complete.
\end{proof}

\appendix
\section{Useful Tools}
\setcounter{equation}{0}
In the appendix, we report a Gronwall-type inequality (see, e.g., \cite[Lemma 2.5]{GGP}) that has been used in this study.

\begin{lemma}\label{GronW}
Let $f$, $m_1$, and $m_2$ be three nonnegative locally summable functions on $[\tau,+\infty)$ which satisfy, for some $\varepsilon>0$, the differential inequality
\begin{align}
\frac{\mathrm{d}}{\mathrm{d}t}f^2(t)+ \varepsilon f^2(t)\leq m_1(t) f(t) + m_2(t)\quad \text{for a.a.}\ t\in [\tau,+\infty).
\label{gron-a}
\end{align}
Besides, assume $f\in C([\tau,+\infty))$. Then
\begin{align}
f^2(t)\leq 2 f^2(\tau) e^{-\varepsilon(t-\tau)}+ \left(\int_\tau^t m_1(s)e^{-\frac{\varepsilon}{2}(t-s)}\,\mathrm{d}s\right)^2
+2\int_\tau^t m_2(s)e^{-\varepsilon(t-s)}\,\mathrm{d}s,
\label{gron-b}
\end{align}
for any $t\in[\tau,+\infty)$. Moreover, the inequality
\begin{align}
    \int_\tau^t m(s)e^{-\varepsilon(t-s)}\,\mathrm{d}s
    \leq \frac{e^{\varepsilon}}{1-e^{-\varepsilon}}\sup_{r\geq \tau}\int_{r}^{r+1}m(s)\,\mathrm{d}s, \quad \forall\, t\in [\tau,+\infty),
    \label{gron-c}
\end{align}
holds for every nonnegative locally summable function $m$ on $[\tau,+\infty)$ and every $\varepsilon>0$.
\end{lemma}
\begin{proof}
For the reader's convenience, we sketch a proof adapted from the argument in \cite{PPV}. Applying the classical Gronwall's lemma to \eqref{gron-a}, we get
\begin{align}
f^2(t)\leq f^2(\tau) e^{-\varepsilon(t-\tau)}
+ \int_\tau^t m_1(s)f(s)e^{-\varepsilon(t-s)}\,\mathrm{d}s + \int_\tau^t m_2(s)e^{-\varepsilon(t-s)}\,\mathrm{d}s,\quad  \forall\, t\in [\tau,+\infty).
\notag
\end{align}
Define
$$g(t)=f(t)e^{\frac{\epsilon}{2}t},\quad \forall\,t\geq \tau.$$
Then we have
\begin{align}
g^2(t)\leq g^2(\tau)
+ \int_\tau^t m_1(s)g(s)e^{\frac{\varepsilon}{2}s}\,\mathrm{d}s + \int_\tau^t m_2(s)e^{ \varepsilon s}\,\mathrm{d}s,\quad  \forall\, t\in [\tau,+\infty).
\label{es-g}
\end{align}
The second term on the right-hand side of \eqref{es-g} can be estimated by
\begin{align*}
\int_\tau^{t} m_1(s)g(s)e^{\frac{\varepsilon}{2}s}\,\mathrm{d}s
& \leq
\frac12 \sup_{s\in[\tau,t]}g^2(s)  + \frac12\left(\int_\tau^t m_1(s) e^{\frac{\varepsilon}{2}s}\,\mathrm{d}s\right)^2,\quad  \forall\, t\in [\tau,+\infty).
\end{align*}
For any given $t\in [\tau,+\infty)$, due to the continuity of $g$, there exists some $t_0\in [\tau,t]$ such that $g(t_0)=\max_{s\in[\tau,t]}g(s)$. From \eqref{es-g} and the above observation, we find
\begin{align}
g^2(t_0)
& \leq g^2(\tau)
+ \frac12 g^2(t_0)
+ \frac12\left(\int_\tau^{t_0} m_1(s) e^{\frac{\varepsilon}{2}s}\,\mathrm{d}s\right)^2
+ \int_\tau^{t_0} m_2(s)e^{ \varepsilon s}\,\mathrm{d}s.\notag
\end{align}
This gives
\begin{align}
g^2(t)\leq g^2(t_0) & \leq 2g^2(\tau)
+ \left(\int_\tau^{t_0} m_1(s) e^{\frac{\varepsilon}{2}s}\,\mathrm{d}s\right)^2
+ 2\int_\tau^{t_0} m_2(s)e^{ \varepsilon s}\,\mathrm{d}s \notag \\
&\leq  2g^2(\tau)
+ \left(\int_\tau^{t} m_1(s) e^{\frac{\varepsilon}{2}s}\,\mathrm{d}s\right)^2
+ 2\int_\tau^{t} m_2(s)e^{ \varepsilon s}\,\mathrm{d}s,
\notag
\end{align}
which is exactly \eqref{gron-b}.

Next, we prove \eqref{gron-c}. Define $\widetilde{m}$ as the zero extension of $m$ on $\big[[\tau],+\infty\big)$ such that $\widetilde{m}(s)=m(s)$ for $s\geq \tau$ and $\widetilde{m}(s)=0$ for $s\in \big[[\tau],\tau\big]$. Here, $[\tau]$ denotes the largest integer that is less than or equal to $\tau$. Obviously, it holds
$$
\sup_{r\geq [\tau]}\int_{r}^{r+1}\widetilde{m}(s)\,\mathrm{d}s=\sup_{r\geq \tau}\int_{r}^{r+1}m(s)\,\mathrm{d}s.
$$
For any $t\in \big[[\tau],[\tau]+1\big]$, we easily find
$$
\int_{[\tau]}^t \widetilde{m}(s)e^{-\varepsilon(t-s)}\,\mathrm{d}s\leq \int_{[\tau]}^{[\tau]+1} \widetilde{m}(s)\,\mathrm{d}s.
$$
Next, for any $t\in \big([\tau]+1,+\infty\big)$, there exists an integer $n\geq [\tau]+1$ such that $t\in (n,n+1]$. Then we can deduce that
\begin{align*}
 \int_{[\tau]}^t \widetilde{m}(s)e^{-\varepsilon(t-s)}\,\mathrm{d}s
 & \leq \sum_{i=[\tau]}^{n} \int_{i}^{i+1} \widetilde{m}(s)e^{-\varepsilon(n-s)}\,\mathrm{d}s
 \\
 &\leq \left( e^{-\varepsilon n}\sum_{i=[\tau]}^{n}e^{\varepsilon(i+1)}\right)\sup_{r\geq [\tau]}\int_{r}^{r+1}\widetilde{m}(s)\,\mathrm{d}s \\
 &\leq \frac{e^\varepsilon}{1-e^{-\varepsilon}} \sup_{r\geq [\tau]}\int_{r}^{r+1}\widetilde{m}(s)\,\mathrm{d}s.
\end{align*}
Collecting the above estimates, we arrive at the conclusion \eqref{gron-c}.
\end{proof}

In what follows, we present a lemma for the approximation of $\sigma$ used in the proof of Proposition \ref{con-CHs-ws}.

Let $H$ and $V$ be two Hilbert spaces. Assume that $V$ is a dense linear subspace of $H$, the inclusion of $V$ into $H$ being continuous, and consider $H$
as embedded into the dual space $V'$ of $V$ by means of the usual formula
$$\langle u,v\rangle_{V',V}=(u,v)_H,\quad \forall\, u\in H, \ v\in V. $$
Introduce the identity or injection operator $I : V\to V'$ and the canonical isomorphism $J$ from $V$ onto $V'$ defined by
$$
\langle Ju,v\rangle_{V',V}=(u,v)_V,\quad \forall\, u,v\in V.
$$
For any $\varepsilon>0$ and $u\in V'$, the following problem
\begin{align}
    (I+\varepsilon^2 J)u_\varepsilon=u
    \label{uu-reg}
\end{align}
admits a unique solution $u_\varepsilon\in V$ thanks to the Lax--Milgram theorem. Moreover, we have the following asymptotic behavior for $u_\varepsilon$ as $\varepsilon\to 0$ (see \cite[Proposition 6.1]{Colli97} for a summary of the results, whose proof follows the arguments in \cite{Lions73}):

\begin{lemma}\label{sig-conv}
Let $\varepsilon>0$. For any $u\in V'$, we denote by $u_\varepsilon$ the unique solution to problem \eqref{uu-reg} in $V$.
\begin{itemize}
    \item[\emph{(1)}] For any $u\in V'$, we have
    \begin{align*}
     & \|u_\varepsilon\|_{V'}\leq \|u\|_{V'}\quad \text{and}\quad u_\varepsilon \to u\ \ \text{in}\ V',\\
     & \varepsilon\|u_\varepsilon\|_{H}\leq \|u\|_{V'}\quad \text{and}\quad \varepsilon u_\varepsilon \to 0\ \ \text{in}\ H\ \ \text{as}\ \varepsilon\to 0,\\
     & \varepsilon^2\|u_\varepsilon\|_{V}\leq \|u\|_{V'}\quad \text{and}\quad \varepsilon^2 u_\varepsilon \to 0\ \ \text{in}\ V \ \text{as}\ \varepsilon\to 0.
    \end{align*}
    \item[\emph{(2)}] Moreover, if $u\in H$, then we have
    \begin{align*}
     &  \|u-u_\varepsilon\|_{V'}\leq \varepsilon\|u\|_H,\\
     &  \|u_\varepsilon\|_{H}\leq \|u\|_{H}\quad \text{and}\quad u_\varepsilon \to u\ \ \text{in}\ H \ \text{as}\ \varepsilon\to 0,\\
     & \varepsilon\|u_\varepsilon\|_{V}\leq \|u\|_{H}\quad \text{and}\quad \varepsilon u_\varepsilon \to 0\ \ \text{in}\ V \ \text{as}\ \varepsilon\to 0.
    \end{align*}
    \item[\emph{(3)}] Finally, if $u\in V$, then we have
    \begin{align*}
     &  \|u-u_\varepsilon\|_{V'}\leq \varepsilon^2\|u\|_V,\\
     &  \|u-u_\varepsilon\|_{H}\leq \varepsilon\|u\|_{V},\\
     & \|u_\varepsilon\|_{V}\leq \|u\|_{V}\quad \text{and}\quad u_\varepsilon \to u\ \ \text{in}\ V \ \text{as}\ \varepsilon\to 0.
    \end{align*}
\end{itemize}
\end{lemma}

\section*{Declarations}
\noindent
\textbf{Conflict of interest.} The authors have no competing interests to declare that are relevant to the content of this article.
\\
\noindent
\textbf{Ethical approval.} This research does not involve humans and/or animals.
\\
\noindent
\textbf{Funding.} J.-N. He was supported by Zhejiang Provincial Natural Science Foundation of China (Grant No. LQ24A010011) and the National Natural Science Foundation of China (Grant No. 12401251). H. Wu was supported by National Natural Science Foundation of China (Grant No. 12071084).
\\
\noindent
\textbf{Data availability.} Data sharing not applicable to this article as no datasets were generated or analyzed during the current study.
\\
\noindent
\textbf{Acknowledgments.}
The authors are greatly indebted to the anonymous referee for many valuable comments on a previous version of this paper. Part of this work was done during the visit of J.-N. He to the Key Laboratory of Mathematics for Nonlinear Sciences (Fudan University), Ministry of Education of China, whose hospitality is gratefully acknowledged.
J.-N. He also acknowledges the support by the RFS Grant from the Research Grants Council (Project P0047825, PI: Prof. Xianpeng Hu) for her research stay in Research Center for Nonlinear Analysis, The Hong Kong Polytechnic University. H. Wu is a member of the Key Laboratory of Mathematics for Nonlinear Sciences (Fudan University), Ministry of Education of China.


\end{document}